\documentclass[11pt,reqno]{amsart}

\usepackage[utf8]{inputenc}
\usepackage{fullpage}
\usepackage{mathrsfs}
\usepackage{latexsym,amsfonts,amssymb,amsmath,amsthm}
\usepackage{graphicx}

\usepackage[usenames,dvipsnames]{color}
\usepackage{ulem}
\usepackage{comment}
\usepackage{latexsym}
\usepackage{booktabs} % For pretty tables
\usepackage{pgfplots}
\usepackage{soul}
\usepackage[all]{nowidow}
\usepackage[utf8]{inputenc}
\usepackage{tikz}
\usetikzlibrary{automata}
\usepackage{multicol}
\usepackage{algpseudocode,algorithm,algorithmicx}
\usepackage{hyperref}
\usepackage{graphics,graphicx}
\usepackage{caption}
\usepackage{subcaption}
\usepackage[inline]{enumitem} 
\usepackage[numbers,sort&compress]{natbib}
\definecolor{blue}{HTML}{1F77B4}
\definecolor{orange}{HTML}{FF7F0E}
\definecolor{green}{HTML}{2CA02C}

\pgfplotsset{compat=1.14}

\newtheorem{tm}{Theorem}[section]
\newtheorem{prop}{Proposition}[section]

\newtheorem{coro}{Corollary}[section]
\newtheorem{lem}{Lemma}[section]

\newtheorem{rk}{Remark}[section]

\numberwithin{equation}{section}
\numberwithin{tm}{section}

\def\p{\partial}

\def\Om{\Omega}

\usepackage[foot]{amsaddr}

\makeatletter
\renewcommand{\email}[2][]{%
  \ifx\emails\@empty\relax\else{\g@addto@macro\emails{,\space}}\fi%
  \@ifnotempty{#1}{\g@addto@macro\emails{\textrm{(#1)}\space}}%
  \g@addto@macro\emails{#2}%
}
\makeatother

\title{Long-time dynamics of a parabolic-ODE SIS epidemic model with saturated incidence mechanism}%

\author{Rui Peng$\ ^1$}
\address{$\ ^1$School of Mathematical Sciences, Zhejiang Normal University, Jinhua, 321004, Zhejiang, China.}
\email{pengrui\_seu@163.com}

\author{Rachidi Salako$\ ^2$}
\address{$\ ^2$Department of Mathematical Sciences, University of Nevada Las Vegas, Las Vegas, NV 89154,
USA.}
\email{rachidi.salako@unlv.edu}

\author{Yixiang Wu$\ ^{3, \dagger}$}
\address{$\ ^3$Department of Mathematical Sciences, Middle Tennessee State University, Murfreesboro, TN
37132, USA.}
\email{yixiang.wu@mtsu.edu}

\usepackage{lipsum}

\begin{document}

% \author{Rui Peng}
% \address{Rui Peng, School of Mathematical Sciences, Zhejiang Normal University, Jinhua, 321004, Zhejiang, China.}
% \email{pengrui\_seu@163.com}

% \author{Rachidi Salako}
% \address{Rachidi Salako, Department of Mathematical Sciences, University of Nevada Las Vegas, Las Vegas, NV 89154,
% USA.}
% \email{rachidi.salako@unlv.edu}

% \author{Yixiang Wu}
% \address{Yixiang Wu, Department of Mathematical Sciences, Middle Tennessee State University, Murfreesboro, TN
% 37132, USA.}
% \email{yixiang.wu@mtsu.edu}

% \title
% [A parabolic-ODE SIS epidemic model with saturated incidence mechanism]
% {Long-time dynamics of a parabolic-ODE SIS epidemic model with saturated incidence mechanism}

\thanks{\footnotesize R. Peng was supported by NSF of China (No. 12271486, 12171176).}
\thanks{\footnotesize  $ ^\dagger$ Corresponding author.}

\begin{abstract}
In this paper, we investigate a parabolic-ODE SIS epidemic model with no-flux boundary conditions in a heterogeneous environment. The model incorporates a saturated infection mechanism  \({SI}/(m(x) + S + I)\) with \(m \geq,\,\not\equiv 0\).  This study is motivated by disease control strategies, such as quarantine and lockdown, that 
limit population movement. We examine two scenarios: one where the movement of the susceptible population is 
restricted, and another where the movement of the infected population is 
neglected. We establish the long-term dynamics of the solutions in each scenario. Compared to previous studies that assume the absence of a saturated incidence function (i.e., $m\equiv 0$), our findings highlight the novel and significant interplay between total population size, transmission risk level, and the saturated incidence function in influencing disease persistence, extinction, and spatial distribution. Numerical simulations are performed to validate the theoretical results, and the implications of the results are discussed in the context of disease control and eradication strategies.


\end{abstract}

% REQUIRED
%\begin{keywords}
%\keywords{Parabolic-ODE SIS epidemic model, saturated incidence mechanism, long-term dynamics, disease persistence/extinction}
%\end{keywords}

% REQUIRED
%\begin{AMS}
%\subjclass[2020]{35M13, 35B40, 92D25}
%\end{AMS}

\date{}
%\date{\today}
\maketitle
\vspace{-1cm}
%\tableofcontents
%\newpage
\section{Introduction}
Traditional epidemic models often assume a well-mixed environment: every individual in a certain area has an equal chance of interacting with any other individual. This simplifies the modeling process by adopting ordinary differential equations (ODEs) to describe disease transmissions \cite{Kermack1927}. However, this assumption frequently fails to accurately represent real-world situations, especially in geographically fragmented populations. In such cases, individual interactions are naturally constrained by geographical barriers, even though movement between regions remains possible. This mobility introduces spatial heterogeneity into disease transmission dynamics, leading to region-specific variations in transmission patterns and population characteristics \cite{Turchin2003}. The COVID-19 pandemic has highlighted the crucial importance of understanding environmental heterogeneity to fully comprehend the complex dynamics of disease transmissions \cite{Bellomo,Bertaglia,LloydSmith2005}. This pandemic has also demonstrated the significant influence of population movement on disease spread \cite{Kraemer2020, Chinazzi2020}. Additionally, it has also shown the effectiveness of disease control measures such as lockdown and quarantines in reducing disease transmissions.  

%To address the complexities introduced by spatial heterogeneity, researchers have transitioned from the traditional ODE models to reaction-diffusion equations. 
To address the complexities arising from spatial heterogeneity, researchers have complemented the study of traditional ODE models with investigations of reaction-diffusion equations. This approach provides a more comprehensive framework by accounting for continuous spatial and temporal variations, enabling a more accurate representation of disease spread across heterogeneous environments \cite{Murray2002}. These systems have become increasingly popular in recent studies  \cite{Diekmann2012}.
Building on this conceptual framework, Allen et al. \cite{ABL}  conducted a pioneering work on a reaction-diffusion SIS (susceptible-infected-susceptible) epidemic model with frequency-dependent incidence function $SI/(S+I)$ in spatially heterogeneous environment:
\begin{equation}\label{SIS1}
\begin{cases}
\displaystyle
    \p_tS - d_S \Delta S = -\frac{\beta(x)SI}{S+I} + \gamma(x) I, & x \in \Omega, \ t > 0, \vspace{1mm} \\
\displaystyle
    \p_tI- d_I \Delta I = \frac{\beta(x)SI}{S+I} - \gamma(x) I, & x \in \Omega, \ t > 0, \vspace{1mm} \\
\displaystyle
 \partial_{\nu}S=\p_{\nu}I=0, & x \in \partial\Omega, \ t > 0, \vspace{1mm} \\
S(0,x) = S_0(x) \ge 0, \ I(0,x) = I_0(x) \ge 0, & x \in \Omega.
\end{cases}
\end{equation}
In this model, the domain $\Omega \subset \mathbb{R}^n$ ($n \geq 1$) is bounded with smooth boundary $\partial\Omega$,
$\nu$ is the outward unit normal to $\p\Om$, $\beta,\,\gamma\in C(\bar{\Omega})$, and  $\beta,\,\gamma>0$ on $\bar{\Omega}$.  
The variables $S(t,x)$ and $I(t,x)$ denote the densities of susceptible and infected individuals at  time $t$ and location $x$, respectively. The positive constants $d_S$ and $d_I$ represent the mobility rates of susceptible and infected individuals. The functions $\beta(x)$ and $\gamma(x)$ are the rates of disease transmission and recovery. The homogeneous Neumann boundary conditions ensure no population flux across the boundary. Inspired by the foundational work of Allen et al. \cite{ABL}, an extensive body of research has emerged focusing on reaction-diffusion epidemic models; see \cite{BDLZ,CCui,CL,CLL,CS, CLPZ,CS2,DW,DKLS2025,KMP,GKLZ,PY, LLT, LZZ, LPW, LPW1, LPX, LLou, LS, LS2023,LSS2023, LX2024,MWW1, MWW2, PZW, PWZZ,PWu,PSW1,PSW2,PZ,TWin1,TWin, TL, Sa,SLX,WWK, WW,WJL, WZ, WWL} among others. These studies delve into the complex dynamics of disease transmission, examining how diseases spread under a variety of conditions and influences. 

Diffusive SIS epidemic models with a modified frequency-dependent infection mechanism, represented as \({SI}/{(m+S+I)}\), have garnered considerable research interest recently \cite{AM,CCui1,CCui,GLPZ-2024}. This model incorporates a nonnegative saturated incidence function \(m\) to account for the limitations on the number of effective contacts an individual can make. These limitations are influenced by various factors, including population density, spatial distribution, and time constraints. Furthermore, \(m\) may represent a subpopulation with natural resistance to specific infections, potentially due to genetic mutations or other biological factors. This approach has been particularly influential in studies by Diekmann and Kretzschmar \cite{DK} and Roberts \cite{R}. For example, a notable COVID-19 human challenge trial \cite{Killingley} demonstrated that 16 out of 36 participants inoculated with SARS-CoV-2 did not become infected. Such natural resistance has also been observed in other infectious diseases, such as HIV/AIDS, malaria, and norovirus, underscoring the importance of incorporating these factors into epidemic modeling. Additional studies, such as those by Anderson and May \cite{AM2} and Heesterbeek \cite{Hee2000}, further emphasize the significance of considering variable susceptibility and resistance in understanding epidemic dynamics.

Specifically, the authors in \cite{GLPZ-2024} explored the following SIS model featuring a saturated incidence function in a spatially heterogeneous environment:
\begin{equation}\label{SIS2}
\begin{cases}
\displaystyle
    \p_tS- d_S \Delta S = -\frac{\beta(x)SI}{m(x)+S+I} + \gamma(x) I, & x \in \Omega, \ t > 0, \vspace{1mm} \\
\displaystyle
    \p_tI- d_I \Delta I = \frac{\beta(x)SI}{m(x)+S+I} - \gamma(x) I, & x \in \Omega, \ t > 0, \vspace{1mm} \\
\displaystyle
 \partial_{\nu}S=\p_{\nu}I=0, & x \in \partial\Omega, \ t > 0, \vspace{1mm} \\
S(0,x) = S_0(x) \ge 0, \ I(0,x) = I_0(x) \ge 0, & x \in \Omega.
\end{cases}
\end{equation}
Here, the  saturated incidence function $m\in C(\bar\Omega)$ satisfies  \(m \geq 0\) and \(m \not\equiv 0\), and
all other parameters and state variables have the same epidemiological meanings as in \eqref{SIS1} (the authors also considered the scenario where parameters $\beta$, $\gamma$, and $m$ are time-periodic). 

Taking into account the realistic implications, it is assumed that the initial data ${S}_0$ and $I_0$ are both nonnegative continuous functions on $\bar{\Omega}$ with $I_0\geq,\not\equiv0$. Thus, by the standard theory for parabolic equations, \eqref{SIS2} admits a unique classical solution $S, I\in C^{1,2}((0,\infty)\times\bar{\Omega}$). Moreover, the strong maximum principle and Hopf boundary Lemma for parabolic equations ensure that $S(t,x)>0$ and $I(t,x)>0$ for all $x\in \bar{\Omega}$ and $t\in(0, \infty)$.

Adding the first two equations in \eqref{SIS2} and integrating the resulting equation over \(\Omega\), one can derive that:
\begin{equation}\label{sum}
\int_{\Omega}(S(t,x) + I(t,x)) \, \textrm{d}x \equiv \int_{\Omega}(S_0(x) + I_0(x)) \, \textrm{d}x =: N, \quad \forall t \geq 0,
\end{equation}
indicating that the total population  $N$ is conserved. Throughout this paper, \(N\) is supposed to be a fixed positive constant.

In \cite{GLPZ-2024}, the authors analyzed the threshold dynamics of the model \eqref{SIS2}, establishing conditions for disease extinction and persistence based on the basic reproduction number. They demonstrated the global attractivity of both the disease-free equilibrium and the endemic equilibrium by constructing suitable Lyapunov functions under certain conditions. Additionally, they determined the spatial distribution of the disease when the dispersal rates $d_S$ or/and $d_I$ is sufficiently small.
Compared to previous works on model \eqref{SIS1}, their results suggest that the presence of a saturation effect reduces the transmission risk of the disease. This effect enables the total population size to significantly influence disease dynamics, potentially leading to substantial alterations in the spatial distribution of the disease under specific conditions.

%When a disease control strategy involving quarantine and lockdown is implemented, it serves to significantly limit the movement of the population. This is a crucial measure in controlling the spread of infectious diseases, as it reduces the interactions between individuals, thereby decreasing transmission opportunities. In mathematical models, such as the SIS model, this is represented by the movement rates of individuals. Specifically, the movement rate \(d_S\) of the susceptible population and the movement rate \(d_I\) of the infected population are restricted to be small. This adjustment reflects the reduced mobility due to quarantine and lockdown, aiming to help to control the disease's spread.

Quarantine and lockdown measures significantly limit population movement, reducing individual interactions and transmission opportunities. In the diffusive SIS models, this is reflected by reducing the movement rates  $d_S$ of the susceptible population and 
 $d_I$ of the infected population. Formally, when the movement rate \(d_S\)  is controlled to approach zero, system \eqref{SIS2}-\eqref{sum} reduces to the following parabolic-ODE coupled system:
\begin{equation}\label{ds=0-model}
\begin{cases}
\displaystyle
    \p_t S = \gamma(x) I - \frac{\beta(x) S I}{m(x) + S + I}, & x \in \Omega,\, t > 0, \\
\displaystyle
    \p_t I = d_I \Delta I - \gamma(x) I + \frac{\beta(x) S I}{m(x) + S + I}, & x \in \Omega,\, t > 0, \\
\displaystyle
    \p_{\nu} I = 0, & x \in \partial \Omega,\, t > 0, \\
\displaystyle
    S( 0,x) = S_0(x) \ge 0, \ \ I( 0,x) =I_0(x)\geq0, & x \in \Omega, \\
\displaystyle
    \int_{\Omega} (S(t,x) + I(t,x))dx = N, & t \ge 0.
\end{cases}
\end{equation}
Similarly, as the movement rate \(d_I\) of the infected individuals approaches zero, system \eqref{SIS2}-\eqref{sum} reduces to the following parabolic-ODE system:
\begin{equation}\label{dI=0-model}
\begin{cases}
\displaystyle
    \p_t S = d_S \Delta S + \gamma(x) I - \frac{\beta(x) S I}{m(x) + S + I}, & x \in \Omega,\, t > 0, \\
\displaystyle
    \p_t I = -\gamma(x) I + \frac{\beta(x) S I}{m(x) + S + I}, & x \in \Omega,\, t > 0, \\
\displaystyle
    \p_{\nu} S = 0, & x \in \partial \Omega,\, t > 0, \\
\displaystyle
    S( 0,x)= S_0(x) \ge 0, \ \ I( 0,x) =I_0(x)\geq0, & x \in \Omega, \\
\displaystyle
    \int_{\Omega} (S(t,x) + I(t,x)) dx = N, & t \ge 0.
\end{cases}
\end{equation}
 
In \cite{LS, salako2024degenerate}, the authors investigated the long-term dynamics of both models \eqref{ds=0-model} and \eqref{dI=0-model} when \(m \equiv 0\). The primary aim of this paper is to address \eqref{ds=0-model} and \eqref{dI=0-model} when \(m \not\equiv 0\) and to understand how the saturated incidence function \(m\) affects disease persistence, extinction, and spatial distribution. It turns out that many techniques in \cite{salako2024degenerate} cannot be used to investigate \eqref{ds=0-model} and \eqref{dI=0-model}. Therefore, novel analytical techniques are needed to overcome the inherent mathematical challenges.

Our results demonstrate that the saturated incidence function \(m\), combined with the transmission risk level—determined by the difference between disease transmission and recovery rates—and the total population size, plays a crucial role in the global dynamics of \eqref{ds=0-model} and \eqref{dI=0-model}. In future work, we plan to rigorously verify that the solution to \eqref{SIS2} converges to that of \eqref{ds=0-model} as \(d_S \to 0\) and to \eqref{dI=0-model} as \(d_I \to 0\), respectively. We will provide a more detailed discussion of our theoretical results and their implications in the discussion section, highlighting the critical insights gained from our analysis. We would like to mention that the works \cite{Burie1, Burie2} employ another approach to investigate the impact of population movement constraints on disease spread; they studied the concentration phenomenon in an epidemic model by establishing an asymptotic shape of the stationary solution for a small diffusion parameter (but not completely zero).

The rest of the paper is organized as follows. In Section 2, we present the main theoretical results for both systems \eqref{ds=0-model} and \eqref{dI=0-model}. Section 3 is dedicated to the proofs of the main results for system \eqref{ds=0-model}, and the proofs of the results  for system \eqref{dI=0-model} are in Section 4. In Section 5, we conduct numerical simulations to support and enhance our theoretical findings. Finally, in Section 6, we conclude the paper with the discussion on the implications of our results for disease control and eradication strategies.

\section{Statement of  main results}
In this section, we present the main results  for systems \eqref{ds=0-model} and \eqref{dI=0-model}. Throughout the paper, we denote \( C^+(\bar{\Omega})= \{u\in C(\bar{\Omega}):\ u\geq0\}\) and $W^{2,p}_\nu(\Omega):=\{u\in W^{2,p}(\Omega): \ \partial_{\nu}u=0\ \text{on}\ \partial\Omega\}$ for $p\ge 1$. Given a set $A\subseteq\mathbb{R}^n$, $|A|$ represents the Lebesgue measure of $A$. We also use $\|u\|_{\infty}$ to represent the usual $L^\infty$ norm of $u\in L^\infty(\Omega)$.

Define
$$M^0 := \{ x \in \bar{\Omega} :\  m(x) = 0 \}\ \ \text{and}\ \
M^+ := \{ x \in \bar{\Omega} :\  m(x) > 0 \}.
$$
Thus, \(M^0\) and \(M^+\) represent, respectively, the regions of absence and presence of the saturated incidence effect.

The region $\bar\Omega$ is divided into three sub-regions according to the risk levels:

%We classify transmission risk into three levels: {\it high risk}, {\it moderate risk}, and {\it low risk}, defined as follows:
\noindent\( H^+ := \{x \in \bar\Omega :\ \beta(x) > \gamma(x)\} \) is the {\it high-risk} region;\\
\( H^0 := \{x \in \bar\Omega :\ \beta(x) = \gamma(x)\} \) is the {\it moderate-risk} region;\\
\( H^- := \{x \in \bar\Omega :\ \beta(x) < \gamma(x)\} \) is the {\it low-risk} region.

\subsection{Results for system \eqref{ds=0-model} }
The following result is concerned about the global existence of the solution to the system \eqref{ds=0-model}. 
\begin{prop}\label{theorem_ds1}
 For any initial data $(S_0,I_0)\in [C^+(\bar{\Omega})]^2$,   system \eqref{ds=0-model} admits a  unique  solution  $(S(t,x),I(t,x))$ defined for all $t\ge 0$ and $x\in\bar{\Omega}$ with 
$$
S,\ I\in C([0,\infty), C^+(\bar{\Omega}))\cap C^{1}((0,\infty), C(\bar{\Omega})) \quad \text{and}\quad   I\in   C((0,\infty), \cap_{p\ge 1}W_\nu^{2,p}({\Omega})).
$$  
If $I_0\not\equiv 0$, then $S(t,x)>0$ and $I(t,x)>0$ for all   $x\in\bar{\Omega}$ for all $t>0$.
\end{prop}

Note that if $I_0\equiv0$ on $\bar\Omega$, then $I(t,x)\equiv0$ on $\bar\Omega$ for all $t\geq0$. Therefore, in the rest of this subsection, we will always assume $(S_0,I_0)\in [C^+(\bar{\Omega})]^2$ with $I_0\not\equiv 0$.  We first consider  the case when the entire habitat \(\bar{\Omega}\) is a high-risk region, i.e. \(\beta > \gamma\) on \(\bar{\Omega}\). 
\begin{tm}\label{th2.1} Fix $d_I>0$ and suppose  $H^+=\bar\Omega$. Let  $(S,I)$ be the solution of \eqref{ds=0-model}. Then exactly one of the following two statements holds:
\begin{itemize}
    \item[\rm (i)] $\|I(t,\cdot)\|_{\infty}\to 0$ and $\|S(t,\cdot)-S^*\|_{\infty}\to 0$ as $t\to\infty$ for some $S^*\in C^+(\bar{\Omega})$ satisfying $\int_{\Omega}S^*=N$ and $S^*>0$ on $M^+$; 

    \item[\rm (ii)]  $\|S(t,\cdot)-\tilde{S}\|_{\infty}\to 0$ and $\|I(t,\cdot)-\tilde{I}\|_{\infty}$ as $t\to\infty$, where 
    \begin{equation}\label{th2.1-eq1}
        \tilde{S}:=\frac{\gamma}{\beta-\gamma}\left(\frac{N-\int_{\Omega}\frac{m\gamma}{\beta-\gamma}}{\int_{\Omega}\frac{\beta}{\beta-\gamma}}+m\right)\quad \text{and}\quad 
        \tilde{I}:=\frac{N-\int_{\Omega}\frac{m\gamma}{\beta-\gamma}}{\int_{\Omega}\frac{\beta}{\beta-\gamma}}.
    \end{equation}
\end{itemize}
Moreover, if $N\le\int_{\Omega}\frac{m\gamma}{\beta-\gamma}$, then {\rm (i)} holds.
\end{tm}

Our next result identifies sufficient conditions under which  alternative {\rm (ii)} of Theorem \ref{th2.1} holds.
\begin{prop}\label{prop1} Fix $d_I>0$. Suppose that $H^+=\bar\Omega$ and $N>\int_{\Omega}\frac{m\gamma}{\beta-\gamma}$.
If either 
\begin{itemize}
\item[\rm(i)] $S_0\ge \frac{m\gamma}{\beta-\gamma}$, or 
\item[\rm (ii)] $S_0\le \frac{m\gamma}{\beta-\gamma}$ and $\|I_0\|_{\infty}<\frac{N-\int_{\Omega}\frac{m\gamma}{\beta-\gamma}}{\int_{\Omega}\frac{\gamma}{\beta-\gamma}}$,
\end{itemize}
then  alternative {\rm (ii)} of Theorem \ref{th2.1} holds.
\end{prop}

\begin{rk}\label{re1}
In light of Proposition \ref{prop1}, we conjecture that  alternative {\rm (ii)} of Theorem \ref{th2.1} holds if $N>\int_{\Omega}\frac{m\gamma}{\beta-\gamma}$. We note that this is the case if $m\equiv0$ on $\bar\Omega$ (\cite[Theorem 4.2(ii)]{salako2024degenerate}).
\end{rk}

If the entire habitat is a high-risk region,  Theorem \ref{th2.1} states that whether the strategy of limiting the movement of susceptible people can eliminate the disease depends on the total population size $N$: if $N$ is small, the disease can be controlled; while if $N$ is large, then it may not be controlled. This is very different from the case $m\equiv 0$.

%consists of high-risk locations, the restriction strategy on the movement of the susceptible population suggests a significant insight from Theorem \ref{th2.1}: the fate of the disease—whether it becomes extinct or not—depends critically on the total population size \(N\). Specifically, the infectious disease will persist throughout the habitat in the long run only when the total population \(N\) exceeds the threshold \(\int_{\Omega}\frac{m\gamma}{\beta-\gamma}\). Otherwise, the disease will eventually die out. Conversely, the susceptible population will not go extinct in parts of the habitat that support a positive saturated incidence effect. Particularly, if the saturated incidence effect is present everywhere, the susceptible population will persist throughout the entire habitat. Furthermore, if the conditions outlined in Proposition \ref{prop1} are met, the susceptible population will persist uniformly across the entire habitat.

%This situation raises the question of whether the susceptible population might eventually become extinct in regions where the saturated incidence effect is absent (i.e., \(m(x) = 0\) for \(x \in \bar{\Omega}\)). In this context, we can state the following result.

We next deal with the case when the habitat \(\bar{\Omega}\) consists only of both high-risk and moderate-risk locations, that is, $H^+\cup H^0=\bar\Omega$ and  $H^0\neq\emptyset$. 

\begin{tm}\label{th2.2}
Fix $d_I>0$. Suppose that $\bar\Omega=H^+\cup H^0$ and $H^0\neq\emptyset$. Let $(S,I)$ be the solution of \eqref{ds=0-model}. The following statements hold:
\begin{enumerate}    
    \item[\rm (i)] If $H^0\cap M^+$ has nonempty interior, then 
 $$
 \mbox{ $\lim_{t\to\infty}\|S( t,\cdot)-S^*\|_\infty=0$\
and\ \, $\lim_{t\to\infty}\|I( t,\cdot)\|_{L^\infty(\Omega)}=0$,}
$$
where $S^*\in C^+(\bar\Omega)$ and $S^*>0$ on $M^+\cup H^0$;
\item[\rm (ii)]   If $H^0$ has zero measure
and $\int_\Omega 1/(\beta-\gamma)dx=\infty$, then there exists $\{t_k\}$ converging to infinity such that  
 $I( t_k,\cdot)\to 0$ in $C(\bar\Omega)$ and $\int_{\Omega}S(t_k,\cdot)dx\to N$ as $k\to\infty$.
\end{enumerate}
\end{tm}

Theorem \ref{th2.2} suggests that if the entire habitat \(\bar{\Omega}\) consists solely of high-risk and moderate-risk locations,  the disease will be driven to extinction in the long run when the movement of susceptible individuals is controlled.

Finally, we consider the case that the habitat \(\Omega\) includes some low-risk locations, i.e., there exists \(x \in \bar{\Omega}\) such that \(\beta(x) < \gamma(x)\). 
\begin{tm}\label{th2.3}
    Fix $d_I>0$ and suppose  $H^-\neq\emptyset$. Let $(S,I)$ be the solution of \eqref{ds=0-model}.  Then $\|I(t,\cdot)\|_{\infty}\to 0$ and $\|S(t,\cdot)-S^*\|_{\infty}\to 0$ as $t\to\infty$ for some $S^*\in C^+(\bar{\Omega})$ satisfying $\int_{\Omega}S^*=N$ and $S^*>0$ on $M^+\cup H^0\cup H^-$.
\end{tm}

Theorem \ref{th2.3} demonstrates that if $\Omega$ contains low-risk sites, restricting the movement of susceptible individuals will lead to disease extinction. Furthermore, the susceptible population will reside in moderate- and low-risk locations, as well as in areas where the effects of saturated incidence are present. 

Theorems \ref{th2.1}, \ref{th2.2} and \ref{th2.3} conclude that \(S^* > 0\) in \(\bar{\Omega} \setminus (H^+ \cap M^0)\). Due to the continuity of $S^*$, $S^*>0$ in some region of $H^+ \cap M^0$. This raises the question whether the susceptible population may become extinct in some places where the risk-level is high and the saturated incidence effect is absent, i.e., within $H^+ \cap M^0$.  Our following result provides an affirmative answer to this question. 

\begin{prop}\label{prop2} Fix $d_I>0$. Suppose that there is an open subset \(\mathcal{O} \subset M^0\) with a Lipschitz boundary such that 
\(\lambda_0>0\), where \(\lambda_0\) is the principal eigenvalue  of the linear elliptic eigenvalue problem:
\begin{equation}\label{eigen-prob1-1}
    \begin{cases}
      d_I\Delta \varphi +(\beta -\gamma)\varphi=\lambda\varphi, & x\in\mathcal{O},\cr
      \varphi=0, & x\in\partial\mathcal{O}.
    \end{cases}
\end{equation}
Let $(S,I)$ be the solution of \eqref{ds=0-model} such that $\|I(t,\cdot)\|_{\infty}\to 0$ and $\|S(t,\cdot)-S^*\|_{\infty}\to 0$ as $t\to\infty$ for some  $S^*\in C^+(\bar{\Omega})$, then we have 
$$|\{x\in\mathcal{O} :\  S^*(x)=0\}|>0.$$
\end{prop}

Combining Theorems \ref{th2.1}, \ref{th2.2} and \ref{th2.3}, and Proposition \ref{prop2}, we can prove the following conclusion.

\begin{coro}\label{cor1}  Let $(S,I)$ be the solution of \eqref{ds=0-model}. Suppose that one of the following conditions holds:
\begin{itemize}
    \item[{\rm (i)}] $H^+=\bar\Omega$,  $M^0$ has nonempty interior, and $N\leq\int_{\Omega}\frac{\gamma m}{\beta-\gamma}$;
    \item[{\rm (ii)}] $H^+\cup H^0=\bar\Omega$ and both $M^+\cap H^0$ and  $M^0\cap H^+$ have nonempty interior;
    \item[{\rm (iii)}] Both $H^-$ and $M^0\cap H^+$ have nonempty interior.
\end{itemize}
Then there is $d_0>0$ such that for every $0<d_I<d_0$, there is a measurable set $\Omega^*\subseteq M^0\cap H^+$ with a positive measure such that $S(t,x)\to 0$ as $t\to\infty$ for every $x\in\Omega^*$. 
\end{coro}

\begin{rk}\label{re2} Let \(\lambda_0\) be defined in Proposition \ref{prop2}. It is well known that $\lambda_0$ is  decreasing in \(d_I \in (0, \infty)\) with
\[
\lim_{d_I \to 0} \lambda_0 = \max_{x \in \bar{\mathcal{O}}} (\beta(x) - \gamma(x)) \ \ \text{and} \quad \lim_{d_I \to \infty} \lambda_0 = -\infty.
\]
Take $\mathcal{O}=M^0\cap H^+$  and assume that $M^0\cap H^+\subset\Omega$ is a domain with Lipschitz boundary. Thus, \(\max_{x \in M^0\cap H^+} (\beta(x) - \gamma(x)) > 0\), and there is a critical value \(d_0^* > 0\) such that \(\lambda_0 > 0\) for \(0 < d_I < d_0^*\) and \(\lambda_0 \leq 0\) for \(d_I \geq d_0^*\).
Thus, it is interesting to explore whether \(S^*\) can vanish somewhere in \(M^0\cap H^+\) for \(d_I \geq d_I^*\) in the context of Theorem \ref{th2.1}(i), Theorem \ref{th2.2} and Theorem \ref{th2.3}. Our numerical simulations suggest that \(S^*>0\) in \(M^0\cap H^+\) for large $d_I$ (see Figure 1(C) in Section 5), which would imply that the susceptible population could occupy the entire habitat in such a case.
\end{rk}

\subsection{Results for system \eqref{dI=0-model}}
The following result concerns the existence and nonnegativity of the solutions of system \eqref{dI=0-model}. We omit the proof since it is similar to that of Proposition \ref{theorem_ds1}.  
\begin{prop}\label{theorem_di1}
 For any initial data $(S_0,I_0)\in [C^+(\bar{\Omega})]^2$,   system \eqref{dI=0-model} admits a  unique  solution  $(S(t,x),I(t,x))$ defined for all $t\ge 0$ and $x\in\bar{\Omega}$ with 
$$
S,\ I\in C([0,\infty), C^+(\bar{\Omega}))\cap C^{1}((0,\infty), C(\bar{\Omega})) \quad \text{and}\quad   S\in   C((0,\infty), \cap_{p\ge 1}W_\nu^{2,p}({\Omega})).
$$  
If $I_0\not\equiv 0$, then $S(t,x)>0$ for all   $x\in\bar{\Omega}$ for all $t>0$; if $I_0(x_0)>0$ for some $x_0\in\bar\Omega$, then $I(t, x_0)>0$ for all $t\ge 0$; $I_0(x_0)=0$ for some $x_0\in\bar\Omega$, then $I(t, x_0)=0$ for all $t\ge 0$.
\end{prop}

In this rest of this subsection, we always assume $(S_0,I_0)\in [C^+(\bar{\Omega})]^2$, $I_0\not\equiv0$, and  $\int_\Omega (S_0+I_0)=N$. 

 \begin{tm}\label{th2.4} Let $(S,I)$ be the solution of \eqref{dI=0-model}. Then there exist two positive numbers $c_1,\,c_2$,  independent of initial data and $m$ such that 
 \begin{equation}\nonumber
     c_1\leq\liminf_{t\to\infty}\min_{x\in\bar{\Omega}}S(t,x)\,\ \mbox{and}\ \ \limsup_{t\to\infty}(\|S(t,\cdot)\|_{\infty}+\|I(t,\cdot)\|_{\infty})\le c_2.
 \end{equation}
 \end{tm}

According to Theorem \ref{th2.4}, if the movement of the infected population is  restricted, the susceptible population will eventually occupy the entire habitat, regardless of the transmission and recovery rates or the presence of a saturated incidence effect. Therefore, in our subsequent analysis, we primarily focus on the long-term dynamics of the infected population.

The following result addresses the scenario when the entire habitat is considered high-risk. Denote by $\chi_{\{I_0>0\}}$ the characteristic function of the subset $\{x\in\bar\Omega:\ I_0>0\}$ of  $\bar\Omega$.

 \begin{tm}\label{th2.5} Fix $d_S>0$ and suppose that $H^+=\bar{\Omega}$.  Let $(S,I)$ be the solution of \eqref{dI=0-model}.   
 Then we have
\begin{equation}\label{th2.5-eq1}
     \lim_{t\to\infty}\Big\|S(t,\cdot)-\frac{1}{|\Omega|}\int_{\Omega}S(t,\cdot)dx\Big\|_{\infty}=0,
\end{equation}
and for any $x\in\{I_0>0\}:=\{x\in\bar\Omega:\ I_0(x)>0\}$,
    \begin{equation}\label{th2.5-eq2}
      \liminf_{t\to\infty}I(t,x)>0\ \  \text{if}\  m(x)=0, 
    \end{equation}
and 
    \begin{equation}\label{th2.5-eq3}
        \lim_{t\to\infty}I(t,x)=0\ \ \text{if} \ N\le \frac{|\Omega|m(x)\gamma(x)}{\beta(x)-\gamma(x)}.
    \end{equation}
     Furthermore,   the following conclusions hold.
    \begin{itemize}
        \item[\rm(i)] If either $$\frac{|\Omega|m\gamma}{\beta-\gamma}<\min\Big\{N,\ N+\int_{\{I_0>0\}}m        -\frac{N}{|\Omega|}\int_{\{I_0>0\}}\frac{\beta-\gamma}{\gamma}\Big\}$$ or $$\frac{m\gamma}{\beta-\gamma}\Big(|\Omega|+\int_{\{I_0>0\}}\frac{\beta-\gamma}{\gamma}\Big)<N,
            $$  
            then $S(t,\cdot)\to \hat{S}$ uniformly on $\bar{\Omega}$ and $I(t,\cdot)\to \hat{I}$ in $L^p(\Omega)\,(p\geq1)$ as $t\to\infty$, where 
        \begin{equation}
            \hat{S}:=\frac{N+\int_{\{I_0>0\}}m}{|\Omega|+\int_{\{I_0>0\}}\frac{\beta-\gamma}{\gamma}}
            \quad \text{and}\quad \hat{I}:=\Big(\frac{\beta-\gamma}{\gamma}\hat{S}-m\Big)\chi_{\{I_0>0\}}. 
        \end{equation}
        \item[\rm(ii)] If   $\frac{|\Omega|m\gamma}{\beta-\gamma}\ge {N}$, then $\|I(t,\cdot)\|_{\infty}\to 0$ and $\|S(t,\cdot)-\frac{N}{|\Omega|}\|_{\infty}\to 0$ as $t\to\infty$.
    \end{itemize}
\end{tm}

\begin{rk} If $I_0>0$ on $\bar\Omega$, under the assumption of Theorem \ref{th2.5}(i), it can be proved that $I(t,\cdot)\to \frac{(\beta-\gamma)}{\gamma}\hat{S}-m$ uniformly on $\bar{\Omega}$ as $t\to\infty$. If additionally $\beta,\,\gamma$ are positive constants and $m$ is a nonnegative constant, then in Theorem \ref{th2.5}(i), the condition $\frac{|\Omega|m\gamma}{\beta-\gamma}<\min\Big\{N,N+\int_{\{I_0>0\}}m
-\frac{N}{|\Omega|}\int_{\{I_0>0\}}\frac{\beta-\gamma}{\gamma}\Big\}$
becomes equivalent to $\gamma<\beta<2\gamma$ and $\frac{|\Omega|m\gamma}{\beta-\gamma}<N$, and the condition $\frac{m\gamma}{\beta-\gamma}\Big(|\Omega|+\int_{\{I_0>0\}}\frac{\beta-\gamma}{\gamma}\Big)<N
            $ becomes equivalent to $\gamma<\beta$ and $\frac{|\Omega|m\beta}{\beta-\gamma}<N$.

On the other hand, if $m\equiv0$ and $H^+=\bar{\Omega}$, then \cite[Theorem 4.6(ii)]{salako2024degenerate} implies Theorem \ref{th2.5}(i).

\end{rk}

Theorem \ref{th2.5} examines the long-term behavior of solutions of model \eqref{dI=0-model} under the condition that \(\beta > \gamma\) across the entire habitat \(\bar{\Omega}\). It suggests that, when the infected individuals are restricted, the susceptible population will eventually distribute homogeneously throughout the habitat. The  fate of the infectious disease, however, is influenced by several factors, particularly the total population size and the presence of the saturated incidence effect. In the absence of the saturated incidence effect, the infectious disease is likely to persist across the entire habitat over time. Conversely, when the saturated incidence effect is considered, the outcome depends significantly on the total population size: a large population size tends to support the persistence of the disease, while a smaller population size may lead to the eventual extinction of the disease. 

We next consider the case where the local infection rate is less or equal to the local recovery rate. Our result is stated as follows.

\begin{tm}\label{th2.6}  Fix $d_S>0$.  Let $(S,I)$ be the solution of \eqref{dI=0-model}.  Then we have
    \begin{equation}
\lim_{t\to\infty}I(t,x)=0\ \ \mbox{for any $x\in H^0\cup H^-$.}
    \end{equation}
    Furthermore, the following conclusions hold.
    \begin{itemize}
        \item[\rm(i)] If $H^0\cup H^-=\bar\Omega$, then $\|I(t,\cdot)\|_{\infty}\to 0$ and $\|S(t,\cdot)-\frac{N}{|\Omega|}\|_{\infty}\to 0$ as $t\to\infty$. 

        \item[\rm(ii)] If $H^+\ne\emptyset$, then for every compact subset $K\subset H^+$, there is $N_{*}>0$, depending on $m$ and $K$, but independent of initial data, such that if $N>N_{*}$, 
        \begin{equation}
        \liminf_{t\to\infty}\min_{x\in K}I(t,x)>0.    
        \end{equation} 
        Moreover, $N_{*}=0$ if $m=0$ on $K$.
    \end{itemize}    
\end{tm}

Theorem \ref{th2.6} shows that in areas where the local infection rate is less than or equal to the local recovery rate, the infection will eventually be eradicated. However, the disease cannot be eradicated in high-risk regions,  when the total population number is sufficiently large.

The following result and its corollary, serving as an interesting complement to Theorems \ref{th2.5} and \ref{th2.6}, provide the precise long-term behavior of the two populations in scenarios where a high-risk subregion may exist. Let $R(x)={\beta(x)}/{\gamma(x)}$.

\begin{prop}\label{prop3}  Fix $d_S>0$. Suppose that $1/(\beta-\gamma)\in L^1(H^+\cap \{I_0>0\})$ and $|\Omega|>\int_\Omega (R-1)_+\chi_{\{I_0>0\}}$.  Let  $(S,I)$ be the solution of \eqref{dI=0-model}.   Then we have
 $$
\lim_{t\to\infty} S(t,\cdot)=S^*\ \ \text{and}\ \ \lim_{t\to\infty} I(t,\cdot)=((R-1)S^*-m)_+\chi_{\{I_0>0\}}\ \ \text{uniformly on}\ \bar\Omega,
 $$
 where the constant $S^*>0$ is the unique solution of the following algebraic equation:
\begin{equation}\label{S^*-def}
S^*|\Omega|+\int_\Omega ((R-1)S^*-m)_+\chi_{\{I_0>0\}}=N. 
\end{equation}
\end{prop}

\begin{coro}\label{coro2.3} Fix $d_S>0$. Suppose that $1/(\beta-\gamma)\in L^1(H^+\cap \{I_0>0\})$ and $|\Omega|>\int_\Omega (R-1)_+\chi_{\{I_0>0\}}$.  Let  $(S,I)$ be the solution of \eqref{dI=0-model}. The following assertions hold.
   \begin{enumerate}
       \item[\rm(i)] If there exists $x\in H^+\cap\{I_0>0\}$ such that $(R(x)-1)N/|\Omega|-m(x)>0$, then 
       $$
\lim_{t\to\infty} S(t,\cdot)=S^*\ \ \text{and}\ \ \lim_{t\to\infty} I(t,\cdot)=((R-1)S^*-m)_+\chi_{\{I_0>0\}}\not\equiv 0\ \ \text{uniformly on}\ \bar\Omega,
 $$
where $S^*\in (0, N/|\Omega|)$ is determined by \eqref{S^*-def}.

  \item[\rm(ii)] If $(R-1)N/|\Omega|-m\le 0$ on $H^+\cap\{I_0>0\}$, then 
  $$
\lim_{t\to\infty} S(t,\cdot)=\frac{N}{|\Omega|}\ \ \text{and}\ \ \lim_{t\to\infty} I(t,\cdot)= 0\ \ \text{uniformly on}\ \bar\Omega.
 $$
   \end{enumerate}
\end{coro}

\begin{rk}\label{re3} We suspect that the conditions $1/(\beta-\gamma)\in L^1(H^+\cap \{I_0>0\})$ and $|\Omega|>\int_\Omega (R-1)_+\chi_{\{I_0>0\}}$ in Proposition \ref{prop3} and Corollary \ref{coro2.3} are merely technical.
\end{rk}

By Theorem \ref{th2.6}(ii), when the effect of saturated incidence is negligible in high-risk areas (i.e., \( m = 0 \) on \( H^+ \)), the infectious disease will ultimately confine itself to these high-risk areas, regardless of the total population size. However, as stated in Proposition \ref{prop3} and Corollary \ref{coro2.3}, if the saturated incidence effect is present in high-risk areas, the total population size plays a pivotal role in determining the disease's eventual distribution. Specifically, while the disease may persist only in certain high-risk locations, the area it can ultimately occupy expands as the total population size increases, since \( S^* \) increases strictly with \( N \).

%Moreover, from  Corollary \ref{coro2.3}, we can conclude that if there exists a high-risk location \( x \) where the disease initially breaks out and satisfies the condition \( (R(x)-1)N/|\Omega|-m(x)>0 \) (indicating a very high transmission risk at that location), the disease will not become extinct in the long run across the entire habitat. Conversely, if at every high-risk location $x$, the condition \( (R(x)-1)N/|\Omega|-m(x) \leq 0 \) holds (suggesting that the transmission risk is not sufficiently high at any of these locations), then the disease will eventually die out in the long run, even if it initially breaks out in all such high-risk locations.

\section{Proof of main results for system \eqref{ds=0-model}}
In this section, we prove the main results for system \eqref{ds=0-model}  stated in Subsection 2.1.

\begin{proof}[Proof of Proposition \ref{theorem_ds1}]
 The local existence of nonnegative solutions of  \eqref{ds=0-model} with specified regularity follows from a standard semigroup method (see, for example, \cite[Proposition 3.1]{salako2024degenerate}).Using $SI/(m+S+I)\le I$, we have
 \begin{equation}\label{ds_com}
\begin{cases}
\displaystyle
    \p_t S \le \gamma(x) I , & x \in \Omega,\, t > 0, \\
\displaystyle
    \p_t I \le d_I \Delta I - \gamma(x) I + \beta(x) I, & x \in \Omega,\, t > 0, \\
\displaystyle
    \p_{\nu} I = 0, & x \in \partial \Omega,\, t > 0.
\end{cases}
\end{equation}
Since the right-hand side of \eqref{ds_com} is linear, by the comparison principle for parabolic equations, the solution of \eqref{ds=0-model} does not blow up at finite time. So it exists globally. The positivity of the solution when $I_0\not\equiv 0$ follows from the strong maximum principle. 
\end{proof}

 The following lemma is well-known, and its proof can be found, for example, in \cite[Lemma 2.2]{salako2024degenerate}.
\begin{lem}\label{lem0} Suppose that $\phi$ is a nonnegative and H\"{o}lder continuous function on $[0,\infty)$ and $
\int_0^\infty\phi(t)dt<\infty$. Then $\phi(t)\to0$ as $t\to\infty$.
\end{lem}

The following Hanack type inequality can be found in  \cite{huska2006harnack}:
\begin{lem}\label{lemma_harnack}
Let $u$ be a  nonnegative solution of the following problem on $\Omega\times (0, T)$:
\begin{equation*}\label{model-aux}
\begin{cases}
\displaystyle\partial_t u=\Delta u+ a(x, t) u,&x\in\Omega,\, t>0,\\
\partial_\nu u=0, &x\in\partial\Omega,\,t>0,
\end{cases}
\end{equation*}
where $a\in L^\infty(\Omega\times (0, \infty))$. Then for any $0<\delta<T$, there exists $C>0$ depending on $\delta$ and $\|a\|_{L^\infty(\Omega\times (0, \infty))}$ such that 
$$
\sup_{x\in\Omega} u(x, t)\le C\inf_{x\in\Omega} u(x, t), \ \ \text{for all }\ t\in [\delta, T).
$$
\end{lem}

The following result follows readily from Lemma \ref{Harnak-inequality}.
\begin{lem}\label{lem1} Let $(S,I)$ be the solution of \eqref{ds=0-model}. Then there exist constants $M=(d_I,N)>0$ and $c_1=c_1(d_I,N)>0$ such that 
\begin{equation}\label{Harnak-inequality}
    \|I(t,\cdot)\|_{\infty}\le c_1\min_{x\in\bar{\Omega}}I(t,x),\quad \forall\, t\ge 1
\end{equation}
    and \begin{equation}\label{L-infty-bound-for-I}
    \|I(t,\cdot)\|_{\infty}\le M,\quad \forall \, t\ge 1.
\end{equation}
\end{lem}
\begin{proof} Since
\begin{equation}\label{Total-size identity}
\int_{\Omega}(S(t,\cdot)+I(t,\cdot))=N,\quad \forall \, t\ge 0
\end{equation} 
and
\begin{equation}\label{Eqa-12}
\Big\|\frac{\beta S(t,x)}{m+S(t,x)+I(t,x)}-\gamma\Big\|_{\infty}\le \|\beta\|_{\infty}+\|\gamma\|_{\infty}, \quad \forall\, t\ge 0,
\end{equation}
one can use Lemma \ref{Harnak-inequality} to deduce \eqref{Harnak-inequality}. In addition, \eqref{L-infty-bound-for-I} is a direct consequence of \eqref{Harnak-inequality} and  \eqref{Total-size identity}.
\end{proof}

To apply Lemma \ref{lem0}, we need the following result.

\begin{lem}\label{lem2} Suppose that $\beta>\gamma$ on $\bar\Omega$. Let $(S,I)$ be the  solution of \eqref{ds=0-model}. Define 
\begin{equation}\label{U-def}
    U(t,x)=\frac{\beta(x)-\gamma(x)}{\gamma(x)}S(t,x)-m(x), \quad \forall\, t\ge 0,\, x\in\bar\Omega,
\end{equation}
and 
\begin{equation}\label{F-defi}
    F(t)=\int_{\Omega}\frac{[U(t,x)-I(t,x)]^2I(t,x)}{m(x)+S(t,x)+I(t,x)}, \quad \forall\, t\ge 0.
\end{equation}
Then we have
\begin{align}\label{Eqa-4}
     S(t,x)= &\frac{\gamma(x)}{\beta(x)-\gamma(x)}(U(t,x)+m(x))\cr
    \le&\Big\|\frac{\gamma}{\beta-\gamma}\Big\|_{\infty}(\|m\|_{\infty}+\max\{M,\|U(1,\cdot)\|_{\infty}\})
    =:\bar{S}^*,\ \ \forall\, t\ge 1,\,x\in\bar\Omega,
\end{align}
where $M$ is given by \eqref{L-infty-bound-for-I}, and the mapping $F$ is H\"older continuous on $[1,\infty)$.
\end{lem}

\begin{proof} Observe that 
\begin{equation}\label{Eqa-3}
\p_tU=\frac{\beta-\gamma}{\gamma}\p_tS=(\beta-\gamma)(I-U)\frac{I}{m+S+I},\quad\,  t\ge0,\, x\in\bar\Omega.
\end{equation}
By the comparison principle of ODE, it follows from \eqref{L-infty-bound-for-I}, \eqref{Eqa-3} and $\beta>\gamma$ on $\bar\Omega$ that 
\begin{equation}\label{Eqa-3a}
U(t,x)\le \max\{M,\ \|U(1,\cdot)\|_{\infty}\},\quad \forall\, t\ge 1.
\end{equation}
As a result, 
\begin{equation*}
    S(t,\cdot)= \frac{\gamma}{\beta-\gamma}(U(t,\cdot)+m)\le \bar{S}^*,\quad \forall\, t\ge 1,
\end{equation*}
which yields \eqref{Eqa-4}.

In view of \eqref{L-infty-bound-for-I} and \eqref{Eqa-12}, the standard regularity theory for parabolic equations allows us to conclude that the mapping $[1,\infty)\ni t\mapsto I(t,\cdot)\in C^1(\bar{\Omega})$ is H\"older continuous. Thus, there exist two constants $0<\theta<1$ and $M_1>0$ such that 
\begin{equation}\label{I-Holder-estimate}
    \|I(t+\tau,\cdot)-I(t,\cdot)\|_{C^1(\bar{\Omega})}\le M_1\tau^{\theta},\quad \forall\, t\ge 1, \ \tau>0.
\end{equation}
Observe from the second equation of \eqref{ds=0-model}, \eqref{Eqa-12}, and the comparison principle for parabolic equations that 
$$
\min_{x\in\bar{\Omega}}I(t+\tau,x)\ge e^{-(\|\beta\|_{\infty}+\|\gamma\|_{\infty})\tau}\min_{x\in\bar{\Omega}}I(t,x),\quad \forall\, t>0,\ \tau>0.
$$
This along with \eqref{Harnak-inequality} gives
\begin{equation}\label{Eqb-1}
    I(t,x)\le  c_1e^{(\|\beta\|_{\infty}+\|\gamma\|_{\infty})\tau}I(t+\tau,x),\quad \forall\, t\ge 1, \ \tau>0, \ x\in\bar{\Omega}.
\end{equation}
Note also from the first equation of \eqref{ds=0-model} that 
$$
\p_tS(t,x)\ge -\|\beta\|_{\infty}S(t,x),\quad \forall\, t>0,\,x\in\bar\Omega.
$$
Hence,  we obtain
\begin{equation}\label{Eqb-5}
S(t+\tau,x)\ge e^{-\|\beta\|_{\infty}\tau}S(t,x),\quad \forall\, t>0,\ \tau>0,\,x\in\bar\Omega.
\end{equation}

For convenience, we will use \( U(t) \), \( S(t) \), and \( I(t) \) to represent \( U(t,x) \), \( S(t,x) \), and \( I(t,x) \), respectively.
Fix $t\ge 1$ and $0<\tau\le 1$. Then,  by \eqref{I-Holder-estimate}, we obtain
\begin{small}
\begin{align}\label{Eqb-6}
   \displaystyle
    &|F(t)-F(t+\tau)|\cr
    \displaystyle
    \le &\int_{\Omega}\Big|\frac{[U(t+\tau)-I(t+\tau)]^2I(t+\tau)}{m+S(t+\tau)+I(t+\tau)}-\frac{[U(t)-I(t)]^2I(t)}{m+S(t)+I(t)}\Big|\cr
    \displaystyle
    \le & \int_{\Omega}\frac{[\frac{(\beta-\gamma)}{\gamma}S(t+\tau)+m+I(t+\tau)]^2|(I(t+\tau)-I(t))|}{m+S(t+\tau)+I(t+\tau)}\cr
    \displaystyle
    &\ \ \ +\int_{\Omega}\Big|\frac{[U(t+\tau)-I(t+\tau)]^2}{m+S(t+\tau)+I(t+\tau)}-\frac{[U(t)-I(t)]^2}{m+S(t)+I(t)}\Big|I(t)\cr
    \displaystyle
    \le & \int_{\Omega}\left\{\left[1+\|\frac{\beta-\gamma}{\gamma}\|_\infty\right]\left[S(t+\tau)+m+I(t+\tau)\right]\right\}|(I(t+\tau)-I(t))|\cr
    \displaystyle
    &\ \ \  +F_{1}(t,\tau)+F_{2}(t,\tau)\cr
    \displaystyle
    \le &M_2\tau^{\theta}+F_{1}(t,\tau)+F_{2}(t,\tau),
\end{align}
\end{small}
where 
 $$
 M_2=\Big(\Big\|\frac{\beta-\gamma}{\gamma}\Big\|_{\infty}+1\Big)\Big(M+\|m\|_{\infty}
    +\bar{S}^*\Big)|\Omega|M_1,
 $$
$$
F_1(t,\tau):=\int_{\Omega}\frac{|(U(t+\tau)-I(t+\tau))^2-(U(t)-I(t))^2|}{m+S(t+\tau)+I(t+\tau)}|I(t)|
$$
and 
$$
F_2(t,\tau):=\int_{\Omega}\Big|\frac{1}{m+S(t+\tau)+I(t+\tau)}-\frac{1}{m+S(t)+I(t)} \Big|(U(t)-I(t))^2I(t).
$$
First, thanks to \eqref{L-infty-bound-for-I} and \eqref{Eqa-3}-\eqref{Eqb-5}, we estimate $F_1(t,\tau)$ as follows:
\begin{align}\label{Eqb-2}
    &\ \  \ F_1(t,\tau) \cr &\le\int_{\Omega}\frac{[\frac{\beta-\gamma}{\gamma}|S(t+\tau)-S(t)|+|I(t+\tau)-I(t)|][|U(t+\tau)|+|U(t)|+I(t+\tau)+I(t)]}{m+S(t+\tau)+I(t+\tau)}I(t)\cr
    &\le M\left [3+c_1e^{(\|\beta\|_{\infty}+\|\gamma\|_{\infty})\tau}
    +\Big\|\frac{\beta-\gamma}{\gamma}\Big\|_{\infty}(1+e^{\|\beta\|_{\infty}\tau})\right]\int_\Omega\Big[\frac{\beta-\gamma}{\gamma}|S(t+\tau)-S(t)|
    +M_1\tau^{\theta}\Big]\cr
    &\le M|\Omega|\left[3+c_1e^{(\|\beta\|_{\infty}+\|\gamma\|_{\infty})\tau}
    +\Big\|\frac{\beta-\gamma}{\gamma}\Big\|_{\infty}(e^{\|\beta\|_{\infty}\tau}+1)\right]\Big(\int_{t}^{t+\tau}\|\p_tU(s)\|_{\infty}ds
    +M_1\tau^{\theta}\Big)\cr
  &\le M|\Omega|\left[3+c_1e^{(\|\beta\|_{\infty}+\|\gamma\|_{\infty})\tau}
  +\Big\|\frac{\beta-\gamma}{\gamma}\Big\|_{\infty}(e^{\|\beta\|_{\infty}\tau}+1)\right]\cr
  &\ \ \ \ \times\Big[(M+\max\{M,\,\|U(1,\cdot)\|_\infty\})\|\beta-\gamma\|_{\infty}\tau
  +M_1\tau^{\theta}\Big]\cr
 &\le M|\Omega|\left[3+c_1e^{(\|\beta\|_{\infty}+\|\gamma\|_{\infty})\tau}
 +\Big\|\frac{\beta-\gamma}{\gamma}\Big\|_{\infty}(e^{\|\beta\|_{\infty}\tau}+1)\right]\cr
 &\ \ \ \ \times
\Big[(M+\max\{M,\,\|U(1,\cdot)\|_\infty\})\|\beta-\gamma\|_{\infty}
+M_1\Big]\tau^{\theta}.
\end{align}

Then, by  \eqref{L-infty-bound-for-I}, \eqref{Eqa-12}, \eqref{Eqa-4}, \eqref{I-Holder-estimate}-\eqref{Eqb-5}, we also estimate $F_2(t,\tau)$ as follows:
\begin{small}
\begin{align}\nonumber%\label{Eqb-3}
   & F_2(t,\tau)\cr
   \le& \int_{\Omega}(|S(t+\tau)-S(t)|+|I(t+\tau)-I(t)|)\frac{|U(t)-I(t)|^2I(t)}{(m+S(t)+I(t))(m+S(t+\tau)+I(t+\tau))}\cr
    \le &\Big[\int_{\Omega}\frac{|U(t)-I(t)|^2I(t)}{(m+S(t)+I(t))(m+S(t+\tau)+I(t+\tau))}\Big]\cr
    &\ \, \times  \Big[\|S(t+\tau)-S(t)\|_{\infty}+\|I(t+\tau)-I(t)\|_{\infty}\Big]\cr 
    \le &\Big[\int_{\Omega}\frac{2(|U(t)|^2+ |I(t)|^2)}{m+S(t+\tau)+I(t+\tau)}\Big]\Big[M(\|\beta\|_{\infty}+\|\gamma\|_{\infty})\tau+M_1\tau^{\theta}\Big]\cr 
    \le & \Big[\int_{\Omega}\frac{2(2(\|(\beta-\gamma)/\gamma\|_{\infty}^2S^2(t)+m^2)+ Mc_1e^{(\|\beta\|_{\infty}+\|\gamma\|_{\infty})\tau}I(t+\tau))}{m+S(t+\tau)+I(t+\tau)}\Big]\cr
    &\ \, \times  \Big[M(\|\beta\|_{\infty}
    +\|\gamma\|_{\infty})\tau+M_1\tau^{\theta}\Big]\cr
    \le & \Big[\int_{\Omega}\frac{2(2(\|(\beta-\gamma)/\gamma\|_{\infty}^2\bar{S}^*e^{\tau\|\beta\|_{\infty}}S(t+\tau)+m^2)+ Mc_1e^{(\|\beta\|_{\infty}+\|\gamma\|_{\infty})\tau}I(t+\tau))}{m+S(t+\tau)+I(t+\tau)}\Big]\cr
    &\ \, \times \Big[M(\|\beta\|_{\infty}
    +\|\gamma\|_{\infty})\tau+M_1\tau^{\theta}\Big]\cr
    \le& 4|\Omega|\Big[{((\|(\beta-\gamma)/\gamma\|_{\infty}^2\bar{S}^*e^{\|\beta\|_{\infty}}+\|m\|_{\infty})+ Mc_1e^{(\|\beta\|_{\infty}+\|\gamma\|_{\infty})}}\Big]\cr
    &\ \, \times
    \Big[M(\|\beta\|_{\infty}+\|\gamma\|_{\infty})+M_1\Big]\tau^{\theta}.
\end{align}
This, together with \eqref{Eqb-6} and \eqref{Eqb-2}, implies that $F$ is H\"older continuous on $[1,\infty)$.
\end{small}
\end{proof}

\vskip10pt
We are now ready to prove Theorem \ref{th2.1}.

\begin{proof}[Proof of Theorem \ref{th2.1}] 
Let $U(t,x)$ be defined as in \eqref{U-def}.  By \eqref{Eqa-3}, we have 
\begin{align}\label{ly}
    \frac{d}{dt}\Big[\int_{\Omega}\frac{\gamma U^2}{2(\beta-\gamma)}+\frac{1}{2}\int_{\Omega}I^2\Big]=&\int_{\Omega}\frac{\gamma IU(I-U)}{m+S+I}+\int_{\Omega}I\Big[d_I\Delta I+\frac{\gamma I(U-I)}{m+S+I}\Big]\cr
    =& -d_I\int_{\Omega}|\nabla I|^2-\int_{\Omega}\frac{\gamma(U-I)^2I}{m+S+I}.
\end{align}
Integrating \eqref{ly} over $(0, t)\times \Omega$ gives
\begin{align}\nonumber
     &\int_{\Omega}\frac{\gamma U^2(t,\cdot)}{2(\beta-\gamma)}+\frac{1}{2}\int_{\Omega}I^2(t,\cdot)+\int_0^{t}\int_{\Omega}\Big[d_I|\nabla I(t,\cdot)|^2+\frac{\gamma(U(t,\cdot)-I(t,\cdot))^2I}{m+S(t,\cdot)+I(t,\cdot)}\Big]\cr
    =&\int_{\Omega}\frac{\gamma U^2(0,\cdot)}{2(\beta-\gamma)}+\frac{1}{2}\int_{\Omega}I^2(0,\cdot),\quad \forall\, t\ge 0.
\end{align}
This implies 
\begin{equation}\label{Eqa-5}
    \int_0^{\infty}\int_{\Omega}\Big[d_I|\nabla I(t,\cdot)|^2+\frac{\gamma(U(t,\cdot)-I(t,\cdot))^2I(t,\cdot)}{m+S(t,\cdot)+I(t,\cdot)}\Big]<\infty.
\end{equation} 

By Lemma \ref{lem2} and \eqref{I-Holder-estimate}, 
 the function 
 $$
 \int_{\Omega}\Big[d_I|\nabla I(t,\cdot)|^2+\frac{\gamma(U(t,\cdot)-I(t,\cdot))^2I(t,\cdot)}{m+S(t,\cdot)+I(t,\cdot)}\Big],\ \ t\in [1,\infty)
 $$
is H\"older continuous. Combined with Lemma \ref{lem0} and \eqref{Eqa-5}, we can infer that
\begin{equation}\label{Eqa-1}    
\lim_{t\to\infty}\int_{\Omega}\Big[d_I|\nabla I(t,\cdot)|^2+\frac{\gamma(U(t,\cdot)-I(t,\cdot))^2I(t,\cdot)}{m+S(t,\cdot)+I(t,\cdot)}\Big]=0.
\end{equation}
Thus, it follows from the well-known Poincare inequality that 
\begin{equation*}
   \Big \|I(t,\cdot)-\frac{1}{|\Omega|}\int_{\Omega}I(t,\cdot)\Big\|_{L^2(\Omega)}\le c_0(|\Omega|)\||\nabla I(t,\cdot)|\|_{L^2(\Omega)}\to 0\quad \text{as}\ t\to\infty.
\end{equation*}
Here, $c_0(|\Omega|)$ is a positive constant. This, along with the regularity theory for parabolic equations, yields 
    \begin{equation}\label{Eqa-2}
   \lim_{t\to\infty}\Big \|I(t,\cdot)-\frac{1}{|\Omega|}\int_{\Omega}I(t,\cdot)\Big\|_{C^1(\bar{\Omega})}= 0.
\end{equation}
Therefore, the $\omega$-limit set 
$$
\omega(I):=\{I^*\in C(\bar\Omega):\ I(t_j, \cdot)\to I^*\ \ \text{for some}\ t_j\to\infty\}
$$ 
consists of nonnegative constants. 

\vskip6pt
\noindent{\bf Case 1:}  $\omega(I)$ consists of at least one positive constant. 

In this case, let $\{t_j\}_{j\ge 1}$ be a sequence with $t_j\to\infty$ such that 
$$
\|I(t_j,\cdot)-c_*\|_{\infty}\to 0\ \ \text{as}\ j\to\infty
$$
for some positive constant  $c_*$.
By \eqref{Eqa-1}, after passing to a subsequence, we have 
$$
\frac{(U(t_j,\cdot)-I(t_j,\cdot))^2I(t_j,\cdot)}{m+S(t_j,\cdot)+I(t_j,\cdot)}\to 0\ \ \text{a.e. \ in}\ \Omega \ \ \text{as}\ j\to\infty.
$$
By \eqref{L-infty-bound-for-I} and \eqref{Eqa-4}, we have $U(t_j,\cdot)\to c_*$ a.e. in $\Omega$ as $j\to\infty$.  Thus 
\begin{equation}\label{S-star-eq}
S(t_j,\cdot)=\frac{\gamma}{\beta-\gamma}(U(t_j,\cdot)+m)
\to \frac{\gamma}{\beta-\gamma}(c_*+m)\ \ \text{a.e. \ in}\ \Omega \ \ \text{as}\ j\to\infty.
\end{equation}
It follows from  \eqref{Total-size identity} that
$$
N=\int_{\Omega}\Big[\frac{\gamma}{\beta-\gamma}(c_*+m)+c_*\Big]
=c_*\int_{\Omega}\frac{\beta}{\beta-\gamma}+\int_{\Omega}\frac{m\gamma}{\beta-\gamma}.
$$
Solving for $c_*$ yields 
 \begin{equation}\label{c-star-equation}
c_*=\frac{N-\int_{\Omega}\frac{m\gamma}{\beta-\gamma}}{\int_{\Omega}\frac{\beta}{\beta-\gamma}}.
\end{equation}
Since the constant $c_*$ in \eqref{c-star-equation} is independent of any sequence, $\omega(I)$ contains only one positive number. Since  $\omega(I)$ is connected, $\omega(I)=\{c_*\}$ and $\|I(t,\cdot)-c^*\|_{\infty}\to 0$ as $t\to\infty$.  Recalling that 
$$
\p_tS=(\beta-\gamma)\Big[\frac{\gamma}{\beta-\gamma}(m+I)- S \Big]\frac{I}{m+S+I},\quad t>0, x\in\bar\Omega,
$$
 by a simple perturbation argument, we can establish that  $$
 \|S(t,\cdot)-\frac{\gamma}{\beta-\gamma}(m+c_*)\|_{\infty}\to 0\ \ \text{as}\ t\to\infty.
 $$
Therefore, (ii) holds in this case.

\vspace{10pt}

\noindent{\bf Case 2:}  $\omega(I)=\{0\}$. 

 In this case, we have  
\begin{equation}\label{claim-1}
    \lim_{t\to\infty}\|I(t,\cdot)\|_{\infty}=0.
\end{equation}
It remains to show the convergence of $S$ and proves (i). Observe that $S(t,x)$ satisfies 
\begin{equation}
\p_tS=(\beta(x)-\gamma(x))\Big[\frac{\gamma(x)(m(x)+I(t,x))}{\beta(x)-\gamma(x)}-S(t,x)\Big]\frac{I(t,x)}{m(x)+S(t,x)+I(t,x)}.
\end{equation}
By the comparison principle for ODEs, for any given $t_0>0$,
$$
S(t,x)\ge \min\Big\{\frac{\gamma(x)m(x)}{\beta(x)-\gamma(x)},S(t_0,x)\Big\}>0,\quad \forall\, t>t_0>0,\  x\in\bar{\Omega}.
$$
Therefore, we have
\begin{equation}\label{sinf}
    \liminf_{t\to\infty}S(t,x)>0, \quad \forall\, x\in M^+. 
\end{equation}

\vskip6pt
For simplicity, let us define 
$$
J:=\int_0^\infty \|I(t,\cdot)\|_\infty dt. 
$$

\noindent{\bf Subcase 2-1:} $J<\infty$. 

By the equation of $S$ in \eqref{ds=0-model},
\begin{equation*}\label{Eqa-7}
\int_0^{\infty}\|\p_tS(t,\cdot)\|_{\infty}dt
\le\int_0^{\infty}\Big\|\gamma-\frac{\beta S}{m+S+I}\Big\|_{\infty}\|I(t,\cdot)\|_{\infty}dt\le (\|\gamma\|_{\infty}+\|\beta\|_{\infty})J<\infty.
\end{equation*}
Hence,  there holds
$$
S(t,\cdot)\to S_0+\int_0^{\infty}\p_tS(t,\cdot)dt=:S^*\  \  \mbox{in} \ C(\bar{\Omega})\ \,\mbox{ as}\ t\to\infty.
$$
Clearly, $S^*\in C^+(\bar\Omega)$. By \eqref{sinf}, we have $S^*>0$ on $M^+$. 

\vskip6pt
\noindent{\bf Subcase 2-2:} $J=\infty$. 

 Fix  $0<\varepsilon\ll 1$. By \eqref{claim-1}, there is $t_{\varepsilon}\gg 1$ such that 
\begin{equation}\label{Eqa-10}
\|I(t,\cdot)\|_{\infty}<\varepsilon,\quad \forall\, t>t_{\varepsilon}.
\end{equation}
By the equation of $S$ in \eqref{ds=0-model},
$$
\p_tS(t,x)\ge \Big[\frac{\gamma(x)(m(x)-\varepsilon)}{\beta(x)-\gamma(x)} -S(t,x)\Big]\frac{(\beta(x)-\gamma(x))I(t,x)}{m(x)+S(t,x)+I(t,x)},\quad \forall\, t>t_{\varepsilon}, \, x\in\bar{\Omega}.
$$
We can now invoke the comparison principle for ODEs to conclude that 
\begin{equation}\label{Eqa-6}
S(t,x)\ge \frac{\gamma(x)(m(x)-\varepsilon)}{\beta(x)-\gamma(x)}
+\Big[S(t_{\varepsilon},x)-\frac{\gamma(x)(m(x)-\varepsilon)}{\beta(x)-\gamma(x)}\Big]e^{-\int_{t_{\varepsilon}}^t
\frac{(\beta(x)-\gamma(x))I(s,x)}{{m(x)+S(s,x)+I(s,x)}}ds}
\end{equation}
for all $t>t_{\varepsilon},\, x\in\bar{\Omega}$. 

Since $J=\infty$,  it follows from \eqref{Harnak-inequality}, \eqref{L-infty-bound-for-I} and \eqref{Eqa-4}  that 
\begin{align}\label{aa-1}
\inf_{x\in\bar{\Omega}}\int_{t_{\varepsilon}}^t\frac{I(s,x)}{{m(x)+S(s,x)+I(s,x)}}ds\ge \frac{1}{(\|m\|_{\infty}+M+\bar{S}^*)c_1}\int_{t_{\varepsilon}}^{t}\|I(s,\cdot)\|_{\infty}ds\to \infty  
\end{align}
as $t\to\infty$. Therefore, by \eqref{Eqa-4} and \eqref{Eqa-6}, we have 
$$
 \liminf_{t\to\infty}S(t,x)\ge  \frac{\gamma(x)(m(x)-\varepsilon)}{\beta(x)-\gamma(x)}\quad \text{uniformly on } \ x\in\bar{\Omega}.
$$
Letting $\varepsilon \to 0^+$, we obtain that 
\begin{equation}\label{Eqa-8}
 \liminf_{t\to\infty}S(t,x)\ge  \frac{\gamma(x)m(x)}{\beta(x)-\gamma(x)}\quad \text{uniformly on}\ x\in \bar{\Omega}.
\end{equation}
Similarly, 
$$
\p_tS(t,x)\le \Big[\frac{\gamma(x)(m(x)+\varepsilon)}{\beta(x)-\gamma(x)} -S(t,x)\Big]\frac{(\beta(x)-\gamma(x))I(t,x)}{m(x)+S(t,x)+I(t,x)},\quad \forall\, t>t_{\varepsilon}.
$$
We can now invoke the comparison principle for ODEs to conclude that
$$
S(t,x)\le \frac{\gamma(x)(m(x)+\varepsilon)}{\beta(x)-\gamma(x)}
+\Big[S(t_{\varepsilon},x)-\frac{\gamma(x)(m(x)+\varepsilon)}{\beta(x)-\gamma(x)}\Big]
e^{-\int_{t_{\varepsilon}}^t\frac{(\beta(x)-\gamma(x))I(s,x)}{{m(x)+S(s,x)+I(s,x)}}ds}
$$
for all $t>t_{\varepsilon},\, x\in\bar{\Omega}$. Using the above inequality, we can show 
$$
\limsup_{t\to\infty}S(t,x)\le \frac{\gamma(x)m(x)}{\beta(x)+\gamma(x)}\quad \text{uniformly on}\ x\in\bar{\Omega}.
$$
Combining this with \eqref{Eqa-8}, we obtain  
$$
\lim_{t\to\infty}S(t,x)=\frac{\gamma(x)m(x)}{\beta(x)-\gamma(x)}=:S^*\quad \text{uniformly on}\ x\in\bar{\Omega}.
$$
This finishes the proof for Subcase 2-2. 

Finally, it is easy to see that if $N\le \int_\Omega \frac{m\gamma}{\beta-\gamma}$, then Case 2 happens and (i) holds. 
\end{proof}

\vskip10pt
In what follows, we provide a proof for Proposition \ref{prop1}.

\begin{proof}[Proof of Proposition \ref{prop1}]  Suppose that $\beta>\gamma$ on $\bar{\Omega}$. Fix $N>\int_{\Omega}\frac{m\gamma}{\beta-\gamma}$ and $d_I>0$. Let $(S,I)$ denote the unique classical solution of \eqref{ds=0-model}, and $U$ be defined by \eqref{U-def}.  It is easily checked that  $(U(t,x),I(t,x))$ satisfies the following cooperative system: 
\begin{equation}\label{coop-system}
\begin{cases}
   \p_t U=(I-U)F_1(t,x), & x\in\bar{\Omega},\ t>0,\cr
  \p_t  I=d_I\Delta I+(U-I)F_2(t,x), & x\in\Omega,\ t>0,\cr
    \partial_{\nu}I=0, & x\in\partial\Omega,\ t>0,
    \end{cases}
\end{equation}
where 
$$F_1(t,x)=\frac{(\beta(x)-\gamma(x))I(t,x)}{m(x)+S(t,x)+I(t,x)}>0,\ \ F_2(t,x)=\frac{\gamma(x) I(t,x)}{m(x)+S(t,x)+I(t,x)}>0
$$ 
for all $t>0$ and $x\in\bar{\Omega}$.

\vskip10pt
{\rm (i)} Suppose  that $S_0\ge \frac{m\gamma}{\beta-\gamma}$. Then, $U(0,x)\ge 0$ for all $x\in\bar{\Omega}$. Since $I(t,\cdot)>0$ on $\bar{\Omega}$ for all $t>0$, it follows from \eqref{Eqa-3} and the comparison principle for ODEs that $U(t,\cdot)>0$ on $\bar{\Omega}$ for all $t>0$.   Let 
$$
c(t)=\min\{\min_{x\in\bar{\Omega}}I(t,x),\ \min_{x\in\bar{\Omega}}U(t,x)\}>0,\quad \forall\, t>0.
$$
For every $t_0>0$, the constant function $(\underline{U}(t,x),\underline{I}(t,x))=(c(t_0),c(t_0))$, $t>0$, solves the cooperative system \eqref{coop-system} and satisfies 
$$
\underline{U}(t_0,x)\le U(t_0,x)\quad \text{and}\quad \underline{I}(t_0,x)\le I(t_0,x),\quad x\in\bar\Omega.
$$
Therefore, by the comparison principle for cooperative systems, we have
$$
c(t_0)=\underline{U}(t,x)\le U(t,x) \quad \text{and}\quad c(t_0)=\underline{I}(t,x)\le I(t,x),\quad \forall\, t>t_0,\ x\in \bar{\Omega}.
$$
Hence, $0<c(t_0)\le c(t)$ for all $t\ge t_0>0$. This shows that 
$$
\inf_{t>t_0}\min_{x\in\bar{\Omega}}I(t,x)\ge c(t_0)>0,\quad \forall\, t_0>0.
$$
So $\liminf_{t\to\infty} I(t, x)>0$ uniformly for $x\in\bar\Omega$. Consequently,   Theorem \ref{th2.1}-{\rm (ii)} holds. 

{\rm (ii)} Suppose that $S_0\le \frac{m\gamma}{\beta-\gamma}$ and $\|I_0\|_{\infty}<\frac{N-\int_{\Omega}\frac{m\gamma}{\beta-\gamma}}{\int_{\Omega}\frac{\gamma}{\beta-\gamma}}$. Define 
$$\hat{U}(t,\cdot)=\|I_0\|_{\infty}-U(t,\cdot)\quad \text{and}\quad \hat{I}(t,\cdot)=\|I_0\|_{\infty}-I(t,\cdot),\quad t\ge 0.
$$
Then $(\hat{U}(t,x),\hat{I}(t,x))$ solves the cooperative system \eqref{coop-system}. Observe also that 
$$
\hat{I}(0,\cdot)\ge 0\quad \text{and}\quad \hat{U}(0,\cdot)=\|I_0\|_{\infty}-\frac{\beta-\gamma}{\gamma}\Big(S_0-\frac{m\gamma}{\beta-\gamma}\Big)>0 .
$$
Therefore, by the comparison principle for cooperative systems, 
$$
\mbox{$\hat{U}(t,\cdot)>0$ and $\hat{I}(t,\cdot)>0$\ for all $t>0$.} 
$$
As a result, we find that
$$
S(t,\cdot)=\frac{\gamma}{\beta-\gamma}[U(t,\cdot)+m]
=\frac{\gamma}{\beta-\gamma}[\|I_0\|_{\infty}-\hat{U}(t,\cdot)+m]
<\frac{\gamma}{\beta-\gamma}(\|I_0\|_{\infty}+m),\quad \forall\ t>0.
$$
This shows that, for all $t>0$, 
\begin{align*}
   \int_{\Omega}I(t,x)=&N-\int_{\Omega}S(t,x)\cr
    >&N-\int_{\Omega}\frac{\gamma}{\beta-\gamma}(m+\|I_0\|_{\infty})\cr
    =& \left(\frac{N-\int_{\Omega}\frac{m\gamma}{\beta-\gamma}}{\int_{\Omega}\frac{\gamma}{\beta-\gamma}}
-\|I_0\|_{\infty}\right)\int_{\Omega}\frac{\gamma}{\beta-\gamma},
\end{align*}
which along with \eqref{Eqa-2} gives that
$$
\liminf_{t\to\infty} I(t, x)\ge \frac{1}{|\Omega|} \left(\frac{N-\int_{\Omega}\frac{m\gamma}{\beta-\gamma}}{\int_{\Omega}\frac{\gamma}{\beta-\gamma}}
-\|I_0\|_{\infty}\right)\int_{\Omega}\frac{\gamma}{\beta-\gamma}>0.
$$
Hence, Theorem \ref{th2.1}-(ii) holds.
\end{proof}

\vskip10pt

Next we will present the proof for Theorem \ref{th2.2}.

\begin{proof}[Proof of Theorem \ref{th2.2}] We first prove the assertion (i). To the end, we pick an interior point $x_0$ of $H^0\cap M^+$. Then there exists a closed ball $B\subset H^0$ centered at $x_0$ such that $m>0$ on $B$.

As in the proof of Lemma \ref{lem1}, the following Harnack's inequality holds:  
\begin{equation}\label{minmax}
\max_{x\in\bar\Omega} I(t,x)\le C \min_{x\in\bar\Omega} I(t,x), \ \ x\in\bar\Omega,\, t\ge 1
\end{equation}
for some constant $C>1$.

Since $\beta\ge\gamma$ on $\bar\Omega$, we have 
$$
\partial_t S=\frac{(-\beta+\gamma)S+\gamma I+\gamma m}{m+S+I}I\le \gamma\frac{I^2+ m I}{m+S+I},\ \ t>0,\,x\in\bar\Omega.
$$ 
Given any $x\in\bar\Omega$,  the comparison principle for parabolic equations ensures that $S(x, t)\le \bar S(t)$ for all $t\ge 1$, where $\bar S$ is the solution of the following ODE problem:
\begin{equation}\label{model-comp}
\begin{cases}
\displaystyle \bar S'(t)=\gamma(x) \frac{I^2(t,x)+m(x)I(t,x)}{m(x)+\bar S+I(t,x)},&t>1,\\
\bar S(1)=S(1,x). &
\end{cases}
\end{equation}

We can show that $KS(x_0, t)$ is an upper solution of \eqref{model-comp} for some large $K>0$.  Indeed, we just need to check
$$
K\partial_tS(t,x_0)=K\gamma(x_0) \frac{I^2(t,x_0)+m(x_0)I(t,x_0)}{m(x_0)+S(t,x_0)+I(t,x_0)}
\ge  \gamma(x) \frac{I^2(t,x)+m(x)I(t,x)}{m(x)+KS(t,x_0)+I(t,x)}.
$$
By \eqref{minmax}, it is sufficient to check the following inequality: 
$$
K\gamma(x_0) \frac{I^2(t,x)/C^2+m(x_0)I(t,x)/C}{m(x_0)+S(t,x_0)+{\frac{I(t,x)}{C}}}
\ge  \gamma(x) \frac{I^2(t,x)+m(x)I(t,x)}{m(x)+KS(t,x_0)+I(t,x)}.
$$
Some elementary computation shows there is a large constant $K>0$ depending on $m|_B$ such that the above inequality holds. It then follows from the parabolic comparison principle that 
$$
KS(t,x_0)\ge \bar S(t)\ge S(t,x),\ \ \forall t\ge 1.
$$ 
By interchanging $x_0$ and $x$, we have 
\begin{equation}\label{xx0}
S(t,x_0)/K\le S(t,x)\le KS(t,x_0), \ \ \ \forall x\in B, \ t\ge 1. 
\end{equation}
Using \eqref{xx0},  we deduce
$$
N\ge \int_\Omega S(t,\cdot)\ge \int_{B} S(t,\cdot)\ge \int_{B} \frac{S(t,x_0)}{K}=\frac{|B| S(t,x_0)}{K}
$$ for all $t\ge 1$. Therefore, we have 
$$
S( t,x_0)\le KN/|B| \ \ \text{and}\ \  S(t,x)\le K^2N/|B|,\ \ \forall x\in\bar\Omega,  t\ge 1. 
$$
This implies that there exists $M>0$ depending on (the $L^\infty$ norm of) initial data such that 
\begin{equation}\label{bound-incidence}
  \|S( t,\cdot)\|_{L^\infty(\Omega)}\le M, \ \ \forall  t\ge 1.
\end{equation}

By \eqref{minmax} and $\int_\Omega I\le N$, there is $M'>0$ such that $0\le I(t,x)\le M'$ for all $x\in\bar\Omega$ and $t\ge 1$. 
Therefore, this and \eqref{bound-incidence} yield
$$
\partial_t S=\frac{\gamma I^2+\gamma m I}{m+S+I}\ge \frac{a }{\|m\|_\infty+M+M'} \min_{y\in B} I(t,y),\ \ \forall x\in B, t>1,
$$
where $a=\min_{B} \gamma m$. 
Integrating the above inequality over $B\times (1, \infty)$ and noticing that $H^0$ has positive measure, by \eqref{minmax}, we have
\begin{equation}\label{Iinfb}
\int_1^\infty \|I( t,\cdot)\|_{\infty}dt<\infty.
\end{equation}
Since  $0\le I(x, t)\le M'$ for all $x\in\bar\Omega$ and $t\ge 1$, by the $L^p$ estimate and the $I$-equation of \eqref{ds=0-model}, $I$ is H\"older continuous on $\bar\Omega\times [1, \infty)$. Thus, it follows from Lemma \ref{lem0} and \eqref{Iinfb} that $\|I( t,\cdot)\|_\infty\to 0$ as $t\to\infty$. 

By the equation of $S$ and \eqref{Iinfb}, we infer
   \begin{align*}
\int_1^\infty\|\p_tS(t,\cdot)\|_\infty dt\le&\int_1^\infty \Big\|\frac{\beta S(t,\cdot) I(t,\cdot)}{m+S(t,\cdot)+I(t,\cdot)}-\gamma I(t,\cdot)\Big\|_\infty dt\cr 
\le&(\|\beta\|_\infty+\|\gamma\|_\infty)\int_1^\infty\|I( t,\cdot)\|_{\infty}dt<\infty. 
\end{align*}  
It follows that $S(t,\cdot)\to S^*:=S(0,\cdot)+\int_0^\infty \p_tS(t,\cdot)dt\in C(\bar\Omega)$ uniformly on $\Omega$ as $t\to\infty$. 

To finish the proof of the assertion (i), it remains to show $S^*>0$ on $M^+\cup H^0$. The proof of $S^*>0$ on $M^+\cap H^+$ is the same as in that of Theorem \ref{th2.1}(i) (using \eqref{sinf}). Now, for any given $x\in H^0$, from the first equation of \eqref{ds=0-model}, we see that
$$
\p_tS(t,x)=\frac{\gamma(x)m(x)+\gamma(x)I(t,x)}{m(x)+S(t,x)+I(t,x)}I(t,x)>0,\ \ \forall t>0.
$$
Thus, $S(t,x)$ is increasing in $t\in(0,\infty)$ for such $x\in H^0$, and hence $S^*>0$ on $H^0$. Therefore,  $S^*>0$ on $H^0\cup (M^+\cap H^+)=M^+\cup H^0$. This completes the proof of the assertion (i).

Using \eqref{ly},  the assertion (ii) can be proved similarly to \cite[Theorem 4.4 (ii)]{salako2024degenerate}, and thus the details are omitted here.
\end{proof}

\vskip10pt
We are now in a position to provide the proof of Theorem \ref{th2.3}.

\begin{proof}[Proof of Theorem \ref{th2.3}]
    Let $x_0\in\bar\Omega$ be such that $\beta(x_0)<\gamma(x_0)$. By the continuity of $\beta$ and $\gamma$, there exists $\epsilon_0>0$ such that $\gamma-\beta\ge \epsilon_0$ on $B$ for some closed ball $B\subset\Omega$.  
    By the first equation of \eqref{ds=0-model}, we have 
    \begin{equation}\label{Slower}
\partial_t S(t, x)\ge (\gamma(x)-\beta(x))I(t, x)\ge\epsilon_0 I(t, x), \ \ x\in B,\, t\ge 0. 
    \end{equation}
By the first equation of \eqref{ds=0-model} again, 
    \begin{equation}\label{Sup}
\partial_t S(t, x)\le \|\gamma\|_\infty I(t, x), \ \ x\in \bar\Omega,\, t\ge 0. 
    \end{equation}

We claim that $\int_1^\infty I(t,x_0)dt<\infty$. Suppose to the contrary that $\int_1^\infty I(t,x_0)dt=\infty$. By Lemma \ref{lem1}, there exists $C>0$ such that
\begin{equation}\label{IC}
\frac{1}{C} \int_1^t I(s,x_0)ds\le \int_1^t I(s, x)ds\le C \int_1^t I(s,x_0)ds, \ \ x\in \bar\Omega,\, t\ge1. 
\end{equation}
On the other hand, integrating \eqref{Slower} over $(1, t)$, we obtain 
\begin{equation}\label{Ses}
    S(t,x)\ge S(1, x)+\epsilon_0 \int_1^t I(s,x)ds\ge S(1, x)+\frac{\epsilon_0}{C} \int_1^t I(s,x_0)ds, \ \ x\in B.  
\end{equation}
Furthermore, integrating \eqref{Ses} over $B$ yields 
$$
\int_B S(t, \cdot)\ge \int_B S(1, \cdot) +\frac{\epsilon_0 |B|}{C}\int_1^t I(s,x_0)ds. 
$$
It follows that 
$$
\lim_{t\to\infty}\int_B S(t,\cdot)\ge \int_B S(1,\cdot) +\frac{\epsilon_0 |B|}{C}\int_1^\infty I(s,x_0)ds=\infty. 
$$
This contradicts with the fact that $\int_\Omega S(t,\cdot) \le N$ for all $t\ge 0$.

Now, integrating \eqref{Sup} over $(1, t)$ and using \eqref{IC}, we find that
$$
S(t, x)\le S(1, x)+\|\gamma\|_\infty \int_1^t I(s, x)ds\le S(1, x)+C\|\gamma\|_\infty \int_1^t I(s,x_0)ds, \ \ \forall t\ge 1, x\in\bar\Omega.
$$
By $\int_1^\infty I(t,x_0)dt<\infty$ and \eqref{IC}, 
\begin{equation}\label{Ib}
\int_1^\infty\|\gamma I(t,\cdot)\|_\infty \le C\|\gamma\|_\infty\int_1^\infty I(t,x_0)<\infty
\end{equation}
and 
$$
\int_1^\infty \Big\|\beta\frac{S(t,\cdot)I(t,\cdot)}{m+S(t,\cdot)+I(t,\cdot)}\Big\|_\infty \le \|\beta\|_\infty \int_1^\infty\|I(t,\cdot)\|_\infty \le C\|\beta\|_\infty \int_1^\infty I(t,x_0)<\infty.
$$
It then follows that 
$$
\int_1^\infty \|\p_tS(t,\cdot)\|_\infty =\int_1^\infty \Big\|\gamma I(t,\cdot)-\beta\frac{S(t,\cdot)I(t,\cdot)}{m+S(t,\cdot)+I(t,\cdot)}\Big\|_\infty <\infty. 
$$
Therefore,
$$
S(t,\cdot)\to S(1,\cdot)+\int_1^\infty \p_tS(t,\cdot):=S^*\in C(\bar\Omega)
$$
uniformly on $\bar\Omega$ as $t\to\infty$. By \eqref{Ib} and the H\"older continuity of $[1, \infty)\ni t\mapsto \|I(\cdot, t)\|_\infty$, Lemma \ref{lem0} results in $I(t,\cdot)\to 0$ in $C(\bar\Omega)$ as $t\to\infty$. Finally, by \eqref{Total-size identity}, we must have $\int_\Omega S^*dx=N$. 

It remains to verify $S^*>0$ on $M^+\cup H^0\cup H^-$. The proof of $S^*>0$ on $(M^+\cap H^+)\cup H^0$ is the same as in that of Theorem \ref{th2.1}(i) and Theorem \ref{th2.2}(i). For any given $x\in H^-$, from the first equation of \eqref{ds=0-model}, it follows that
$$
\p_tS(t,x)=\frac{\gamma(x)m(x)+\gamma(x)I(t,x)+(\gamma(x)-\beta(x))S(t,x)}{m(x)+S(t,x)+I(t,x)}I(t,x)>0,\ \ \forall t>0,
$$
which implies that $S(t,x)$ is increasing in $t\in(0,\infty)$ for such $x\in H^-$, and hence $S^*>0$ on $H^-$. The proof is completed.
\end{proof}

In the sequel, we  verify Proposition \ref{prop2}. 
\begin{proof}[Proof of Proposition \ref{prop2}] Let  $\varphi>0$ be an eigenfunction associated with $\lambda_0$. Then the standard regularity theory of elliptic equations guarantees 
$\varphi\in C^{2}({\mathcal{O}})\cap C(\bar{\mathcal{O}})$. 
Suppose that $(S,I)$ is a solution of \eqref{ds=0-model} with  initial data  $(S_0,I_0)\in[C^{+}(\Omega)]^2$ with $I_0\ge,\ne 0$ such that $\|I(t,\cdot)\|_{\infty}\to 0$ and $\|S(t,\cdot)-S^*\|_{\infty}\to 0$ as $t\to\infty$ for some  $S^*\in C^+(\bar{\Omega})$.
By Hopf's boundary lemma,  $\partial_{\nu}\varphi<0$ on $\partial\mathcal{O}$. Since $I(t,\cdot)>0$ on $\bar{\mathcal{O}}$ for all $t>0$,  we have  
\begin{align*}
    \frac{d}{dt}\int_{\mathcal{O}}{\varphi}I=&d_I\int_{\mathcal{O}}{\varphi}\Delta I+\int_{\mathcal{O}}\Big(\frac{\beta S}{S+I}-\gamma\Big){\varphi}I\cr
    =&d_I\left(\int_{\mathcal{O}}I\Delta {\varphi}+\int_{\partial\mathcal{O}}\varphi\partial_{\nu}I-\int_{\partial\mathcal{O}}I\partial_{\nu}{\varphi}\right) +\int_{\mathcal{O}}\Big(\frac{\beta S}{S+I}-\gamma\Big){\varphi}I\cr
    >& d_I\int_{\mathcal{O}}I\Delta\varphi +\int_{\mathcal{O}}\Big(\frac{\beta S}{S+I}-\gamma\Big)\varphi I \cr
    = &{\lambda}_0\int_{\mathcal{O}}{\varphi}I-\int_{\mathcal{O}}\frac{\beta I}{S+I}{\varphi}I\cr
    = &\Big({\lambda}_0-\int_{\mathcal{O}}\frac{\beta I}{S+I}\frac{{\varphi}I}{\int_{\mathcal{O}}{\varphi}I}\Big)\int_{\mathcal{O}}{\varphi}I,\quad t>0.
\end{align*}
By the comparison principle for ODEs,
$$
\int_{\mathcal{O}}\varphi(x) I(t,x)dx\ge\int_{\mathcal{O}}\varphi I(1,x)e^{\int_1^t\Big[{\lambda}_0-\int_{\mathcal{O}}\frac{\beta I(s,y)}{S(s,y)+I(s,y)}\frac{{\varphi}I(s,y)}{\int_{\mathcal{O}}{\varphi}(y)I(s,y)dy}dy\Big]ds}dx,\quad t\ge 1.
$$
Since $\int_{\mathcal{O}}I(1,\cdot)\varphi>0$, we  divide both sides of the above inequality by $\int_{\mathcal{O}}I(1,\cdot)\varphi$ and then take $\log$ to obtain
\begin{equation}\label{Eqb-7}
\frac{1}{t}\ln\Big(\frac{\int_{\mathcal{O}}{\varphi}I(t,x)}{\int_{\mathcal{O}}{\varphi}I(1,x)}\Big)\ge\frac{{\lambda}_0(t-1)}{t}- \frac{1}{t}\int_1^t\int_{\mathcal{O}}\frac{\beta I(s,x)}{S(s,x+I(s,x)}\frac{{\varphi}I(s,x)}{\int_{\mathcal{O}}{\varphi}(y)I(s,y)dy}dxds,\quad t>1.
\end{equation}
Define
$$
G(t,x):=\frac{1}{t}\int_1^t\frac{\beta(x) I(s,x)}{S(s,x)+I(s,x)}\frac{{\varphi}(x)I(s,x)}{\int_{\mathcal{O}}{\varphi}(y)I(s,y)}ds,\quad x\in\mathcal{O}, \, t\ge 1.
$$
Since $\|I(t,\cdot)\|_{\infty}\to 0$ as $t\to\infty$, we deduce from \eqref{Eqb-7} that 
\begin{align}\label{Eqb-4}
\liminf_{t\to\infty}\int_{\mathcal{O}}G(t,x)=&  \liminf_{t\to\infty}\frac{1}{t}\int_1^t\int_{\mathcal{O}}\frac{\beta I}{S+I}\frac{{\varphi}I}{\int_{\mathcal{O}}{\varphi}I}\cr 
  \ge&\liminf_{t\to\infty}\Big[\frac{{\lambda}_0(t-1)}{t}-\frac{1}{t}\ln\Big(\frac{\int_{\mathcal{O}}{\varphi}I(t,x)}
  {\int_{\mathcal{O}}{\varphi}I(1,x)}\Big)\Big]\cr
  \geq&{\lambda}_0>0.
\end{align}

On the other hand, by \eqref{Harnak-inequality} and the definition of $G$, 
\begin{align}\label{bound-G}
0< G(t,x)\le&\frac{\|\beta\|_{\infty}}{t}\int_1^t\frac{c_{1}{\varphi}(x)\min_{z\in\bar{\Omega}}I(s,z)}{(\int_{{\mathcal{O}}}{\varphi}(y)dy)\min_{z\in\bar{\Omega}}I(s,z)}ds 
\le  \frac{c_1\|\beta\|_{\infty}\|{\varphi}\|_{\infty}}{\int_{\mathcal{O}}{\varphi}},\quad \forall\, x\in{\mathcal{O}},\, t>1.
\end{align}
Suppose that $S^*(x_0)>0$ for some $x_0\in\mathcal{O}$. Then, we  have
$$
\inf_{t\ge 1}S(t,x_0)>0.
$$
By $\|I(t,\cdot)\|_{\infty}\to 0$ as $ t \to\infty$ and  \eqref{Harnak-inequality}, we have
\begin{align}\label{conv-zero}
0<G(t,x_0) =&\frac{1}{t}\int_1^t\Big[\frac{\beta(x_0) }{S(s,x_0)+I(s,x_0)}\Big]\Big[\frac{{\varphi}(x_0)I(s,x_0)}{\int_{\mathcal{O}}{\varphi}(y)I(s,y)dy}\Big]I(s,x_0)ds\cr 
 \le& \left(\frac{\|\beta\|_{\infty}}{\inf_{s\ge 1}S(s,x_0)}\right)\left(\frac{c_1\|{\varphi}\|_{\infty}}{\int_{\mathcal{O}}{\varphi}}\right)\left(\frac{1}{t}\int_{1}^tI(s,x_0)ds\right)\to 0
\end{align}
 as $ t \to\infty$. 

Hence, if it was the case that $|\{x\in\mathcal{O} : S^*(x)=0 \}|=0$,  in view of \eqref{bound-G} and \eqref{conv-zero}, we would derive from the Lebesgue Dominated Convergence Theorem that 
$$
\lim_{t\to\infty}\int_{\mathcal{O}}G(t,\cdot)=0,
$$
which  contradicts  \eqref{Eqb-4}. Therefore, we must have $ |\{x\in\mathcal{O} : S^*(x)=0 \}|>0$.
\end{proof}

With the aid of Proposition \ref{prop2}, we can readily prove Corollary \ref{cor1}.
\begin{proof}[Proof of Corollary \ref{cor1}]
By Theorems \ref{th2.1}, \ref{th2.2} and \ref{th2.3}, the  solution \((S, I)\) of \eqref{ds=0-model} satisfies \(\|I(t, \cdot)\|_{\infty} \to 0\) and \(\|S(t, \cdot) - S^*\|_{\infty} \to 0\) as \(t \to \infty\) for some \(S^* \in C^+(\bar{\Omega})\) under one of the assumptions (i)-(iii). 

Since $M^0\cap H^+$ has non-empty interior, there is an open set $\mathcal{O}\subseteq M^0\cap H^+$ with a Lipschitz boundary.
Let \(\lambda_0\) be defined in Proposition \ref{prop2}. It is well known that $\lambda_0$ is  decreasing in \(d_I \in (0, \infty)\) with
\[
\lim_{d_I \to 0} \lambda_0 = \max_{x \in \bar{\mathcal{O}}} (\beta(x) - \gamma(x))>0 \ \ \text{and} \quad \lim_{d_I \to \infty} \lambda_0 = -\infty.
\]
 Thus,  there is a critical value \(d_0^* > 0\) such that \(\lambda_0 > 0\) for \(0 < d_I < d_0^*\) and \(\lambda_0 \leq 0\) for \(d_I \geq d_0^*\).
Then the existence of $\Omega^*\subset M^0\cap H^+$  follows from Proposition \ref{prop2}.
\end{proof}

% \vskip10pt
% To conclude this section, we now verify Corollary \ref{cor2}.
% \begin{proof}[Proof of Corollary \ref{cor2}]
% In Theorems \ref{th2.2} and \ref{th2.3}, it has been shown that the positive solution \((S, I)\) of \eqref{ds=0-model} satisfies \(\|I(t, \cdot)\|_{\infty} \to 0\) and \(\|S(t, \cdot) - S^*\|_{\infty} \to 0\) as \(t \to \infty\), for some \(S^* \in C^+(\bar{\Omega})\) if either \(H^0\) or \(H^-\) has a nonempty interior. Hence, one can follow the same proof as in Corollary \ref{cor1} to conclude Corollary \ref{cor2}.
% \end{proof}

\section{Proof of main results for system \eqref{dI=0-model}}
In this section, we are concerned with system \eqref{dI=0-model}, and aim to prove Theorems \ref{th2.4}-\ref{th2.6}, Proposition \ref{prop3} and Corollary \ref{coro2.3}. We begin with the following lemma.
  \begin{lem}\label{lem4.1} Let $(S,I)$ be the solution of \eqref{dI=0-model}. Then we have
 \begin{equation}\label{Eqd-1}
     \liminf_{t\to\infty}\int_{\Omega}S(t,\cdot)\ge Nl_0, 
 \end{equation}
 where 
 $$
\gamma_{\min}:=\min_{x\in\bar\Omega}\gamma(x),\ \ \ l_0:=\frac{\gamma_{\min}}{\gamma_{\min}+\|\beta\|_{\infty}}.
 $$ 
 \end{lem}
 \begin{proof} In view of
 \begin{equation}\label{total-population}
 \int_{\Omega} (S (t,\cdot)+ I(t,\cdot)) = N, \ \  \ \forall\, t \ge 0,
 \end{equation}
 we observe that 
   \begin{align*}
\frac{d}{dt}\int_{\Omega}S(t,\cdot)=&\int_{\Omega}\Big[\gamma -\frac{\beta S(t,\cdot)}{m+S(t,\cdot)+I(t,\cdot)}\Big]I(t,\cdot)\cr 
 \ge&\ \gamma_{\min}\int_{\Omega}I(t,\cdot)-\|\beta\|_{\infty}\int_{\Omega}S(t,\cdot)\cr
 =&\ \gamma_{\min}N-(\gamma_{\min}+\|\beta\|_{\infty})\int_{\Omega}S(t,\cdot), \quad \forall \,t>0.
\end{align*}  
Therefore, the desired result follows by a simple comparison principle for ODEs.
\end{proof}
 
 The next result provides the uniform persistence property of the susceptible population.
 \begin{prop}\label{prop3.1}  Let $(S,I)$ be the solution of \eqref{dI=0-model}. Then there exists a positive number $l_1>0$  independent of initial data, $m$ and $N$, such that 
 \begin{equation}\label{Eqd-3}
\liminf_{t\to\infty}\min_{x\in\bar{\Omega}}S(t,x)\ge Nl_1.
 \end{equation}
  \end{prop}
 \begin{proof}Let $l_0>0$ be given by \eqref{Eqd-1}. Observe from  \eqref{dI=0-model} that 
 $$
\p_tS\ge d_S\Delta S-\|\beta\|_{\infty}S,\ x\in\Omega,\ t>0; \quad \partial_{\nu}S=0,\ x\in\partial\Omega,\ t>0.
 $$
 Denoting by $\{e^{t\Delta}\}_{t\ge 0}$, the analytical $C_0$-semigroup on $L^p(\Omega)$ $(p\ge1)$, generated by the Laplace operator $\Delta$ subject to the homogeneous Neumann boundary conditions on $\partial\Omega$. Then appealing to the comparison principle for parabolic equations, we can infer that
 \begin{equation}\label{Eqd-2}
     S(t+1,\cdot)\ge e^{-\|\beta\|_{\infty}} e^{d_S\Delta}S(t,\cdot), \quad\, t>0.
 \end{equation}
 By the Harnack's inequality for parabolic equations \cite{huska2006harnack}, there exists   $M_1>0$ such that 
 \begin{equation*}
     \big\|e^{t\Delta}u\big\|_{\infty}\le M_1 \min_{x\in\bar{\Omega}}(e^{t\Delta}u)(x),\quad \forall\,t\ge d_S,\ \forall \,u\in C^+(\bar{\Omega}).
 \end{equation*}
 It then follows from \eqref{Eqd-2} that 
 \begin{align*}
 S(t+1,\cdot)\ge& \frac{e^{-\|\beta\|_{\infty}}}{M_1} \|e^{d_S\Delta}S(t,\cdot)\|_{\infty}\cr
 \ge&\frac{e^{-\|\beta\|_{\infty}}}{|\Omega|M_1} \int_{\Omega}e^{d_S\Delta}S(t,\cdot) \cr
 =&\frac{e^{-\|\beta\|_{\infty}}}{|\Omega|M_1} \int_{\Omega}S(t,\cdot), \quad t>0.
 \end{align*}
 Therefore,
$$\liminf_{t\to\infty}\min_{x\in\bar{\Omega}}S(t,x)\ge \frac{e^{-\|\beta\|_{\infty}}}{|\Omega|M_1}\liminf_{t\to\infty}\int_{\Omega}S(t,\cdot)\ge \frac{Nl_0e^{-\|\beta\|_{\infty}}}{|\Omega|M_1}=:Nl_1,$$
where $l_0>0$ is given as in \eqref{Eqd-1}.
\end{proof}

 We then establish the eventual uniform boundedness of solutions of \eqref{dI=0-model}.

 \begin{prop}\label{prop3.2} Let $(S,I)$ be the solution of \eqref{dI=0-model}. Then there exists a positive number $l_2>0$, independent of initial data and $m$,  such that 
 \begin{equation}\label{eventual-bonds}
     \limsup_{t\to\infty}(\|S(t,\cdot)\|_{\infty}+\|I(t,\cdot)\|_{\infty})\le l_2.
 \end{equation}
 \end{prop}
\begin{proof} By \eqref{total-population}, the  quantity
\begin{align*}
p_{\infty}:=\sup\Big\{ p\in[1,\infty) & :\ \text{there is }\, N_p>0, \, \text{ independent of initial data and }  m,\, \text{ such that}\, \cr
&\ \ \ \limsup_{t\to\infty}(\|S(t,\cdot)\|_{L^p(\Omega)}+\|I(t,\cdot)\|_{L^p(\Omega)})\le N_p \Big\}
\end{align*}
is well defined. 

We now claim that $p_{\infty}=\infty$. Indeed, suppose to the contrary that $p_{\infty}<\infty$. Set 
\begin{equation}\label{choice-p}
p:=\max\big\{p_{\infty}-\frac{1}{2},\ 1\big\}.
\end{equation}
By the definition of $p_{\infty}$, we can find $N_p$, independent of initial data, and $t_p\gg 1$, such that 
\begin{equation}\label{Np}
\|S(t,\cdot)\|_{L^p(\Omega)}+\|I(t,\cdot)\|_{L^p(\Omega)} \le 2N_p,\quad \forall\, t\ge t_p.
\end{equation}
 Multiplying  the equation of $S(t,x)$ by $(p+1)S^p$, integrating by parts and using Young's inequality, we obtain that
\begin{align}\label{xxb-01}
    \frac{d}{dt}\int_{\Omega}S^{p+1}=& -\frac{4pd_S}{(p+1)}\int_{\Omega}|\nabla S^{\frac{p+1}{2}}|^2 +(p+1)\int_{\Omega}\gamma S^{p}I-(p+1)\int_{\Omega}\beta \frac{S}{S+I+m}IS^{p}\cr
    \le& -\frac{4pd_S}{(p+1)}\int_{\Omega}|\nabla S^{\frac{p+1}{2}}|^2 +\|\gamma\|_{\infty}\left[ p\varepsilon\int_{\Omega}S^{p+1}+\frac{1}{\varepsilon^p}\int_{\Omega}I^{p+1}\right],\ \ t>0,
\end{align}
 for every $\varepsilon>0$. 
 
 Similarly, from the equation of $I$, we get
\begin{align}\label{xxb-02}
    \frac{d}{dt}\int_{\Omega}I^{p+1}=& (p+1)\int_{\Omega}\beta\frac{I}{m+S+I}SI^p-(p+1)\int_{\Omega}\gamma I^{p+1}\cr
    \le & (p+1)\|\beta\|_{\infty}\int_{\Omega}SI^p -(p+1){\gamma}_{\min}\int_{\Omega}I^{p+1}\cr
    \le & \|\beta\|_{\infty}\left[p\tilde{\varepsilon}\int_{\Omega}I^{p+1}+\frac{1}{\tilde{\varepsilon}^p}\int_{\Omega}S^{p+1} \right] -(p+1){\gamma}_{\min}\int_{\Omega}I^{p+1}\cr
    = &\frac{ \|\beta\|_{\infty}}{\tilde{\varepsilon}^p}\int_{\Omega}S^{p+1}-(p+1)\Big({\gamma}_{\min}-\frac{p\tilde{\varepsilon}\|\beta\|_{\infty}}{p+1} \Big)\int_{\Omega}I^{p+1},\ \ t>0,
\end{align}
for every $\tilde{\varepsilon}>0$. 

Hence,  a combination of \eqref{xxb-01} and \eqref{xxb-02} yields
\begin{align*}
    \frac{1}{dt}\left[\int_{\Omega}( S^{p+1}+I^{p+1}) \right]\le & -\frac{4d_Sp}{p+1}\int_{\Omega}|\nabla S^{\frac{p+1}{2}}|^2 +\Big(p\varepsilon\|\gamma\|_{\infty}+\frac{\|\beta\|_{\infty}}{\tilde{\varepsilon}^p}\Big)\int_{\Omega}S^{p+1}\cr
    & -(p+1)\Big({\gamma}_{\min}-\frac{p\tilde{\varepsilon}\|\beta\|_{\infty}+\|\gamma\|_{\infty}\varepsilon^{-p}}{p+1}\Big)\int_{\Omega}I^{p+1}
\end{align*}
for all $\varepsilon,\tilde{\varepsilon}>0$ and  $t>0$.
By taking 
$$ 
\tilde{\varepsilon}_p:=\tilde{\varepsilon}=\frac{(p+1){\gamma}_{\min}}{4p\|\beta\|_{\infty}} \quad \text{and}\quad \varepsilon_p:=\varepsilon=\left[\frac{4\|\gamma\|_{\infty}}{(p+1){\gamma_{\min}}}\right]^{\frac{1}{p}},
$$
the last inequality becomes
\begin{align}\label{xxb-03}
    &\frac{1}{dt}\left[\int_{\Omega}( S^{p+1}+I^{p+1}) \right]\cr
    \le & -\frac{4pd_S}{p+1}\int_{\Omega}|\nabla S^{\frac{p+1}{2}}|^2 +\Big(p\varepsilon_p\|\gamma\|_{\infty}+\frac{\|\beta\|_{\infty}}{\tilde{\varepsilon}_p^p}\Big)\int_{\Omega}S^{p+1} -\frac{(p+1)\gamma_{\min}}{2}\int_{\Omega}I^{p+1}\cr
    =& -a_{1,p}\int_{\Omega}|\nabla S^{\frac{p+1}{2}}|^2 +a_{2,p}\int_{\Omega}S^{p+1} -a_{3,p}\int_{\Omega}I^{p+1},\ \ t>0,
\end{align}
where 
$$
a_{1,p}=\frac{4pd_S}{p+1},\ \ a_{2,p}=\big(p\varepsilon_p\|\gamma\|_{\infty}+\frac{\|\beta\|_{\infty}}{\tilde{\varepsilon}_p^p}\big)\ \  \text{and} \ a_{3,p}=\frac{(p+1){\gamma_{\min}}}{2}.
$$

Recall the well-known Gagliardo-Nirenberg inequality: there a positive constant $C_{p}$  such that 
\begin{equation}\label{xxb-04}
    \|u\|_{L^2(\Omega)}\leq C_{p}\||\nabla u|\|_{L^{2}(\Omega)}^{\theta_p}\|u\|_{L^{\frac{2p}{p+1}}(\Omega)}^{1-\theta_p}+C_{p}\|u\|_{L^{\frac{2}{p+1}}(\Omega)},\ \ \forall u\in W^{1,2}(\Omega),
\end{equation}
where $\theta_p=\frac{n}{n+2p}$.
Substituting $u=S^{\frac{p+1}{2}}$ in \eqref{xxb-04}, and using again Young's inequality and \eqref{Np}, we get for $t\ge t_p$,
\begin{align}\label{xxb-06}
    \int_{\Omega}S^{p+1}\le & \left[ C_{p}\||\nabla S^{\frac{p+1}{2}}|\|_{L^2(\Omega)}^{\frac{\theta_p}{2}}\|S\|_{L^p(\Omega)}^{\frac{(1-\theta_p)(p+1)}{2}}+C_{p}\|S\|_{L^{1}(\Omega)}^{\frac{p+1}{2}} \right]^2\cr
    \le & 2C_{p}^2\||\nabla S^{\frac{p+1}{2}}|\|_{L^2(\Omega)}^{\theta_p}\|S\|_{L^p(\Omega)}^{(1-\theta_p)(p+1)}+2C_{p}^2\|S\|_{L^{1}(\Omega)}^{p+1}\cr
    \le & \hat{\varepsilon}\||\nabla S^{\frac{p+1}{2}}|\|_{L^2(\Omega)}^{2} +\frac{(2-\theta_p)}{2}\left[2C_{p}^2\Big(\frac{\theta_p}{2\hat{\varepsilon}}\Big)^{\frac{2}{\theta_p}}\right]^{\frac{2}{2-\theta_p}}\|S\|_{L^p(\Omega)}^{\frac{2(1-\theta_p)(p+1)}{2-\theta_p}}+2C_{p}^2N^{p+1}\cr
    \leq & \hat{\varepsilon}\||\nabla S^{\frac{p+1}{2}}|\|_{L^2(\Omega)}^{2} +b_{\hat{\varepsilon},p}, \ \ t>0,
\end{align}
for every $\hat{\varepsilon}>0$, where $$
b_{\hat{\varepsilon},p}:=\frac{(2-\theta_p)}{2}\left[2C_{p}^2
\Big(\frac{\theta_p}{2\hat{\varepsilon}}\Big)^{\frac{2}{\theta_p}}\right]^{\frac{2}{2-\theta_p}}(2N_p)^{\frac{2(1-\theta_p)(p+1)}{2-\theta_p}}
+2C_{p}^2N^{p+1}.
$$ 
Multiplying \eqref{xxb-06} by $\frac{a_{1,p}}{\hat{\varepsilon}}$ and rearranging the terms yield
\begin{equation*}
    -a_{1,p}\int_{\Omega}|\nabla S^{\frac{p+1}{2}}|^2\leq -\frac{a_{1,p}}{\hat{\varepsilon}}\int_{\Omega}S^{p+1} +\frac{a_{1,p}b_{\hat{\varepsilon},p}}{\hat{\varepsilon}}, \ \ t>0.
\end{equation*}
As a result, we obtain from \eqref{xxb-03} that 
\begin{equation*}
    \frac{1}{dt}\left[\int_{\Omega}(S^{p+1}+I^{p+1})\right]\leq -\frac{a_{2,p}}{\hat{\varepsilon}}\big( \frac{a_{1,p}}{a_{2,p}}-\hat{\varepsilon}\big)\int_{\Omega}S^{p+1}-a_{3,p}\int_{\Omega}I^{p+1} +\frac{a_{1,p}b_{\hat{\varepsilon},p}}{\hat{\varepsilon}}, \ \ t>0.
\end{equation*}
Finally, by taking $\hat{\varepsilon}=\frac{a_{1,p}}{2a_{2,p}}$, it follows from the last inequality that
\begin{equation*}
    \frac{1}{dt}\left[\int_{\Omega}(S^{p+1}+I^{p+1})\right]\leq -\min\Big\{a_{2,p},a_{3,p}\Big\}\left[\int_{\Omega}S^{p+1}+I^{p+1}\right] +2a_{2,p}b_{\hat{\varepsilon}_p,p}, \ \ t>0,
\end{equation*}
which implies that 
$$ 
\limsup_{t\to\infty}(\|S(t)\|_{L^{p+1}(\Omega)}^{p+1}+\|I(t)\|_{L^{p+1}(\Omega)}^{p+1})\leq  \frac{2a_{2,p}b_{\hat{\varepsilon}_p,p}}{\min\{a_{2,p},a_{3,p}\}}.
$$
%Therefore $N_{p+1}<\infty$, and  hence $N_{p_{\infty}}<\infty$ since $p_{\infty}<p+1$ (due to \eqref{choice-p}) and $|\Omega|<\infty$. This clearly contradicts the definition on $p_{\infty}$. So we must have $p_{\infty}=\infty$.
Recalling the definition of $p_\infty$, the above inequality implies that $p_\infty\ge p+1\ge p_\infty-\frac{1}{2}+1$, which is a contradiction. So we must have $p_\infty=\infty$. 

Next,  we  proceed to establish the eventually uniform boundedness of $S(t,\cdot)$ in $C(\bar{\Omega})$. First, fix $p>n$ and let $N_p>0$, independent of initial data and $m$, such that 
$$
\limsup_{t\to\infty}(\|S(t,\cdot)\|_{L^p(\Omega)}+\|I(t,\cdot)\|_{L^p(\Omega)})\le N_p.
$$
 By the $L_p$--$L_q$-estimates for the heat semigroup $\{e^{t\Delta}\}_{t\ge 0}$ subject to the homogeneous Neumann boundary conditions, there is a positive number $C>0$ such that 
$$
\|e^{t\Delta }u\|_{L^{\infty}(\Omega)}\le Ct^{-\frac{n}{2p}}\|u\|_{L^p(\Omega)},\quad \forall\, 0<t\le 1,\,u\in L^p(\Omega).
$$
Therefore, applying the variation of constant formula, from the first equation of \eqref{dI=0-model} and the positivity of $\{e^{t\Delta}\}_{t\ge1}$, we deduce 
\begin{align*}
   0<  S(1+t,\cdot)=& e^{d_S\Delta}S(t,\cdot)+\int_{0}^1e^{(1-s)d_S\Delta}\Big[\gamma I(t+s,\cdot)-\frac{\beta S(t+s,\cdot)I(t+s,\cdot)}{m+S(t+s,\cdot)+I(t+s,\cdot)}\Big]ds
    \cr
    \le & e^{d_S\Delta}S(t,\cdot)+\int_{0}^1e^{(1-s)d_S\Delta}(\gamma I(t+s,\cdot))ds\cr
    \le & C\|S(t,\cdot)\|_{L^p(\Omega)}+C\|\gamma\|_{\infty}d_S^{-\frac{n}{2p}}\int_{0}^1{(1-s)}^{-\frac{n}{2p}}\|I(t+s,\cdot)\|_{L^p(\Omega)}ds,\ \ t>0.
\end{align*}
As a result,
\begin{equation*}
    \limsup_{t\to\infty}\|S(t,\cdot)\|_{\infty}\le C\Big[1+\|\gamma\|_{\infty}d_S^{-\frac{n}{2p}}\int_0^1(1-s)^{-\frac{n}{2p}}ds\Big]N_p=:N_\infty<\infty.
\end{equation*}
So there is $t_0>0$ such that 
$$
\|S(t,\cdot)\|_{\infty}\le 2N_\infty,\quad \forall\, t\ge t_0.
$$
Thus, we have 
\begin{align*}
    \p_tI=&\Big(\frac{\beta S}{m+S+I}-\gamma\Big)I\cr
    \le& \Big(\frac{2\beta N_\infty}{2N_\infty+I}-\gamma\Big)I \cr
    =& [2N_\infty(\beta-\gamma)-\gamma I]\frac{I}{2N_\infty+I}\cr
    \le & (2N_\infty\|\beta\|_{\infty}-\gamma_{\min} I)\frac{I}{2N_\infty+I},\quad \forall\, t\ge t_0.
\end{align*}
It follows from the comparison principle for ODEs that 
$$
\limsup_{t\to\infty}\|I(t,\cdot)\|_{\infty}\le \frac{2N_\infty\|\beta\|_{\infty}}{\gamma_{\min}}.
$$
Therefore, \eqref{eventual-bonds} holds with $l_2:=\max\{N_\infty, \frac{2N_\infty\|\beta\|_{\infty}}{\gamma_{\min}}\}$, which is independent of initial data and the saturation function $m$. 
\end{proof}

\vskip10pt
Theorem \ref{th2.4} is a combination of Propositions \ref{prop3.1} and \ref{prop3.2}. We now present the proof of Theorem \ref{th2.5}.

\begin{proof}[Proof of Theorem \ref{th2.5}] First, for fixed $x\in\bar{\Omega}$ such that $m(x)=0$, we have  
$$
\p_tI(t,x)=\Big[\frac{\beta(x)S}{S+I}-\gamma(x)\Big]I(t,x)
=\Big[\frac{(\beta(x)-\gamma({x}))}{\gamma(x)}S(t,x)-I(t,x)\Big]\frac{\gamma(x)I(t,x)}{S(t,x)+I(t,x)},\quad \forall\, t>0.
$$
It then follows from \eqref{Eqd-3} and the comparison principle for ODEs that 
$$
\liminf_{t\to\infty}I(t,x)\ge \frac{(\beta(x)-\gamma(x))l_1}{\gamma(x)}>0,
$$
where $l_1>0$ is as in \eqref{Eqd-3}. This proves \eqref{th2.5-eq2}. 

We introduce the transformation:
$$
Z(t,x)=\frac{\gamma(x)(m(x)+I(t,x))}{\beta(x)-\gamma(x)},\quad t\ge 0,\, x\in \bar{\Omega}.
$$
Then, according to \eqref{dI=0-model}, $(S,Z)$ satisfies the cooperative system  
\begin{equation}
\begin{cases}
S_t=d_S\Delta S+\frac{\beta-\gamma}{\gamma}(Z-S)G(t,x), & x\in\Omega, t>0,\cr
Z_t=(S-Z)G(t,x), & x\in\bar{\Omega},\ \forall\ t>0,\cr
\partial_{\nu}S=0, & x\in\partial\Omega,\ t>0,
\end{cases}
\end{equation}
where 
$$
G(t,x)=\frac{\gamma(x) I(t,x)}{m(x)+S(t,x)+I(t,x)}>0,\  \ t>0,\,x\in\bar\Omega.
$$
One can check that
\begin{align*}
    \frac{d}{dt}\Big[\frac{1}{2}\int_{\Omega}S^2+\int_{\Omega}\frac{(\beta-\gamma)Z^2}{2\gamma}\Big]=-d_S\int_{\Omega}|\nabla S|^2-\int_{\Omega}(\beta-\gamma)\frac{(S-Z)^2I}{m+S+I},\quad t>0.
\end{align*}
Due to Proposition \ref{prop3.2}, an integration of the last equation on $(0, \infty)$ gives
\begin{equation}\label{Eqd-5}
\int_{0}^{\infty}\Bigg[d_S\int_{\Omega}|\nabla S|^2+\int_{\Omega}(\beta-\gamma)\frac{(S-Z)^2I}{m+S+I}\Bigg]<\infty.
\end{equation}
By Proposition \ref{prop3.2}, we have $\sup_{t\ge 0}\|(\beta S(t,\cdot)/(m+S+I)-\gamma)I\|_{\infty}<\infty$. By the regularity theory for parabolic equations  applied to the equation of $S$ in \eqref{dI=0-model}, the mapping $[1,\infty)\ni t\mapsto S(t,\cdot)\in C^{1}(\bar{\Omega})$ is H\"older continuous. So, $\|\nabla S(t,\cdot)\|_{L^2(\Omega)}$ is H\"older's continuous on $[1,\infty)$.   In view of these, we can conclude from \eqref{Eqd-5} that 
\begin{equation}\label{Eqd-7}
    \lim_{t\to\infty}\|\nabla S(t,\cdot)\|_{L^2(\Omega)}=0.
\end{equation}
Thanks to \eqref{Eqd-7}, the Poincaré inequality, and the regularity of \( S(t, \cdot) \), we can proceed in a similar manner to prove \eqref{Eqa-2} and establish that
\begin{equation}\label{Eqd-9}
    \lim_{t\to\infty}\Big\|S(t,\cdot)-\frac{1}{|\Omega|}\int_{\Omega}S(t,\cdot)\Big\|_{C^1(\bar{\Omega})}=0,
\end{equation}
 which proves \eqref{th2.5-eq1}.
 
Define the function $H: \Omega\to (0,\infty]$  by
$$
H(x):=\int_{0}^{\infty}\frac{[\beta(x)-\gamma(x)][S(t,x)-Z(t,x)]^2I(t,x)}{m(x)+S(t,x)+I(t,x)},\quad x\in\Omega.
$$
It follows  from \eqref{Eqd-5} and Fubini's Theorem  that 
$$
\int_{\Omega}H=\int_{\Omega}\int_0^{\infty}(\beta-\gamma)\frac{(S-Z)^2I}{m+S+I}
=\int_0^{\infty}\int_{\Omega}\frac{(\beta-\gamma)(S-Z)^2I}{m+S+I}<\infty.
$$
Therefore, there is a measurable set $\tilde{\Omega}\subset \Omega$  with null measure such that 
\begin{equation}\label{HH}
H(x)=\int_{0}^{\infty}(\beta-\gamma)\frac{(S-Z)^2I}{m+S+I}<\infty, \quad \forall\, x\in\Omega\setminus\tilde{\Omega}.
\end{equation}
By Proposition \ref{prop3.2} and the equation of $I$ in \eqref{dI=0-model}, $\|\partial_t I\|_\infty$ is uniformly bounded for $t\in (0, \infty)$.  So, the function $[0,\infty)\ni I(t,\cdot)\in C(\bar{\Omega})$ is Lipschitz continuous.  Note that the mapping $[1,\infty)\ni S(t,\cdot)\in C(\bar{\Omega})$ is H\"older continuous. So for each $x\in\bar\Omega$, the mapping
$$
[0,\infty)\ni(\beta(x)-\gamma(x))\frac{(S(t,x)-Z(t,x))^2I(t,x)}{m(x)+S(t,x)+I(t,x)}
$$
is H\"older continuous. By \eqref{HH}, Lemma \ref{lem0}, and Proposition \ref{prop3.1},
\begin{equation}\label{Eqd-10}
\lim_{t\to\infty}(S(t,x)-Z(t,x))^2I(t,x)=0, \quad \mbox{a.e. in}\  \Omega.
\end{equation} 
 
\vskip10pt
In the sequel,  we are going to prove the assertion {\rm (i)}.

{\bf Case 1:} Suppose that $\frac{|\Omega|m\gamma}{\beta-\gamma}<\min\Big\{{N},N+\int_{\{I_0>0\}}m
-\frac{N}{|\Omega|}\int_{\{I_0>0\}}\frac{(\beta-\gamma)}{\gamma}\Big\}$. 
 
To the end, let us introduce the quantities
\begin{equation*}
    \underline{c}^*:=\liminf_{t\to\infty}\int_{\Omega}I(t,\cdot) \quad \text{and}\quad \bar{c}^*=\limsup_{t\to\infty}\int_{\Omega}I(t,\cdot).
\end{equation*}
By \eqref{total-population} and \eqref{Eqd-9}, for every $\varepsilon>0$, there is $t_{\varepsilon}\gg 1$ such that 
\begin{equation}\label{ineq-s}
\frac{(N-\bar{c}^*-\varepsilon)_+}{|\Omega|}< S(t,\cdot)<\frac{(N-\underline{c}^*)+\varepsilon}{|\Omega|},\quad \forall\, t\ge t_{\varepsilon}.
\end{equation}

Recall that if $I_0(x_0)=0$ for some $x_0\in\bar\Omega$, then $I(t, x_0)=0$ for all $t\ge 0$, and if $I_0(x_0)>0$ for some $x_0\in\bar\Omega$, then $I(t, x_0)>0$ for all $t\ge 0$. Thus, to determine the limit of $I$ as $t\to\infty$, it suffices to work on the set $\{I_0>0\}:=\{x\in\bar\Omega:\ I_0(x)>0\}$.

From \eqref{ineq-s} it follows that
$$
\p_tI\le\left[\frac{\beta \frac{(N-\underline{c}^*+\varepsilon)}{|\Omega|}  }{m+\frac{(N-\bar{c}^*+\varepsilon)}{|\Omega|}+I}-\gamma\right]I,\ \quad \forall\, x\in\{I_0>0\},\, t\ge t_{\varepsilon}.  
$$
Hence, by a comparison principle for ODEs, we have that 
$$
\limsup_{t\to\infty}I(t,x)\le\left[\frac{(\beta(x)-\gamma(x))(N-\underline{c}^*+\varepsilon)}{|\Omega|\gamma(x)}-m(x)\right]_+, \quad \forall\, x\in\{I_0>0\}.
$$
Letting $\varepsilon\to 0^+$, we get 
\begin{equation}\label{Eqd-11}
    \limsup_{t\to\infty}I(t,x)\le {\left[\frac{(N-\underline{c}^*)}{|\Omega|}-\frac{m(x)\gamma(x)}{\beta(x)-\gamma(x)}\right]_+}\frac{\beta(x)-\gamma(x)}{\gamma(x)}, \quad \forall\, x\in\{I_0>0\}.
\end{equation}
Similarly, we can proceed to establish that 
\begin{equation}\label{Eqd-12}
    \liminf_{t\to\infty}I(t,x)\ge {\left[\frac{(N-\bar{c}^*)}{|\Omega|}-\frac{m(x)\gamma(x)}{\beta(x)-\gamma(x)}\right]_+}\frac{\beta(x)-\gamma(x)}{\gamma(x)}, \quad \forall\, x\in\{I_0>0\}.
\end{equation}
It is clear that \eqref{th2.5-eq3} follows from \eqref{Eqd-11}. 

Thanks to \eqref{Eqd-11}, Fatou Lemma, and the uniform boundedness of $\|I(t,\cdot)\|_{\infty}$, we have 
$$
\bar{c}^*=\limsup_{t\to\infty}\int_{\Omega}I(t,x)\le\int_{\Omega}\limsup_{t\to\infty}I(t,x)\le \int_{\{I_0>0\}}\Big[\frac{N(\beta-\gamma)}{|\Omega|\gamma}-m\Big]_+.
$$

Therefore,  by means of \eqref{Eqd-12}, we obtain that
\begin{align}\label{Eqd-10-aa}
    \liminf_{t\to\infty}I(t,x)\ge& \left[\frac{N-\int_{\{I_0>0\}}\Big(\frac{N(\beta-\gamma)}{|\Omega|\gamma}-m\Big)}{|\Omega|}-\frac{\gamma m}{\beta-\gamma}\right]\frac{\beta(x)-\gamma(x)}{\gamma(x)}\cr 
    =&\left[N+\int_{\{I_0>0\}}m-\frac{N}{|\Omega|}\int_{\{I_0>0\}}\frac{(\beta-\gamma)}{\gamma}
    -\frac{|\Omega|m\gamma}{\beta-\gamma}\right]\frac{\beta(x)-\gamma(x)}{|\Omega|\gamma(x)}\cr
    >&0, \quad \forall\, x\in\{I_0>0\}.
\end{align}

Then, we deduce from \eqref{Eqd-10-aa} and \eqref{Eqd-10}  that 
\begin{equation}\label{Eqd-15}
    \lim_{t\to\infty}(S(t,\cdot)-Z(t,\cdot))=0 \quad \text{a.e. in}\ \{I_0>0\}.
\end{equation}

Let $\{t_j\}_{j\ge 1}$ be a sequence of positive numbers converging to infinity. By \eqref{Eqd-9} and Proposition \ref{prop3.1}, possibly after passing to a subsequence, there is a positive number $\hat S$ such that $\|S(t_j,\cdot)-\hat S\|_{\infty}\to 0$ as $j\to\infty$. By the Lebesgue Dominated Convergence Theorem,  \eqref{Eqd-15}, and Proposition \ref{prop3.2}, we have 
$$
\mbox{$Z(t_j,\cdot)\to \hat S$\ \ as\ $j\to\infty$\ \ in\, $L^{p}(\{I_0>0\})$\ \ for all\ $p\ge 1$.}
$$
Hence
$$
I(t_j,\cdot)=\frac{\beta-\gamma}{\gamma}Z(t_j,\cdot)-m\to \frac{\beta-\gamma}{\gamma}\hat S-m \quad \text{as}\ j\to\infty, 
$$
in $L^p(\{I_0>0\})$ for any $p\ge 1$. Thus
\begin{align}\nonumber
    N=&\lim_{j\to\infty}\int_{\Omega}[S(t_j,\cdot)+I(t_j,\cdot)]
    =\lim_{j\to\infty}\int_{\Omega}S(t_j,\cdot)+\lim_{j\to\infty}\int_{\{I_0>0\}}I(t_j,\cdot)\cr
    =&|\Omega|\hat S+\int_{\{I_0>0\}}\Big(\frac{\beta-\gamma}{\gamma}\hat S-m\Big)\cr
    =&|\Omega|\hat S+\hat S\int_{\{I_0>0\}}\frac{\beta-\gamma}{\gamma}-\int_{\{I_0>0\}}m.
\end{align}

Solving  for $\hat S$, we get 
$$\hat S=\frac{N+\int_{\{I_0>0\}}m}{|\Omega|+\int_{\{I_0>0\}}\frac{\beta-\gamma}{\gamma}}.$$
Since $\hat S$ is independent of the chosen sequence, we conclude that 
$$
\|S(t,\cdot)-\hat S\|_{\infty}\to 0\ \  \text{as}\ t\to \infty.
$$
Furthermore, by the equation of $I$, 
$$
\Big\|I(t,\cdot)-\hat I(\cdot)\Big\|_{L^p(\Omega)}\to 0\ \ \text{as}\ t\to\infty
$$
for any $p\geq1$, where $\hat I$ is given in \eqref{th2.5}(i). This proves {\rm(i)} in Case 1.

\vskip5pt {\bf Case 2:} Suppose that $\frac{m\gamma}{\beta-\gamma}\Big(|\Omega|+\int_{\{I_0>0\}}\frac{\beta-\gamma}{\gamma}\Big)<N$.

In view of \eqref{Eqd-9} and Proposition \ref{prop3.1}, for any sequence $\{t_j'\}$ with $t_j'\to\infty$ as $j\to\infty$, we can find a subsequence (labeled by itself without loss of generality), and 
a positive number $S_*$ such that $\|S(t_j',\cdot)-S_*\|_{\infty}\to 0$ as $j\to\infty$. Here the positive constant $S_*$ may depend on the choice of the sequence $\{t_j'\}$.

Due to \eqref{Eqd-10}, we may assume that 
\begin{equation}\label{Eqd-15-a}
S(t_j',x)-Z(t_j',x)\to 0\ \ \ \mbox{for}\ x\in\Omega_*\subset\Omega\ \ \mbox{as}\ j\to\infty,
\end{equation}
and 
\begin{equation}\label{Eqd-15-b}
I(t_j',x)\to0\ \ \ \mbox{for}\ \  x\in\Omega\setminus\Omega_*\ \ \mbox{as}\ j\to\infty,
\end{equation}
where $\Omega_*$ may be an empty set and depend on the choice of the sequence $\{t_j'\}$.
Arguing as in Case 1, combined with Proposition \ref{prop3.2}, we deduce from \eqref{Eqd-15-a} that 
\begin{equation}\label{Eqd-15-c}
\mbox{$Z(t_j',\cdot)\to S_*$\ \ as\ $j\to\infty$\ \ in\, $L^{p}(\Omega_*)$\ \ for all\ $p\ge 1$.}
\end{equation}
Thus, it follows from \eqref{Eqd-15-c} that
\begin{equation}\label{Eqd-15-d}
I(t_j',\cdot)=\frac{\beta-\gamma}{\gamma}Z(t_j',\cdot)-m\to \frac{\beta-\gamma}{\gamma}S_*-m \quad \text{as}\ j\to\infty, 
\end{equation}
in $L^p(\{I_0>0\}\cap\Omega_*)$ for any $p\ge 1$.

Since $\int_{\Omega}[S(t_j',\cdot)+I(t_j',\cdot)]=N$ for each $j$, as in Case 1, using \eqref{Eqd-15-b} and \eqref{Eqd-15-d}, together with the fact of $I(t_j',x)=0$ for each $j$ and $x\in\{I_0=0\}$, we get
$$
S_*=\frac{N+\int_{\{I_0>0\}\cap\Omega_*}m}
{|\Omega|+\int_{\{I_0>0\}\cap\Omega_*}\frac{\beta-\gamma}{\gamma}}\geq
\frac{N}{|\Omega|+\int_{\{I_0>0\}}\frac{\beta-\gamma}{\gamma}}.
$$
Notice that $\frac{N}{|\Omega|+\int_{\{I_0>0\}}\frac{\beta-\gamma}{\gamma}}$ is independent of the choice of the sequence $\{t_j'\}$. This implies that for any given $\epsilon>0$, there is a large $t_{\epsilon,*}>0$ such that
$$
S(t,x)\geq \frac{N}{|\Omega|+\int_{\{I_0>0\}}\frac{\beta-\gamma}{\gamma}}-\epsilon,\ \ \ \forall t\geq t_{\epsilon,*},\,x\in\bar\Omega.
$$

As a result, we see from the $I$-equation that
$$
\p_tI\ge\left[\frac{\beta \Big(\frac{N}{|\Omega|+\int_{\{I_0>0\}}\frac{\beta-\gamma}{\gamma}}-\epsilon \Big) }{m+\frac{N}{|\Omega|+\int_{\{I_0>0\}}\frac{\beta-\gamma}{\gamma}}-\epsilon+I}-\gamma\right]I,\ \quad \forall\, x\in\{I_0>0\},\, t\ge t_{\epsilon,*}.  
$$
Then an application of the comparison principle for ODEs yields 
$$
\liminf_{t\to\infty}I(t,x)\ge\left[\frac{(\beta(x)-\gamma(x))}{\gamma(x)}\Big(\frac{N}{|\Omega|
+\int_{\{I_0>0\}}\frac{\beta-\gamma}{\gamma}}-\epsilon\Big)-m(x)\right]_+, \quad \forall\, x\in\{I_0>0\}.
$$
By sending $\varepsilon\to 0^+$, we get
\begin{equation}\label{Eqd-15-e}
\liminf_{t\to\infty}I(t,x)\ge\left[\frac{(\beta(x)-\gamma(x))}{\gamma(x)}\Big(\frac{N}{|\Omega|
+\int_{\{I_0>0\}}\frac{\beta-\gamma}{\gamma}}\Big)-m(x)\right]>0, \quad \forall\, x\in\{I_0>0\}
\end{equation}
due to our assumption. 

With the help of \eqref{Eqd-15-e}, we can use the same analysis as in Case 1 to conclude the desired assertion. 

{\rm(ii)} Suppose that $\frac{\gamma m}{\beta-\gamma}\ge \frac{N}{|\Omega|}$. It follows from \eqref{th2.5-eq3} that $I(t,\cdot)\to0$ uniformly on $\bar\Omega$ as $t\to\infty$. Hence $\int_{\Omega}S(t,x)dx\to N$ as $t\to\infty$. This along with \eqref{th2.5-eq1} yields {\rm (ii)}. 
\end{proof}

\vskip10pt
We are now in a position to prove Theorem \ref{th2.6}.

\begin{proof}[Proof of Theorem \ref{th2.6}] Thanks to Proposition \ref{prop3.2}, there is $L>0$ such that 
$$
\|S(t,\cdot)\|_{\infty}+\|I(t,\cdot)\|_{\infty}\le L.
$$
Fix $x\in\bar{\Omega}$ such that $ \beta(x)\le \gamma(x)$. Then 
$$
\p_tI(t,x)\le -\frac{\gamma I^2(t,x)}{m(x)+S(t,x)+I(t,x)}\le -\frac{\gamma I^2(t,x)}{m(x)+L},\quad  \forall\, t>0.
$$
Then by the comparison principle for ODEs, we obtain
\begin{equation}\label{Iine}
I(t,x)\le\frac{m(x)+L}{\frac{m(x)+L}{I_0(x)}+\gamma t}\to 0 \quad \text{as}\ t\to\infty. 
\end{equation}

{\rm (i)} Suppose that $\beta\le \gamma$ on $\bar\Omega$. Then, it follows from \eqref{Iine} that 
$$
\|I(t,\cdot)\|_{\infty}\to 0 \quad \mbox{as}\ t\to\infty.
$$
Hence, $\|(\beta S/(m+S+I)-\gamma)I(t,\cdot)\|_{\infty}\to 0$ and $\int_{\Omega}S=N-\int_{\Omega}I\to N$ as $t\to\infty$. By the first equation of \eqref{dI=0-model} and standard results on the large time behavior of solutions of the heat equations subject to the homogeneous boundary conditions, we can easily conclude that $\|S(t,\cdot)-\frac{N}{|\Omega|}\|_{\infty}\to 0$ as $t\to\infty$.

{\rm (ii)}  Let $l_1>0$ be the positive number given by \eqref{Eqd-3}, which is independent of $m$ and $N$. 
Let $K\subset H^+$ be a compact set and set 
$$
N_{*}:=\max_{x\in K}\frac{2m(x)\gamma(x)}{(\beta(x)-\gamma(x))l_1}.
$$
It is clear that $N_{*}=0$ if $m=0$ on $K$. Note also that $N_*$ is independent of initial data and $N$.

Thanks to \eqref{Eqd-3}, there is $t_1>0$ such that 
$$
S(t,x)\ge \frac{Nl_1}{2},\quad \forall\, x\in\bar{\Omega}, \,  t\ge t_1.
$$
Thus, for every $N>N_{*}$ and $x\in K$, there holds that 
\begin{align*}
\p_tI=&\Big[\frac{\beta S}{m+S+I}-\gamma\Big]I\cr 
\ge& \Big[\frac{\frac{\beta Nl_{1}}{2}}{m+\frac{Nl_1}{2}+I}-\gamma\Big]I\cr
=&\Big[\frac{(\beta-\gamma)l_1}{2\gamma}\Big(N-\frac{2m\gamma}{\beta-\gamma}\Big)-I\Big]\frac{\gamma I}{m+\frac{Nl_1}{2}+I}\cr
\ge& \Big\{(N-N_*)l_1\Big[\min_{x\in K}\frac{(\beta(x)-\gamma(x))}{2\gamma(x)}\Big]-I\Big\}\frac{\gamma I}{m+\frac{Nl_1}{2}+I},\quad t>t_1.
\end{align*}
It then follows from the comparison principle for ODEs that 
$$
\liminf_{t\to\infty}\min_{x\in K}I(t,x)\ge (N-N_*)l_1\Big[\min_{x\in K}\frac{(\beta(x)-\gamma(x))}{2\gamma(x)}\Big]>0,
$$
which yields the desired result.
\end{proof}

\vskip5pt
In order to establish Proposition \ref{prop3}, we need to prepare some preliminary results. The first one reads as follows.
\begin{lem}\label{lem4.2}  Fix $d_S>0$ and let $(S,I)$ be the solution of \eqref{dI=0-model}. Then for each $x\in\bar\Omega$,
 \begin{equation}
     ((R-1)Nl_1-m)_+\chi_{\{I_0>0\}}\le \liminf_{t\to\infty} I(t, x)\le \limsup_{t\to\infty} I(t, x)\le ((R-1)l_2-m)_+\chi_{\{I_0>0\}},
 \end{equation}
 where $l_1$ and $l_2$ are given in Proposition \ref{prop3.1} and Proposition \ref{prop3.2}, respectively.  
\end{lem}

\begin{proof}
    By Proposition \ref{prop3.1} and Proposition \ref{prop3.2}, we have 
 \begin{equation}\label{infsup}
     Nl_1\le \liminf_{t\to\infty} S(t, x)\le \limsup_{t\to\infty} S(t, x)\le l_2, \quad \text{uniformly for}\ x\in\bar\Omega.
 \end{equation}
So for every $\varepsilon>0$, there is $t_{\varepsilon}\gg 1$ such that 
$$
Nl_1-\varepsilon< S(t,x)<l_2+\varepsilon,\quad \forall\, t>t_{\varepsilon},\, x\in\bar\Omega,
$$
which in turn implies that
$$
\p_tI\le\Big[\frac{\beta (l_2+\varepsilon) }{m+(l_2+\varepsilon)+I}-\gamma\Big]I,\ \quad x\in\bar{\Omega},\ t\ge t_{\varepsilon}.  
$$
Hence, by the comparison principle for ODEs, we have that 
$$
\limsup_{t\to\infty}I(t,x)\le \Big[(R(x)-1)(l_2+\varepsilon)-m(x)\Big]_+\chi_{\{I_0>0\}}, \quad \forall\, x\in\bar{\Omega}.
$$
Taking $\varepsilon\to 0^+$, we get 
\begin{equation*}
    \limsup_{t\to\infty}I(t,x)\le[(R-1)l_2-m]_+\chi_{\{I_0>0\}},  \quad \forall\, x\in\bar{\Omega}.
\end{equation*}
Similarly, we can proceed to establish that 
\begin{equation*}
\liminf_{t\to\infty}I(t,x)\ge [(R(x)-1)Nl_1-m(x)]_+\chi_{\{I_0>0\}}, \quad \forall\, x\in\bar{\Omega}.
\end{equation*}
Thus, the desired results are proved.
\end{proof}

\begin{lem}\label{lemma_S}  Fix $d_S>0$ and suppose that $1/(\beta-\gamma)\in L^1(H^+\cap \{I_0>0\})$.   Let $(S,I)$ be the solution of \eqref{dI=0-model}.   Then we have
\begin{equation}\label{Sc}
     \lim_{t\to\infty}\Big\|S(t,\cdot)-\frac{1}{|\Omega|}\int_{\Omega}S(t,\cdot)\Big\|_{\infty}=0.
\end{equation}
\end{lem}
\begin{proof}
Define 
\begin{equation}\nonumber
V(t):=\frac{1}{2}\int_{\Omega}S^2+\int_{H^+\cap \{I_0>0\}}\frac{\gamma (I+m)^2}{2(\beta-\gamma)},\ \ t>0.
\end{equation}
Due to $1/(\beta-\gamma)\in L^1(H^+\cap \{I_0>0\})$, $V(t)$ is well defined for all $t>0$. It is easy to check that  
\begin{align}\label{Vp}
    \frac{dV}{dt}=&-d_S\int_{\Omega}|\nabla S|^2-\int_{H^+\cap \{I_0>0\}}\frac{I((\beta-\gamma)S-\gamma I-\gamma m)^2}{(\beta-\gamma)(S+I+m)}\cr
    &\ +\int_{H^-\cup H^0\cup \{I_0=0\}} \left(\frac{\beta S}{S+I+m}-\gamma\right)SI,\ \ \forall\, t>0.
\end{align}

By the equation of $I$, we obtain
$$
\int_0^t\int_{H^-\cup H^0\cup \{I_0=0\}} \left(\frac{\beta S}{S+I+m}-\gamma\right)I = \int_{H^-\cup H^0\cup \{I_0=0\}}(I-I_0).
$$
If $x\in H^-\cup H^0\cup \{I_0=0\}$, then 
$$
\left(\frac{\beta S}{S+I+m}-\gamma\right)I \le 0.
$$
Hence, 
$$
0\le -\int_0^\infty\int_{H^-\cup H^0\cup \{I_0=0\}} \left(\frac{\beta S}{S+I+m}-\gamma\right)I<\infty.
$$
Due to Proposition \ref{prop3.1}, we further have 
$$
0\le -\int_0^\infty\int_{H^-\cup H^0\cup \{I_0=0\}} \left(\frac{\beta S}{S+I+m}-\gamma\right)SI<\infty.
$$
Therefore, integrating \eqref{Vp} over $\Omega\times (0, \infty)$, we obtain 
\begin{equation}\label{gS}
    \int_0^\infty \int_\Omega |\nabla S|^2<\infty.
\end{equation}
%and 
%\begin{equation}
% \int_0^\infty \int_{H^+\cap \{I^0>0\}}\frac{I((\beta-\gamma)S-\gamma I-\gamma m)^2}{(\beta-\gamma)(S+I+m)}dxdt<\infty.  
%\end{equation}

 By Proposition \ref{prop3.2}, we have $\sup_{t\ge 0}\|(\beta S(t,\cdot)/(m+S+I)-\gamma)I\|_{\infty}<\infty$. By the regularity theory for parabolic equations, the mapping $[1,\infty)\ni t\mapsto S(t,\cdot)\in C^{1}(\bar{\Omega})$ is Holder's continuous. In particular, $\|\nabla S(t,\cdot)\|_{L^2(\Omega)}$ is H\"older's continuous on $[1,\infty)$. So, we can conclude from \eqref{gS} and Lemma \ref{lem0} that 
\begin{equation}\label{SS}
    \lim_{t\to\infty}\||\nabla S(t,\cdot)|\|_{L^2(\Omega)}=0.
\end{equation}
By virtue of \eqref{SS}, the Poincar\'e inequality and the regularity of $S(t,\cdot)$, we can proceed as in \eqref{Eqa-2} to establish that 
\begin{equation}\nonumber
    \lim_{t\to\infty}\Big\|S(t,\cdot)-\frac{1}{|\Omega|}\int_{\Omega}S(t,\cdot)\Big\|_{C^1(\bar{\Omega})}=0,
\end{equation}
 which proves \eqref{Sc}.
\end{proof}

\vskip10pt
With the above preparation, we are now ready to verify Proposition \ref{prop3}.

\begin{proof}[Proof of Proposition \ref{prop3}]
Since $0\le\int_\Omega S\le N$ for all $t\ge 0$, there exist constants $\zeta, \eta\in [0, N/|\Omega|]$ with $\zeta\le \eta$ such that 
 \begin{equation}\label{limS}
\zeta|\Omega|=\liminf_{t\to\infty} \int_\Omega S(t,\cdot)\le \limsup_{t\to\infty}\int_\Omega S(t,\cdot)= \eta|\Omega|.  
\end{equation}

For any $\epsilon>0$, by Lemma \ref{lemma_S}, there exists $t_\epsilon>0$ such that 
$$
\zeta-\epsilon\le S( t,x)\le \eta+\epsilon, \ \ \forall x\in\bar\Omega, t\ge t_\epsilon. 
$$
Similarly to the proof of Lemma \ref{lem4.2}, for each $x\in\bar\Omega$, 
 \begin{equation}\label{limI}
     ((R-1)\zeta-m)_+\chi_{\{I_0>0\}}\le \liminf_{t\to\infty} I(t, x)\le \limsup_{t\to\infty} I(t, x)\le ((R-1)\eta-m)_+\chi_{\{I_0>0\}}.
 \end{equation}

By  \eqref{total-population} and Fatou's lemma, we have 
 $$
|\Omega|\eta=\limsup_{t\to\infty} \int_\Omega S=N-\liminf_{t\to\infty}\int_\Omega I\le N-\int_\Omega  ((R-1)\zeta-m)_+\chi_{\{I_0>0\}}
 $$
 and 
 $$
|\Omega| \zeta=\liminf_{t\to\infty} \int_\Omega S=N-\limsup_{t\to\infty}\int_\Omega I\ge N-\int_\Omega  ((R-1)\eta-m)_+\chi_{\{I_0>0\}}.
 $$
 It follows that 
\begin{equation}\label{ul}
   |\Omega|\eta-\int_\Omega  ((R-1)\eta-m)_+\chi_{\{I_0>0\}}\le |\Omega|\zeta-\int_\Omega  ((R-1)\zeta-m)_+\chi_{\{I_0>0\}}. 
\end{equation}

Define the function 
$$
f(\tau):=|\Omega|\tau-\int_\Omega  ((R-1)\tau-m)_+\chi_{\{I_0>0\}}, \ \ \tau\in [0, N/|\Omega|].
$$ 
We claim that $f$ is  strictly increasing. Indeed, if $\sigma>0$, then 
\begin{eqnarray*}
 f(\tau+\sigma)-f(\tau)&=&|\Omega| \sigma-\left[\int_\Omega  ((\tau+\sigma)(R-1)-m)_+\chi_{\{I_0>0\}}   - \int_\Omega  (\tau(R-1)-m)_+\chi_{\{I_0>0\}}\right]\\
 &\ge&\sigma\left(|\Omega|-\int_\Omega(R-1)_+\chi_{\{I_0>0\}}\right)>0,
\end{eqnarray*}
where we have used the inequality $a_++b_+\ge (a+b)_+$ and the assumption $|\Omega|-\int_\Omega(R-1)_+\chi_{\{I_0>0\}}>0$. 

Since $f$ is increasing and $\zeta\le \eta$, we have $f(\zeta)\le f(\eta)$. However, \eqref{ul} is equivalent to $f(\eta)\le f(\zeta)$. Hence, $f(\eta)=f(\zeta)$ and $\eta=\zeta=:S^*$ is the unique solution of 
$$
S^*|\Omega|+\int_\Omega ((R-1)S^*-m)_+\chi_{\{I_0>0\}}=N. 
$$
Thus,  Proposition \ref{prop3} is a direct consequence of \eqref{limS} and \eqref{limI}. 
\end{proof}

\begin{proof}[Proof of Corollary \ref{coro2.3}]
It is easy to see that the unique solution $S^*\in (0, N/|\Omega|]$ of the algebraic equation  
    $$
    S^*|\Omega|+\int_\Omega ((R-1)S^*-m)_+\chi_{\{I_0>0\}}=N
    $$
is $S^*=N/|\Omega|$ if and only if 
$$
\mbox{$(R-1)N/|\Omega|-m\le 0$\ \ on $H^+\cap\{I_0>0\}$.}
$$
Then the desired results readily follow from Proposition \ref{prop3}. 
\end{proof}

\section{Numerical simulations}

In this section, we conduct numerical simulations to illustrate the main results. 

For the simulations, we take \(\Omega = (0, 1)\). Let \(S_0 = 2 + \cos(\pi x)\) and \(I_0 = 1.5 + \cos(\pi x)\). Then, \(\int_\Omega (S_0 + I_0) = 3.5 = N\). Define
\[ 
m_0(x) = 
\begin{cases} 
1 - 2x, & \text{if } x \in (0, 0.5), \\ 
0, & \text{if } x \in [0.5, 1). 
\end{cases} 
\]
Recall that \(M^0 := \{x \in \bar{\Omega} : m(x) = 0\}\). Thus, \(M^0 = [0.5, 1]\) when \(m = m_0\).

\subsection{Simulations for system \eqref{ds=0-model}}
In this subsection, we perform numerical simulations for the system \eqref{ds=0-model} to verify and complement the main theoretical results presented in Subsection 2.1.

{\bf Simulation 1.}  In this simulation, we examine the case where \(\beta > \gamma\) on \(\Omega\). We set \(d_I = 1\), \(\beta = k + x\), \(\gamma = 5 + x\), and \(m = m_0\), where \(k\) is a positive constant. Define
\[
f(k) := \int_0^1 \frac{m\gamma}{\beta - \gamma}.
\]
We find that \(f(5.1) \approx 12.9167\), \(f(5.37) \approx 3.4910\), and \(f(6) \approx 1.2917\). Fig.\ref{fig:1}(A) illustrates the scenario where \(\beta = 6 + x\), resulting in \(f(k) < N = 3.5\). In this case, \(S\) converges to a positive function, and \(I\) converges to a positive constant. Fig.\ref{fig:1}(B) shows the scenario where \(\beta = 5.1 + x\), with \(f(k) > N\). Here, \(S\) converges to a positive function, and \(I\) converges to zero. Fig.\ref{fig:1}(A)-(B) are consistent with the predictions from Theorem \ref{th2.1}.

To demonstrate Corollary \ref{cor1}, we set \(k = 5.2\) and choose \(d_I = 0.1, 0.21\) or 0.5. Then, \(N \leq \int m\gamma/(\beta - \gamma)\). Fig.\ref{fig:1}(C) shows that \(S\) converges to zero in a subset of $M^0$, when $d_I$ is small ($d_I=0.1$ or 0.21). 
If $d_I$ is large ($d_I=0.5$), then $S$ persists in $\Omega$. This is consistent with Corollary \ref{cor1}.

\begin{figure}
     \centering
     \begin{subfigure}[b]{0.32\textwidth}
         \centering
          \includegraphics[scale=.3]{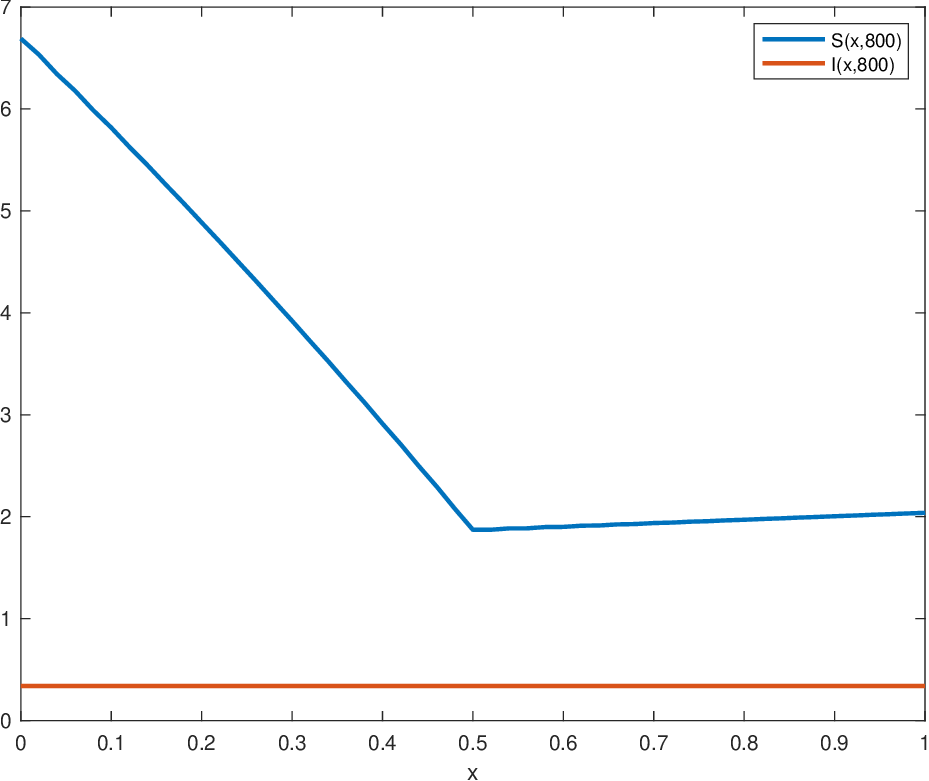}
        \caption{{$d_I=1,\,\beta=6+x,\, m=\quad m_0$}}
         \label{fig:a14}
     \end{subfigure}
     \begin{subfigure}[b]{0.32\textwidth}
        \centering
     \includegraphics[scale=.3]{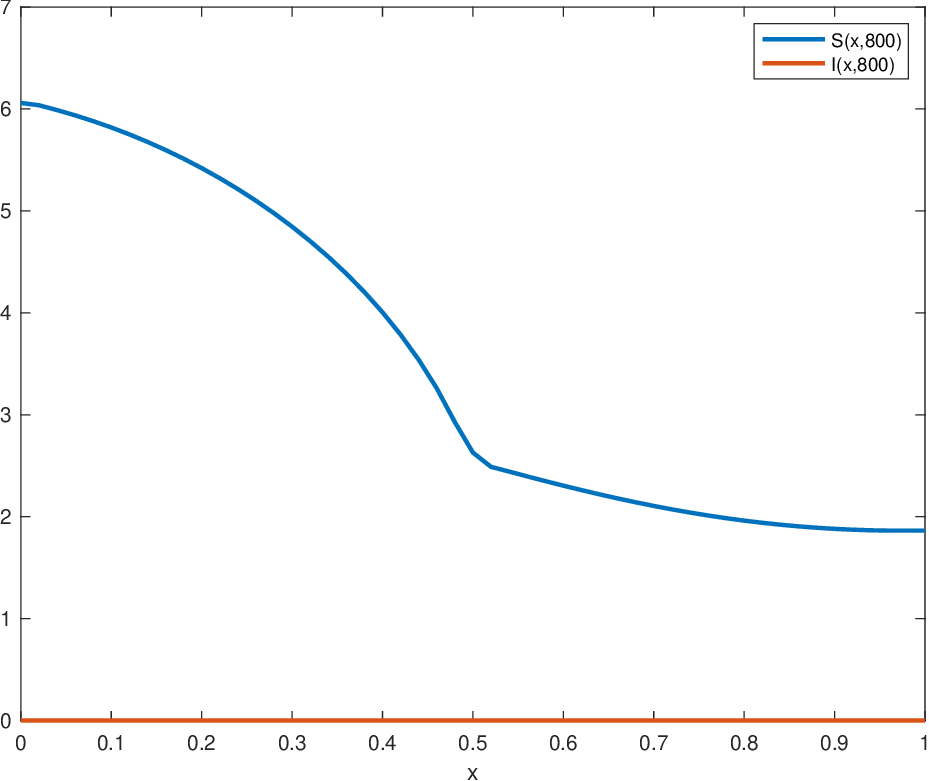}
    \caption{{$d_I=1,\, \beta=5.1+x,\, m= m_0$}}
         \label{fig:b14}
     \end{subfigure}
     \begin{subfigure}[b]{0.32\textwidth}
        \centering
     \includegraphics[scale=.3]{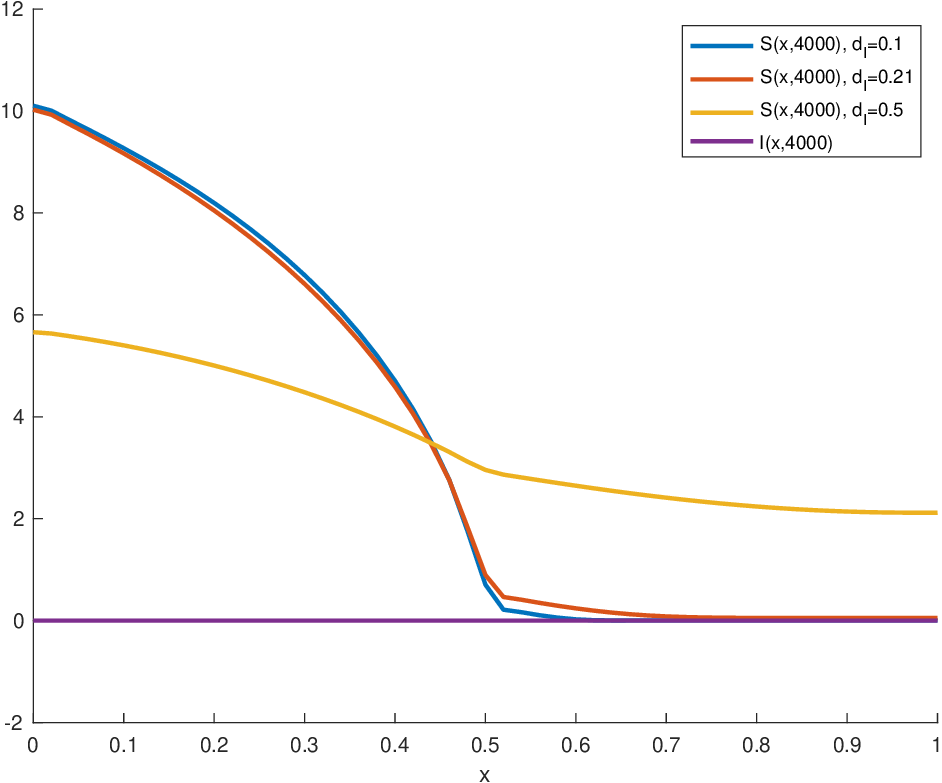}
  \caption{{\footnotesize$d_I=0.1, 0.21, \text{or}\, 0.5, \beta=5.2+x,m=m_0$}} \label{fig:b15}
     \end{subfigure}
     \caption{Simulation of model \eqref{ds=0-model} with $\beta>\gamma$.}
     \label{fig:1}
\end{figure}

\begin{figure}
     \centering
     \begin{subfigure}[b]{0.32\textwidth}
         \centering
          \includegraphics[scale=.3]{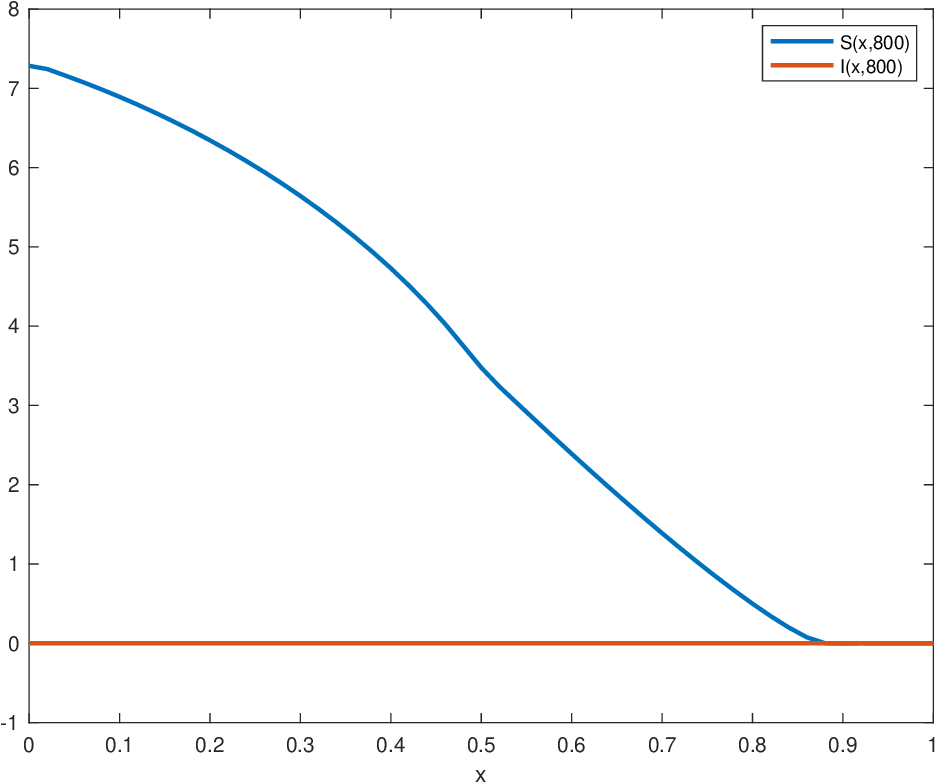}
        \caption{$\beta=\max\{5+x, 4+3x\}$, $m=m_0$}
     \end{subfigure}
     \begin{subfigure}[b]{0.32\textwidth}
        \centering
     \includegraphics[scale=.3]{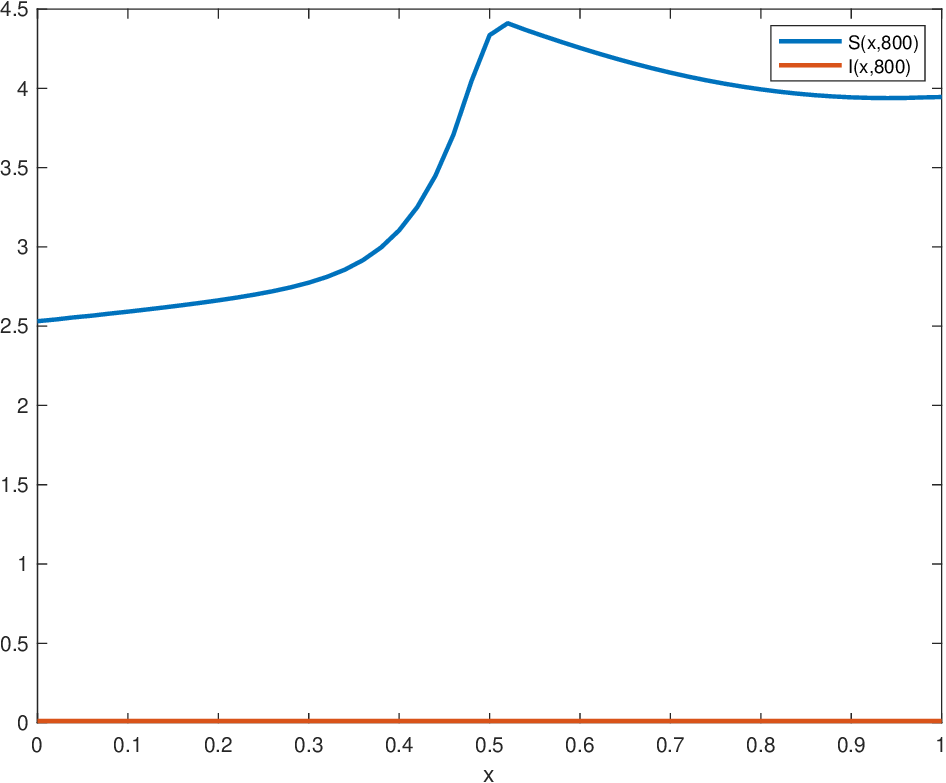}
             \caption{$\beta=\max\{5+x, 7-3x\}$, $m=m_0$}
     \end{subfigure}
               \begin{subfigure}[b]{0.32\textwidth}
         \centering
          \includegraphics[scale=.3]{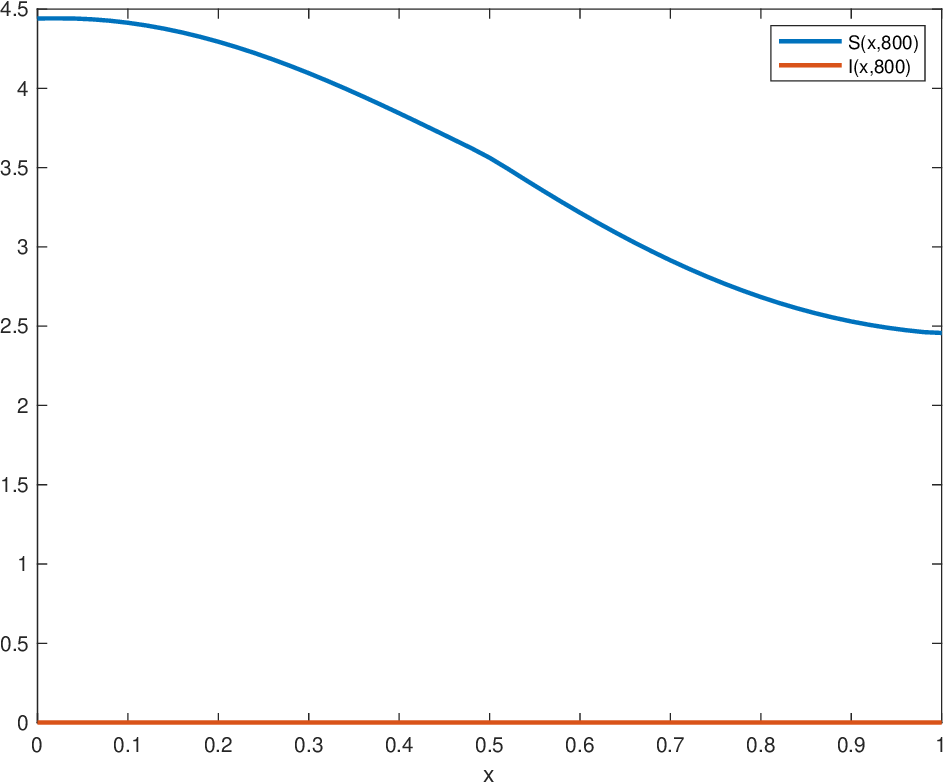}
         \caption{$\beta=\max\{5+x, 4+3x\}$, $m=1$}
     \end{subfigure}
     \caption{Simulation of model \eqref{ds=0-model} with $\beta\ge\gamma$ and $H^0\neq\emptyset$.}
     \label{fig:2}
 %    \vskip -20pt
\end{figure}

\textbf{Simulation 2.} In this simulation, we examine the case where \(\beta \ge \gamma\) and \(H^0 = \{x \in \bar{\Omega} : \beta(x) = \gamma(x)\}\) is non-empty. First, we set \(d_I = 1\), \(\gamma = 5 + x\), \(\beta = \max\{5 + x, 4 + 3x\}\), and \(m = m_0\). Thus, we have
\[
H^+ = (0.5, 1], \quad H^0 = [0, 0.5] \ \ \text{and} \ \ M^0 \cap H^+ = (0.5, 1).
\]
Fig\ref{fig:2}(A) shows that \(I\) converges to zero, and \(S\) converges to \(S^*\), where \(\{x \in \bar{\Omega} : S^*(x) = 0\}\) is a proper subset of \(M^0\cap H^+\). In this scenario, the disease depletes the susceptible population in part of the high-risk region where no saturated incidence effect exists. This agrees well with Theorem \ref{th2.2} and Corollary \ref{cor1}.

Next, we set \(d_I = 1\), \(\gamma = 5 + x\), \(\beta = \max\{5 + x, 7 - 3x\}\), and \(m = m_0\). Thus, we have
\[
H^+ = [0, 0.5), \quad H^0 = [0.5, 1]\ \ \text{and} \ \ M^0 \cap H^+ = \emptyset.
\]
Fig.\ref{fig:2}(B) shows that \(I\) still converges to zero, similar to the observation in Fig.\ref{fig:2}(A). However, in such a case, \(S\) converges to a positive function, and the infection does not deplete the disease in any part of the high-risk region.  This again agrees with Theorem \ref{th2.2}.

Finally, we use the same \(\beta\) and \(\gamma\) as in Fig.\ref{fig:2}(A) but set \(m = 1\), resulting in \(M^0 \cap H^+ = \emptyset\). Fig.\ref{fig:2}(C) shows that \(S\) also converges to a positive function in this scenario. Combining the results from Fig.\ref{fig:2}(A)-(C), we find that the region \(H^0\) may aid in eradicating the disease. Moreover, whether the susceptible population will be depleted in some subregion of the high-risk area depends on whether \(M^0 \cap H^+\) is empty or not. This simulation result is also consistent with Theorem \ref{th2.2}.

\textbf{Simulation 3.} In this simulation, we explore the case where \(H^- := \{x \in \bar{\Omega} : \beta(x) < \gamma(x)\} \neq \emptyset\). We fix \(d_I = 1\) and \(\gamma = 5 + x\). We choose \(m = m_0\) and \(\beta = k - x\) for \(k = 5.5, 6, 6.5\), resulting in \(H^+ = [0, 0.25), [0, 0.5), [0, 0.75)\), respectively. Fig.\ref{fig:30}(A) suggests that the limit of \(S\) is positive on \([0, 1]\) for \(k = 5.5, 6\). However, when \(k = 6.5\), it appears that \(S\) converges to zero in a subset of \(H^+ \cap M^0 = [0.5, 0.75)\).

Next, we choose \(m = m_0\) and \(\beta = k + 2x\) for \(k = 4.25, 4.5, 4.75\), resulting in \(H^+ = (0.75, 1], (0.5, 1], (0.25, 1]\), respectively. Although \(H^+ \cap M^0\neq \emptyset\) for all such \(k\), \(S\) converges to zero in a subset only when \(k = 4.75\), as shown in Fig.\ref{fig:30}(B).

Finally, if \(m = 1\) and \(\beta = k - x\) for \(k = 5.5, 6, 6.5\), Fig.\ref{fig:30}(C) shows that \(S\) always converges to a positive function since \(H^+ \cap M^0 = \emptyset\). In this simulation, \(I\) always converges to zero, consistent with Theorem \ref{th2.3}. 

\begin{figure}
     \centering
     \begin{subfigure}[b]{0.32\textwidth}
         \centering
          \includegraphics[scale=.3]{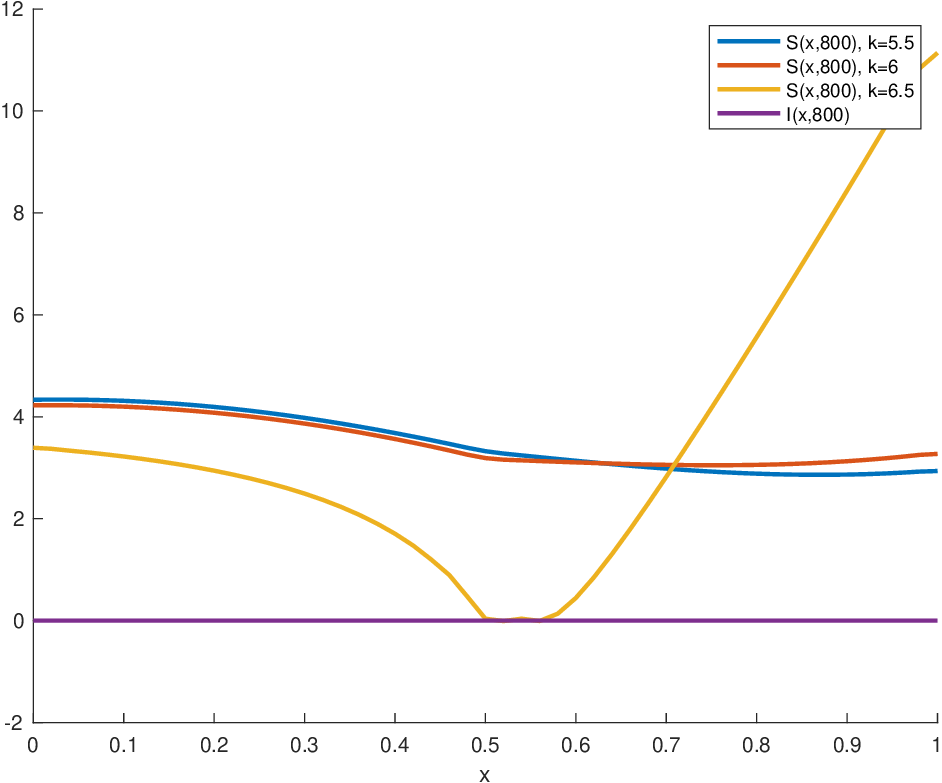}
        \caption{$\beta=k-x$, $m=m_0$}
     \end{subfigure}
     \begin{subfigure}[b]{0.32\textwidth}
        \centering
     \includegraphics[scale=.3]{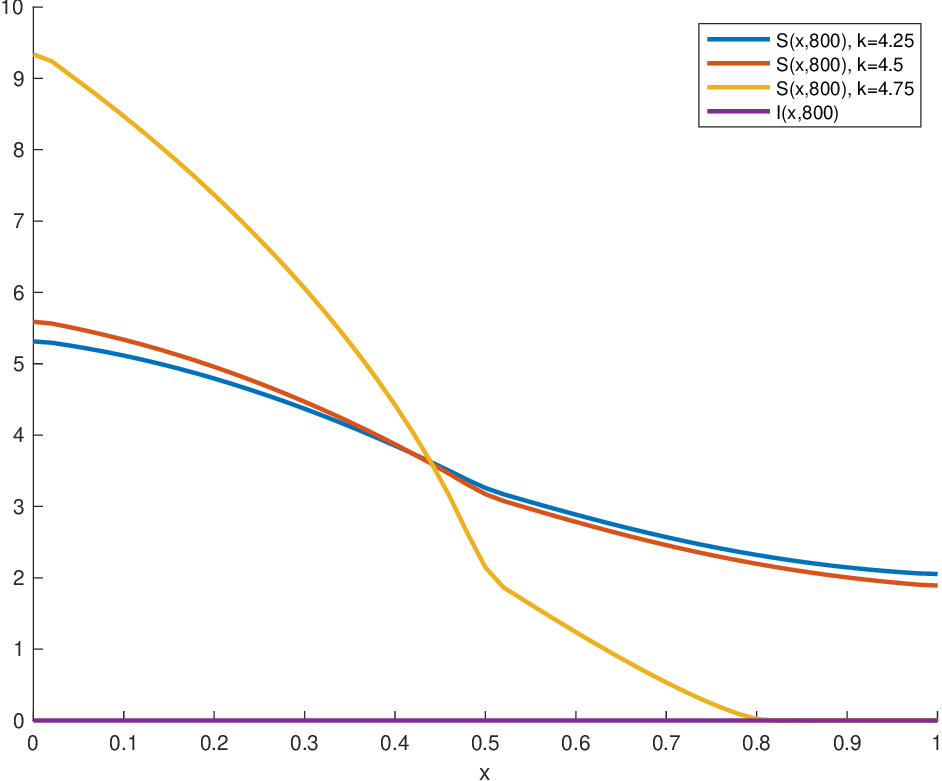}
             \caption{$\beta=k+2x$, $m=m_0$}
     \end{subfigure}
               \begin{subfigure}[b]{0.32\textwidth}
         \centering
          \includegraphics[scale=.3]{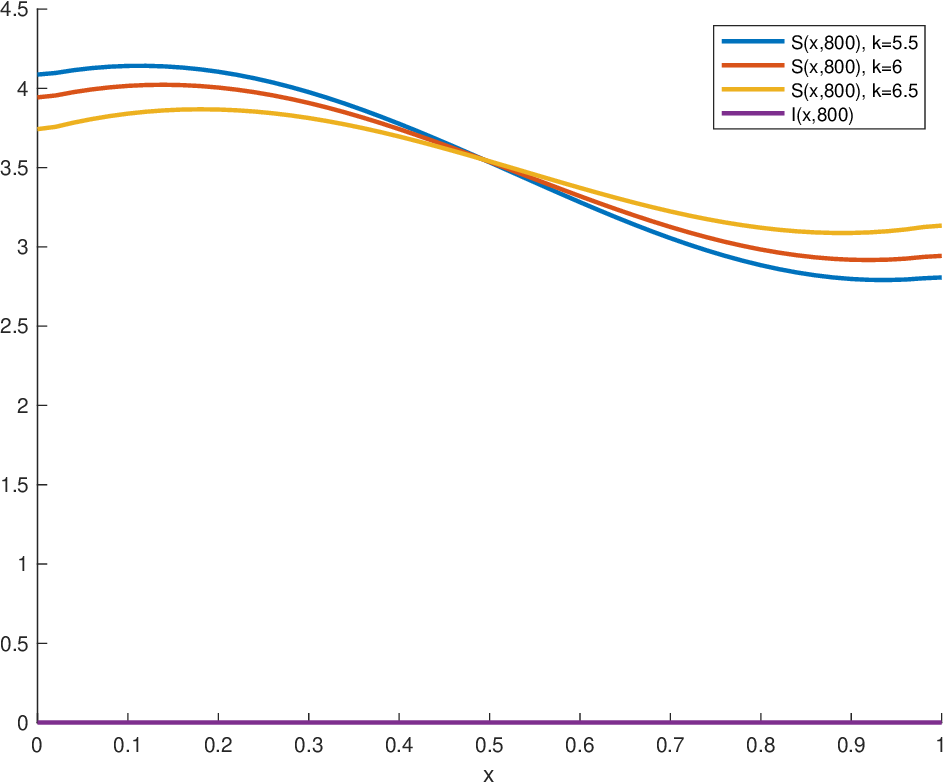}
         \caption{$\beta=k-x$, $m=1$}
     \end{subfigure}
     \caption{Simulation of model \eqref{ds=0-model} with $H^-\neq \emptyset$.}
     \label{fig:30}
 %    \vskip -20pt
\end{figure}

\subsection{Simulations for system \eqref{dI=0-model}}
In this subsection, we perform numerical simulations for the system \eqref{dI=0-model} to verify and complement the main theoretical results presented in Subsection 2.2.

\textbf{Simulation 4.} In this simulation, we explore the cases where \(\beta > \gamma\) and \(\beta < \gamma\) on \(\Omega\). We set \(\beta = 1 + x\) and \(d_S=m = 1\).

First, we choose \(\gamma = 0.8\) such that \(\beta > \gamma\). According to Theorem \ref{th2.5}, and noting that \(m\) is positive, the behavior of the system depends on the size of \(N\). If \(N\) is small, \((S, I)\) converges to \((N/|\Omega|, 0)\). If \(N\) is large, \((S, I)\) converges to \((\hat{S}, \hat{I})\) with \(\hat{S} > 0\) and \(\hat{I} \not\equiv 0\) as \(t \to \infty\). To verify this, we solve the system \eqref{dI=0-model} with initial data \((aS_0, aI_0)\). When \(a = 0.1\), Fig.\ref{fig:di0} (A) shows that \((S, I)\) converges to \((N/|\Omega|, 0)\). When \(a = 0.5\), Fig.\ref{fig:di0}(B) shows that \(I\) persists in some region of \(\Omega\).

Next, we choose \(\gamma = 2.5\) such that \(\beta < \gamma\). Theorem \ref{th2.6} implies that \((S, I)\) always converges to \((N/|\Omega|, 0)\). To confirm this, we solve the system \eqref{dI=0-model} with initial data \((aS_0, aI_0)\). Fig.\ref{fig:di0}(C) shows that \(S\) always converges to a positive constant and \(I\) converges to zero.

\begin{figure}
     \centering
     \begin{subfigure}[b]{0.32\textwidth}
         \centering
          \includegraphics[scale=.3]{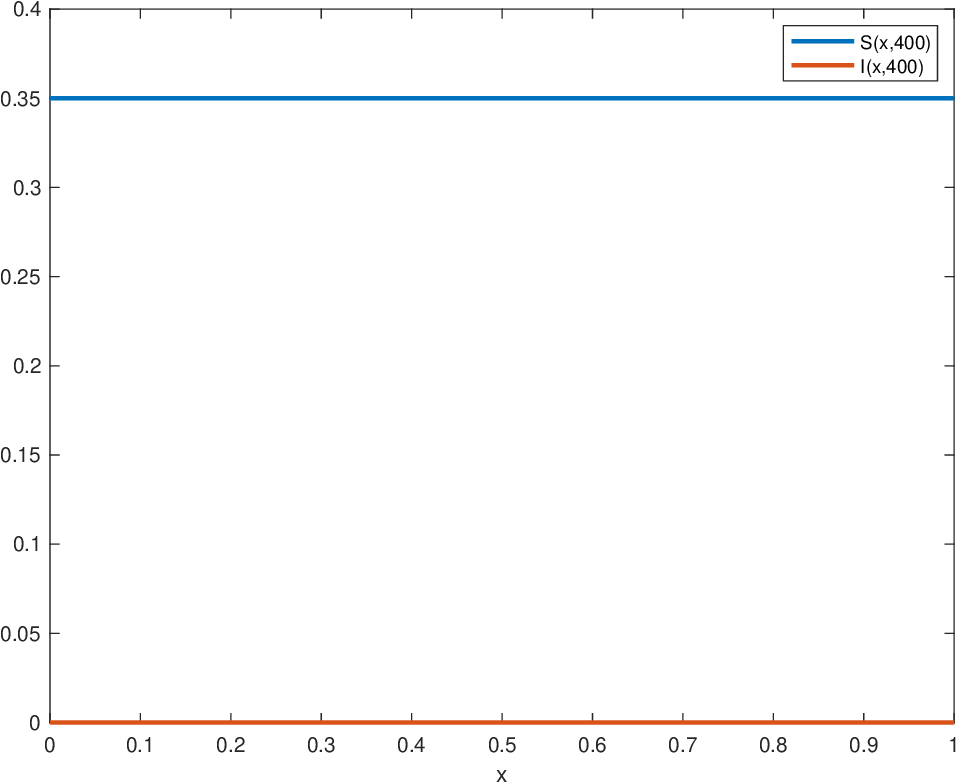}
        \caption{$\beta>\gamma$, $N$ small}
     \end{subfigure}
     \begin{subfigure}[b]{0.32\textwidth}
        \centering
     \includegraphics[scale=.3]{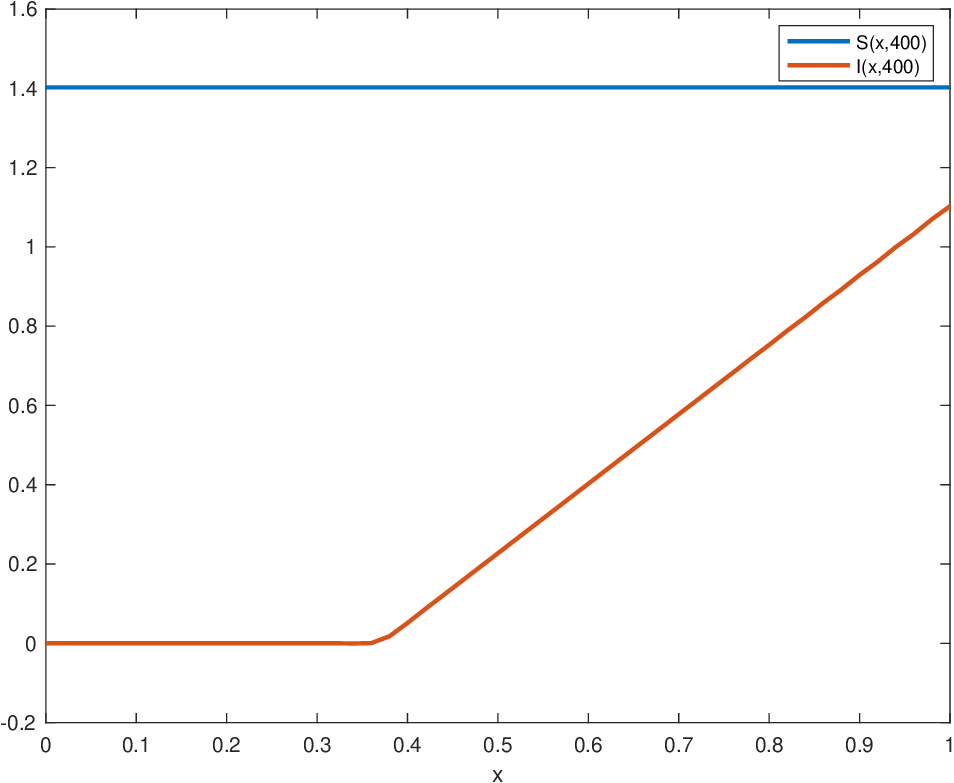}
             \caption{$\beta>\gamma$, $N$ large}
     \end{subfigure}
               \begin{subfigure}[b]{0.32\textwidth}
         \centering
          \includegraphics[scale=.3]{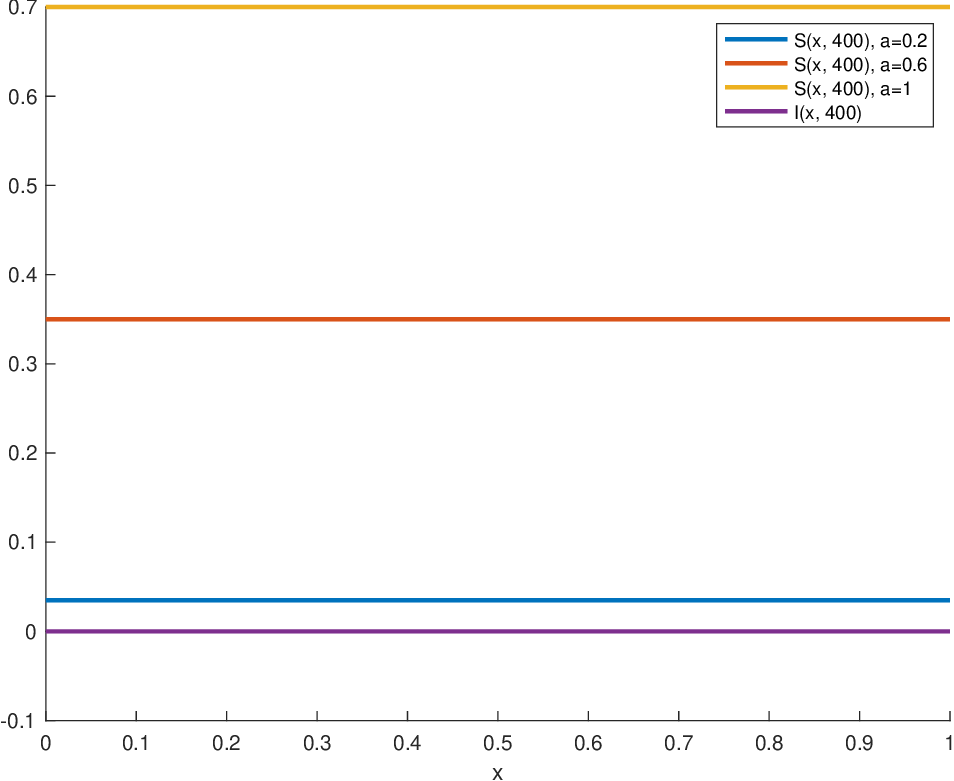}
         \caption{$\beta<\gamma$, initial data $(aS_0, aI_0)$}
     \end{subfigure}
     \caption{Simulation of model \eqref{dI=0-model}. Left two figures: $\beta=1+x$, $\gamma=0.8$, $m=1$. Right figure: $\beta=1+x$, $\gamma=2.5$, $m=1$. }
     \label{fig:di0}
 %    \vskip -20pt
\end{figure}

\textbf{Simulation 5.} In this simulation, we explore the scenario where \(\beta - \gamma\) changes sign on \(\Omega\). We set \(d_S = \gamma = 1\) and \(m = m_0\).

First, consider \(\beta = \sqrt{x} + 0.5\). In this case, \(H^+ = (0.25, 1]\) and \(H^- = [0, 0.25)\). Note that \(M^0 = [0.5, 1]\). For any \(S^* > 0\), we have:
\[
[0.5, 1] \subseteq \{x \in \bar{\Omega} : ((R(x) - 1)S^* - m(x))_+ > 0\}.
\]

We solve the system \eqref{dI=0-model} with initial data \((aS_0, aI_0)\) for different values of \(a\). Fig.\ref{fig:4}(A) confirms that the infection persists on \([0.5, 1]\) and becomes extinct at least in the low-risk region \(H^- = [0, 0.25]\) for any value of \(a\). This is consistent with Lemma \ref{lem4.2}, which states that \(\liminf_{t \to \infty} I(\cdot, t) \ge ((R - 1)Nl_1 - m)_+ \chi_{\{I_0 > 0\}}\). Note that Proposition \ref{prop3} is inapplicable here because \(\int_{0.5}^1 \frac{1}{\beta - \gamma} = \infty\) for the chosen parameters.

Next, consider \(\beta = 1.5 - \sqrt{x}\). In this case, \(H^+ = [0, 0.25)\) and \(H^- = (0.25, 1]\). Here, \(H^+ \subset M^0\). It is straightforward to verify that \((R - 1)N/|\Omega| - m \leq 0\) on \(H^+\) if and only if \(N \leq 2\).
We solve the system \eqref{dI=0-model} with initial data \((aS_0, aI_0)\) for different values of \(a\) (\(N = 2\) if \(a \approx 0.57\)). If \(a = 1\), there exists \(x \in H^+ \cap \{I_0 > 0\}\) such that \((R(x) - 1)N/|\Omega| - m(x) > 0\). In this case, Fig.\ref{fig:4}(B) shows that the infection persists in a subset of \(H^+\), which is consistent with Corollary \ref{coro2.3}, even though \(\int_{0}^{0.25} \frac{1}{\beta - \gamma} = \infty\). If \(a = 0.5\), then \((R - 1)N/|\Omega| - m \leq 0\) on \(H^+\), and \(I\) converges to zero, as shown in Fig.\ref{fig:4}(B), consistent with Corollary \ref{coro2.3}.
Fig.\ref{fig:4}(C) shows \(\int_0^1 I(x, 600) \, dx\) as a function of \(a\). It reveals a critical value of \(a\) between 0.5 and 1, below which the infected population converges to zero and above which it persists. Moreover, the size of the infected population increases as the total population size increases.

\begin{figure}
     \centering
     \begin{subfigure}[b]{0.32\textwidth}
         \centering
          \includegraphics[scale=.3]{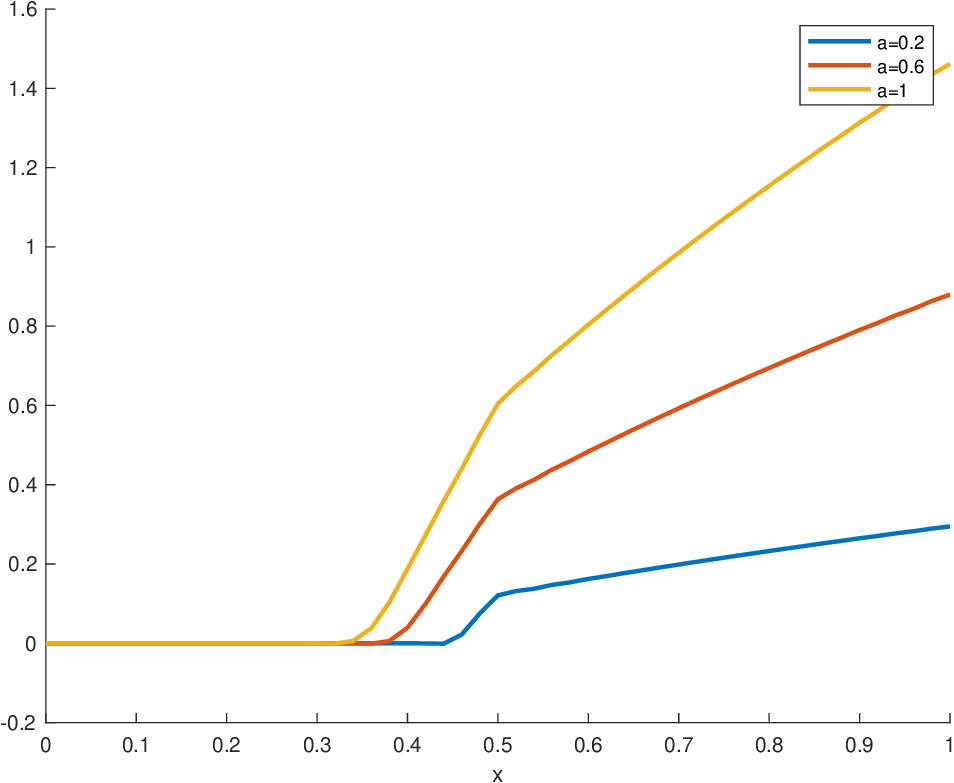}
        \caption{$\beta=\sqrt{x}+0.5$}
     \end{subfigure}
     \begin{subfigure}[b]{0.32\textwidth}
         \centering
          \includegraphics[scale=.3]{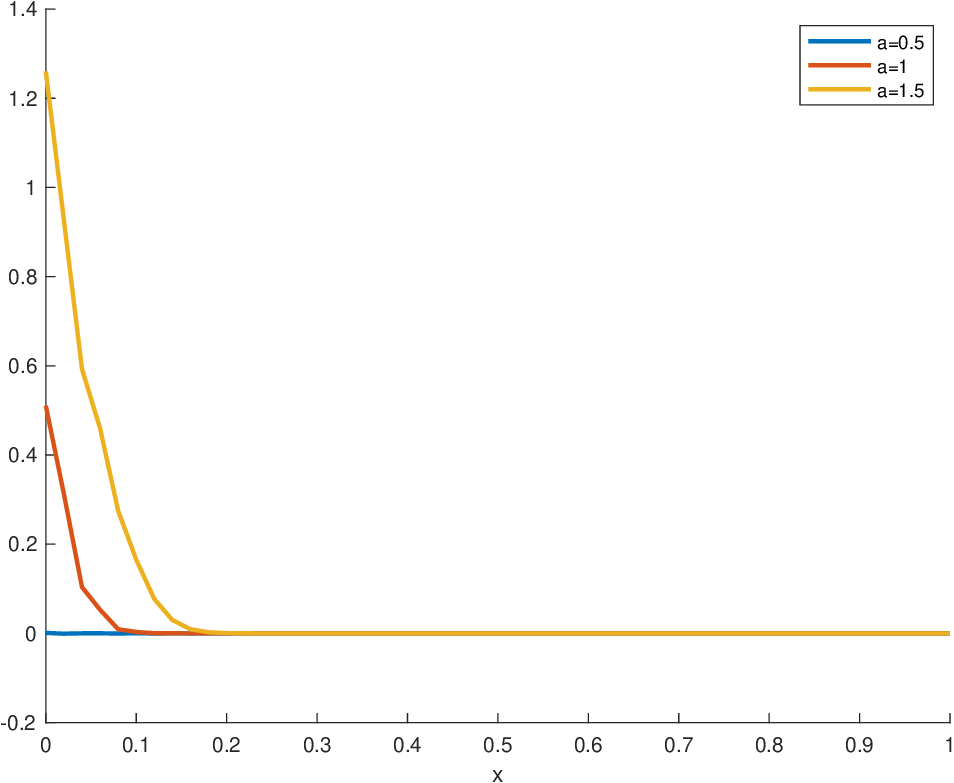}
         \caption{$\beta=1.5-\sqrt{x}$}
     \end{subfigure}
\begin{subfigure}[b]{0.32\textwidth}
         \centering
          \includegraphics[scale=.3]{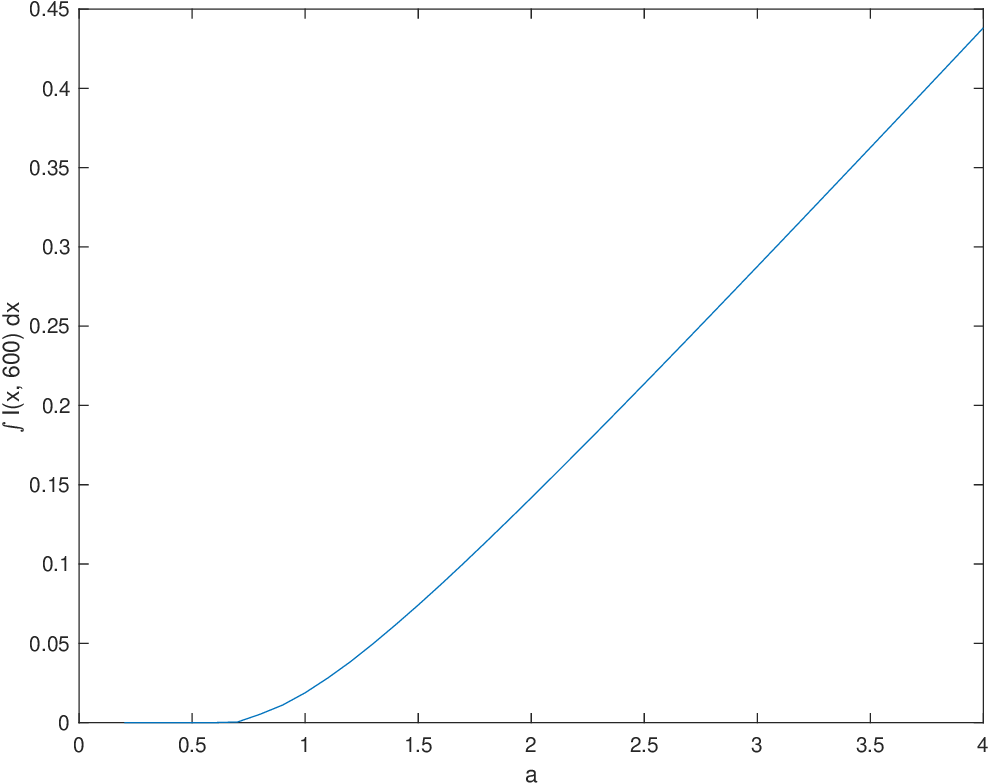}
         \caption{$\beta=1.5-\sqrt{x}$}
     \end{subfigure}
  \caption{Simulation of model \eqref{dI=0-model} with $\gamma=1$, $m=m_0$, and initial data $(aS_0, aI_0)$.  Fig.(A)-(B) show $I(x, 600)$; Fig.(C) shows $\int_0^1 I(x, 600)dx$ as a function of  $a$.}
     \label{fig:4}
\end{figure}

\section{Discussion and Conclusion}
The SIS model serves as one of the most fundamental mathematical frameworks for understanding infectious disease dynamics. 
This paper presents a comprehensive analysis of the SIS systems \eqref{ds=0-model} and \eqref{dI=0-model}, offering insights into the dynamics of infectious disease spread and control. %Our results may provide helpful implications for public health strategies, particularly in the context of the COVID-19 pandemic and similar infectious diseases.

\subsection{For system \eqref{ds=0-model}: control on the susceptible's movement}
System  \eqref{ds=0-model} reflects the strategic movement constrains imposed to the susceptible population. If  the habitat is entirely composed of high-risk areas ($\beta>\gamma$ on $\Omega$) and the saturated incidence exists (i.e. $m\not\equiv 0$), Theorem \ref{th2.1} indicates that the fate of a disease—whether it diminishes or persists—depends significantly on the total population size \(N\). Specifically, the disease will persist across the habitat in the long term if \(N\) exceeds the threshold value \(\int_{\Omega}\frac{m\gamma}{\beta-\gamma}\). Conversely, if \(N\) falls below this threshold (considering the saturated incidence factor), the disease is destined for eradication. In contrast, if no saturated incidence effect exists (i.e., \(m\equiv0\)), \cite[Theorem 4.2(ii)]{salako2024degenerate} demonstrates that the disease will persist throughout the habitat in the long run, regardless of the total population size.

Theorem \ref{th2.2} suggests that in habitats consisting solely of high-risk and moderate-risk locations, disease extinction can be achieved in the long run if the movement of susceptible individuals is sufficiently restricted. Similarly, Theorem \ref{th2.3} demonstrates that the presence of low-risk areas, coupled with movement restrictions, can lead to the eventual extinction of the disease. Theorems \ref{th2.2}-\ref{th2.3} imply that controlling the mobility of susceptible individuals is a viable strategy for disease eradication as long as moderate-risk or low-risk locations are created.

%Furthermore, the persistence of the susceptible population in areas exhibiting a positive saturated incidence effect such as natural resistance to the infection is notable. 

Furthermore, the presence of saturated incidence effects significantly impacts the distribution of the susceptible population. The susceptible population can persist in areas where the risk is moderate or low, as well as in regions where saturated incidence effects are present. In particular, when this effect is consistently present throughout the habitat, the susceptible population will persist on the entire habitat. In contrast, Proposition \ref{prop1} and Corollary \ref{cor1} suggest that the susceptible population cannot occupy all the high-risk places where the saturated incidence effect is negligible, particularly with low infected population mobility.

%The susceptible population will mainly remain in moderate- and low-risk locations, as well as in areas with a positive saturated incidence effect. 

\subsection{For system \eqref{dI=0-model}: control on the infected's movement}
System \eqref{dI=0-model} reflects the strategic movement restrictions imposed to the infected population.
According to Theorem \ref{th2.4},  restricting the movement of the infected population allows the susceptible population to occupy the entire habitat. This outcome holds regardless of transmission and recovery rates or the presence of a saturated incidence effect.% These findings emphasize the critical role of controlling the mobility of infected individuals in maintaining a healthy population and preventing the spread of disease throughout the habitat. 
\ Theorem \ref{th2.5} demonstrates that in high-risk environments, restricting the movement of infected individuals can result in a homogeneous distribution of the susceptible population. The persistence or extinction of the infectious disease, however, depends on the total population size and the presence of the saturated incidence effect. In the absence of this effect, the disease persists. Conversely, when the saturated incidence effect is considered, a larger total population size tends to support disease persistence, while a smaller population size increases the likelihood of eventual extinction.

Theorem \ref{th2.6} demonstrates that in low- and moderate-risk regions, infections will eventually be eradicated. However, in high-risk regions with a large total population, eradication of the infection is not possible. Proposition \ref{prop3} and Corollary \ref{coro2.3} further reveal that as the population increases, the area where the disease can persist also expands. This highlights the importance of population control in managing disease spread, particularly in high-risk regions. These results  surprisingly differ from \cite[Theorem 4.6(i)]{salako2024degenerate} for the model without saturated incidence effect, which indicates that the disease will persist exactly on high-risk areas when the mobility of the infected population is sufficiently controlled.

%Additionally, Theorem \ref{th2.6}(ii) and Proposition \ref{prop3} provide insights into the dynamics within high-risk areas. When the saturated incidence effect is negligible, the disease will eventually concentrate within these high-risk zones, independent of the total population size. Conversely, if this effect is present, the total population size becomes a determining factor. As 
%Corollary \ref{coro2.3} offers a nuanced view of disease dynamics, emphasizing the role of initial outbreak conditions. Assume that the total population number is conserved over time. If a high-risk location experiences an outbreak with a sufficiently high transmission risk, the disease will persist across the habitat; conversely, if the transmission risk is not high enough at any location, the disease will eventually die out, even if it initially breaks out in high-risk areas.

%The aforementioned theoretical results align well with the numerical simulations conducted in Subsection 5.2. They also present a significant contrast to the scenario where no saturated incidence effect is considered. As demonstrated by \cite[Theorem 4.6(i)]{salako2024degenerate}, the disease can persist only in high-risk areas once the mobility of the infected population is sufficiently controlled.

\subsection{Control on both susceptible and infected's movements}
If we restrict the movement of both susceptible and infected populations, that is, as \(d_S, d_I\) approach zero in \eqref{SIS2}, formally we are led to the following system of ODEs:
\begin{equation}\label{SISode}
\begin{cases}
\displaystyle
\dfrac{d S}{d t}=-\dfrac{\beta(x) SI}{m(x)+S+I}+\gamma(x) I,\ \ &t > 0,\,x\in\bar\Omega, \vspace{1mm} \\
\displaystyle
\dfrac{d I}{d t}= \dfrac{\beta(x) SI}{m(x)+S+I} -\gamma(x) I, & t > 0, \,x\in\bar\Omega,\vspace{1mm} \\
\displaystyle
S(0,x)+I(0,x)=n(x)\ge 0,&\,x\in\bar\Omega
\end{cases}
\end{equation}
with $\int_\Omega n=N>0$. 

We define the local basic reproduction number \(\mathcal{R}_0(x)\)  of system \eqref{SISode} by
\[
\mathcal{R}_0(x)=\frac{\beta(x) n(x)}{\gamma(x) (m(x)+n(x))},\ \ \ x\in\bar\Omega.
\]
Note that if a point $x\in\bar\Omega$ is of low risk (i.e., $\beta(x)<\gamma(x)$), then $\mathcal{R}_0(x)<1$, regardless of the distribution of the population $n(x)$. Moreover, if $m(x)=0$ for some $x\in\bar\Omega$, then the point $x$ is of low-risk if and only if  $\mathcal{R}_0(x)<1$. According to \cite[Theorem A.1]{GLPZ-2024}, the following statements hold:
\begin{enumerate}
   \item[\rm{(i)}] If \(\mathcal{R}_0(x)\le 1\), then \((S(t,x),I(t,x))\to (n(x),0)\) as \(t\to\infty\);
   \item[\rm{(ii)}] If \(\mathcal{R}_0(x)> 1\), then \((S(t,x),I(t,x))\to \left(\frac{1}{\mathcal{R}_0(x)}n(x),\left(1-\frac{1}{\mathcal{R}_0(x)}\right) n(x)\right)\) as \(t\to\infty\).
\end{enumerate}

These results reveal that the long-term dynamics of system \eqref{SISode} are determined by the local basic reproduction number \(\mathcal{R}_0\). Specifically, at position $x$, if \(\mathcal{R}_0(x) > 1\) (equivalently, \((\beta(x) - \gamma(x))n(x) > m(x)\gamma(x)\)), the disease will persist; if \(\mathcal{R}_0(x) \leq 1\) (equivalently, \((\beta(x) - \gamma(x))n(x) \leq m(x)\gamma(x)\)), the disease will eventually become extinct. 
Thus, in each moderate- or low-risk location, controlling the movement of both susceptible and infected populations can lead to disease eradication. However, in every high-risk location, the success of such control measures in achieving disease extinction depends on factors such as the total population size and the presence of a saturated incidence effect. Specifically, a large total population size and the absence of a saturated incidence effect tend to support disease persistence, whereas a small total population size and the presence of a saturated incidence effect can promote disease extinction.
These findings also suggest that merely creating moderate- or low-risk sub-areas is insufficient for eradicating the disease across the entire habitat, when the movement of both populations is strictly limited. This outcome is similar to the results observed for system \eqref{dI=0-model}, when the movement of only infected population is controlled. In contrast, the results on system \eqref{ds=0-model} suggest that limiting the movement of susceptible population may eliminate the disease if moderate- or low-risk sub-areas exist. Henceforth, restricting the movement of both populations may not be the most effective strategy for eradicating the disease.

Indeed,  targeting the movement of either the susceptible or infected population can sometimes yield better results than addressing the movement of both simultaneously. This can be explained by several possible reasons. (1) Targeted intervention: Focusing on the susceptible population can significantly reduce new infections, as shown in the systematic review by Balcan et al. \cite{Balcan}, where restricting movement into high-infection areas lowers transmission rates. (2) Resource allocation: Managing both populations may dilute intervention effectiveness; concentrating resources on one group allows for more efficient strategies \cite{MBSC}. (3) Behavioral dynamics: The behavior of individuals affects disease spread; if infected individuals are less mobile, their impact on transmission diminishes \cite{Kraemer2020}.  

\subsection{Summary}
Existing research has emphasized the importance of tailored approaches involving multiple intervention strategies to contain the spread of infectious diseases \cite{Balcan,Chinazzi2020,Flaxman2020,Ferretti2020,Lipsitch2003,MBSC}.   This study focuses on the  role of movement restrictions on either susceptible or infected population in controlling infectious diseases. By analyzing systems \eqref{ds=0-model} and \eqref{dI=0-model} compared with \eqref{SISode}, it offers some insights into developing effective public health strategies,  during pandemics like COVID-19. The findings underscore the necessity of adaptive strategies that consider regional variations of transmission risk and targeted interventions to achieve disease eradication.

\vspace{.2in}
%{\large\noindent{\bf Declarations}}

%\begin{center}
%CONFLICT OF INTEREST
%\end{center}
%The authors declare that they have no conflict of interest.
%\vspace{.1in}
% \begin{center}
% DATA AVAILABILITY STATEMENT
% \end{center}
% Data sharing is not applicable to this article as no new data were created or analysed in this study. 

\end{document}